\documentclass[11pt,leqno]{amsart}

\usepackage[a4paper, right=20mm, left=20mm, top=20mm]{geometry}

\usepackage{amsthm}     %
\usepackage[english]{babel} %
\usepackage[utf8]{inputenc} %
\usepackage[hidelinks]{hyperref}
\usepackage{amsmath}
\usepackage{amssymb}
\usepackage{mathtools}
\usepackage{mathrsfs}
\usepackage{enumerate}

\usepackage{zref-clever}
\newcommand\Cref[1]{\zcref[S]{#1}}

\usepackage{tikz}
\usetikzlibrary{matrix,arrows,cd}
\tikzcdset{diagrams={arrows={shorten >=-0.5ex,shorten <=-0.5ex},nodes={inner sep=0.4em},row sep=1em,column sep=1em}}

\usepackage{tikz}

\tikzset{
    partition/.style={
      scale=0.4,
      yscale=-1,
      baseline={([yshift=-0.5ex]current bounding box.center)}
    }
}

\tikzset{
    bend/.cd,
    0/.style={},
    1/.style={bend right},
    -1/.style={bend left}
}
\newcommand\makePartPt[1]{({Mod(#1,10)},{(#1-Mod(#1,10))*.1})}
\newcommand\makePartLn[2]{%
\pgfmathtruncatemacro\bend{%
(int(#1/10)==int(#2/10)) ?
(#1<10 ? 1 : -1)*(#1>#2 ? 1 : -1)
: 0%
}
\draw[draw=black,line width=0.5pt,line cap=round] ({mod(#1,10)},{int(#1/10}) to[bend/\bend] ({mod(#2,10)},{int(#2/10)});
}
\newcommand\tp[2][] {%
~\tikz[partition] {
\draw[white,opacity=0] (1,0)--(1,1); % fix alignment for partition with 1 pt
\def\j{0}
\foreach \i [remember=\i as \j] in {#2} {
  \ifnum \i>0
    \ifnum \j>0
      \makePartLn{\i}{\j};
    \fi
  \fi
} %foreach
\foreach \i [remember=\i as \j] in {#2} {
  \ifnum \i>0
    \draw[draw=none,fill=black] \makePartPt\i circle (4pt);
  \fi
} %foreach
{#1}
}~}

\newcommand\listorempty[1]{\def\temp{#1}\ifx\temp\empty \emptyset \else (#1) \fi}

\numberwithin{equation}{section}
\newtheorem{Theorem}{Theorem}[section]
\zcsetup{countertype={Theorem=theorem}}

\newtheorem{thmalph}{Theorem}

\zcsetup{countertype={thmalph=theorem}}

\newcommand\newtheoremx[3]{%
\AddToHook{env/#1/begin}{%
\zcsetup{countertype={Theorem=#2}}}
\newtheorem{#1}[Theorem]{#3}
}

\theoremstyle{plain}
\newtheoremx{Lemma}{lemma}{Lemma}
\newtheoremx{Proposition}{proposition}{Proposition}
\newtheoremx{Corollary}{corollary}{Corollary}

\theoremstyle{definition}
\newtheoremx{Definition}{definition}{Definition}
\newtheoremx{Remark}{remark}{Remark}
\newtheoremx{Example}{example}{Example}

\newcommand{\EatDot}[1]{}

\newcommand{\C}{\mathbb{C}}

\newcommand{\Q}{\mathbb{Q}}
\newcommand{\Z}{\mathbb{Z}}
\newcommand{\N}{\mathbb{N}}
\newcommand{\F}{\mathbb{F}}
\newcommand{\kk}{\Bbbk}

\newcommand{\id}{\mathrm{id}}

\newcommand{\Aut}{\mathrm{Aut}}
\newcommand{\Hom}{\mathrm{Hom}}
\newcommand{\Ext}{\mathrm{Ext}}

\newcommand{\Rep}{\mathrm{Rep}}
\newcommand{\Tilt}{\mathrm{Tilt}}
\newcommand{\Proj}{\mathrm{Proj}}
\newcommand{\GL}{\mathrm{GL}}

\newcommand{\PSh}{\mathrm{PSh}}
\newcommand{\Ind}{\mathrm{Ind}}

\newcommand\Kar{\operatorname{Kar}}
\newcommand\uPerm{\underline{\operatorname{Perm}}}
\newcommand\uRep{\underline{\mathrm{Re}}\mathrm{p}}
\newcommand\RepSt{\uRep(S_t)}

\DeclareMathOperator{\chark}{char}

\newcommand{\aff}[1]{\smash{\mathfrak{\widehat{#1}}}}
\newcommand{\kacmoody}[1]{\smash{\mathfrak{\widetilde{#1}}}}

\newcommand{\Tens}{\mathrm{Tens}}
\newcommand{\Vect}{\mathrm{Vect}}

\newcommand\op{\mathrm{op}}

\newcommand\cA{\mathcal A}
\newcommand\cC{\mathcal C}
\newcommand\cD{\mathcal D}
\newcommand\cM{\mathcal M}
\newcommand\cR{\mathcal R}
\newcommand\cS{\mathcal S}
\newcommand\cT{\mathcal T}
\newcommand\cU{\mathcal U}

\newcommand\im{\operatorname{im}}
\newcommand\coim{\operatorname{coim}}
\newcommand\Rel{\mathrm{Rel}}
\newcommand\End{\operatorname{End}}

\newcommand\one{\mathbf1}
\renewcommand\:{\colon}
\renewcommand\o{\otimes}
\newcommand\ti\widetilde

\title[Monoidal Ringel duality]{Monoidal Ringel duality\\and monoidal highest weight envelopes}
\author{Johannes Flake}
\author{Jonathan Gruber}
\date{\today}

\begin{document}

\begin{abstract}
	We show that a large class of non-abelian monoidal categories can be realized as subcategories of tilting objects in abelian monoidal categories with a highest weight structure.
	The construction relies on a monoidal enhancement of Brundan--Stroppel's semi-infinite Ringel duality and applies to many of Sam--Snowden's triangular categories and Knop's tensor envelopes of regular categories.
	We also explain how monoidal Ringel duality gives rise to monoidal structures on categories of representations of affine Lie algebras at positive levels.
\end{abstract}

\maketitle

\section*{Introduction}

Many categories of interest in representation theory come equipped with certain natural additional structures.
For instance, categories of a diagrammatic nature (such as partition categories or Brauer categories) are usually equipped with a monoidal structure (i.e.\ with a well-behaved tensor product).
Categories of representations of Lie theoretic objects (such as algebraic groups, quantum groups or affine Lie algebras) often admit a highest weight structure, as introduced by Cline--Parshall--Scott \cite{CPShighestweight}.
The goal of this article is to explain some natural interactions between monoidal structures and highest weight structures, namely how monoidal structures on highest weight categories behave under Brundan--Stroppel's semi-infinite Ringel duality \cite{BrundanStroppelSemiInfiniteHighestWeight}.
In short (and as explained in more detail in Theorems~\ref{thm:introlowertoupper} and \ref{thm:introuppertolower} below), we show that a monoidal structure on a highest weight category (subject to some compatibility conditions) gives rise to a canonical monoidal structure on the Ringel dual highest weight category.
Our two main applications of this ``monoidal Ringel duality'' are as follows.

\subsection*{Interpolation categories}

Since the seminal work of Deligne \cite{Deligne-RepSt}, there has been a lot of interest in studying \emph{interpolation categories}, that is, families of monoidal categories $\cA_t$ with a parameter $t$ that interpolate classical families of categories of representations, such as the complex representations of symmetric groups $S_n$ or general linear groups $\GL_n(\C)$ for different $n$.
More precisely, the categories $\cA_t$ are typically additive and semisimple at generic parameters, but at certain special parameters $t=n$, they are non-semisimple and admit monoidal quotient functors to classical categories of representations.

An important question that arises in this context, also going back to the work of Deligne \cite[Proposition~8.19]{Deligne-RepSt}, is whether at a special parameter $t=n$, the non-semisimple category $\cA_t$ can be embedded as a monoidal additive subcategory in some monoidal abelian category $\cC_t$.
If there is a universal monoidal abelian category $\cC_t$ with this property, then $\cC_t$ is called a \emph{monoidal abelian envelope} of $\cA_t$ \cite{CEH,CoulembierMonAbEnv,CEOPAbEnvQuotProp}.
For the interpolation categories $\cA_t = \uRep(S_t)$ corresponding to representations of symmetric groups, Deligne has constructed a suitable family of monoidal abelian categories $\cC_t = \uRep^\mathrm{ab}(S_t)$, which were further studied in some detail and shown to be monoidal abelian envelopes by Comes--Ostrik \cite{ComesOstrikBlocksRepSt,ComesOstrikRepabSd}.
Somewhat surprisingly, one finds that the category $\uRep^\mathrm{ab}(S_t)$ is a lower finite highest weight category in the sense of \cite{BrundanStroppelSemiInfiniteHighestWeight} (because all of its blocks are either semisimple or equivalent to the principal block of quantum $\mathfrak{sl}_2(\C)$ at a root of unity), and the additive subcategory $\uRep(S_t)$ identifies with the full subcategory of tilting objects in $\uRep^\mathrm{ab}(S_t)$.
Similar observations have subsequently been made for other families of interpolation categories (see \cite[Section 6.6]{BrundanStroppelSemiInfiniteHighestWeight}), but so far, there has not been a conceptual explanation for why interpolation categories should embed as categories of tilting objects in monoidal lower finite highest weight categories.
To provide such an explanation was our main motivation in writing this article, and we do so in \Cref{thm:introinterpolation} below.

While there is no general definition of interpolation categories to date, a large class of interpolation categories (including interpolation categories for representations of symmetric groups $S_n$ and finite general linear groups $\GL_n(\F_q)$) has been constructed uniformly by Knop \cite{KnopTensorEnvelopes}.
The construction starts from a regular category $\mathcal{R}$ and uses a calculus of relations and a degree function $\delta$ to define a rigid symmetric monoidal category $\cT(\mathcal{R},\delta)$, called the tensor envelope of $\mathcal{R}$ (see \Cref{sec:tensorenvelopes} below for more details).
We may regard the set of degree functions as a space of interpolation parameters $\delta$ for the family of monoidal categories $\cT(\mathcal{R},\delta)$.

\begin{thmalph}
\label{thm:introinterpolation}
    Let $\mathcal{R}$ be a subobject-finite exact regular Mal'cev category with a degree function $\delta$, taking values in an algebraically closed field of characteristic zero.
    Then there is a lower finite highest weight category $\cC$ with a rigid symmetric monoidal structure such that the full subcategory of tilting objects $\Tilt(\cC)$ is a monoidal subcategory of $\cC$ closed under taking duals and
    \[ \cT(\cR,\delta) \simeq \Tilt(\cC) \]
    as symmetric monoidal categories.
    If $\cT(\mathcal{R},\delta)$ admits a monoidal abelian envelope, then $\cC$ is a monoidal abelian envelope of $\cT(\cR,\delta)$. %
\end{thmalph}

In particular, the category $\cT(\cR,\delta)$ always admits a monoidal embedding into a monoidal abelian category, which was an open question (to the best of our knowledge) in the cases where $\cT(\cR,\delta)$ is non-semisimple.
\Cref{thm:introinterpolation} implies that whenever $\cT(\cR,\delta)$ admits a monoidal abelian envelope, the latter is a lower finite highest weight category such that $\cT(\cR,\delta)$ identifies with the full subcategory of tilting objects.
Notably, this explains the existence of highest weight structures on the monoidal abelian envelopes of the interpolation categories $\RepSt$ and $\uRep(\GL_t(\F_q))$. 
Using a comparison result due to Snowden \cite{Snow-regular}, it also follows that, under certain additional assumptions,
the highest weight category $\cC$ from \Cref{thm:introinterpolation} is simultaneously a monoidal abelian envelope of $\cT(\cR,\delta)$ and of a category of ``permutation modules'' $\uPerm(G,\mu)$ constructed from a pro-oligomorphic group $G$ (associated with $\cR$) and a measure $\mu$ (associated with $\delta)$ as in \cite{HS-oligo}, see \Cref{prop:comparison-Knop-HS}.
This abelian envelope was constructed using different methods in \cite{HS-oligo}.

A brief (and not entirely accurate) outline of the proof of \Cref{thm:introinterpolation} is as follows:
We first show that Knop's tensor envelopes fit into Sam--Snowden's framework of (monoidal) triangular categories \cite{SamSnowdenTriangular} and deduce that the category of presheaves $\mathcal{D} = \PSh\big( \cT(\cR,\delta) \big)$ (i.e.\ the category of contravariant functors from $\cT(\cR,\delta)$ to vector spaces) is an upper finite highest weight category in the sense of \cite{BrundanStroppelSemiInfiniteHighestWeight} (see \Cref{prop:triangularstructurehighestweightcategory} and Corollary \ref{prop:knop-monoidal-triangular}).
Furthermore, $\cT(\cR,\delta)$ is equivalent to the category of finitely generated projective objects in $\mathcal{D}$, and the Day convolution tensor product (see Subsection~\ref{subsec:functorcategories}) gives rise to a monoidal structure on $\mathcal{D}$.
Then, using monoidal Ringel duality (\Cref{thm:introuppertolower} below), we see that the Ringel dual $\cC = \prescript{\vee}{}{\mathcal{D}}$ is a lower finite highest weight category with a monoidal structure such that $\Tilt(\cC)$ is monoidally equivalent to $\cT(\cR,\delta)$, since Ringel duality interchanges finitely generated projective objects and tilting objects.
For more details, see \Cref{thm:knop-highest-weight-envelope} and \Cref{cor:knop-mon-ab-env}.

\subsection*{Affine Lie algebras at positive levels}

Let $\mathfrak{g}$ be a complex simple Lie algebra and let
\[ \aff{g} = \mathfrak{g} \otimes \C[t,t^{-1}] \, \oplus \, \C c \]
be the corresponding affine Lie algebra.
A $\aff{g}$-module $M$ is said to be of \emph{level} $\kappa \in \C$ if the central element $c$ of $\aff{g}$ acts on $M$ via multiplication with the scalar $\kappa - h^\vee$, where $h^\vee$ denotes the dual Coxeter number of $\mathfrak{g}$.
Let $\mathcal{O}_\kappa$ be the parabolic BGG category of $\aff{g}$-modules of level $\kappa$ on which the parabolic subalgebra $\mathfrak{g} \otimes \C[t] \oplus \C c$ acts locally finitely and on which $\mathfrak{g} \otimes t \C[t]$ acts locally nilpotently.
Further let $U_\zeta = U_\zeta(\mathfrak{g})$ be the quantum group (with divided powers) corresponding to $\mathfrak{g}$, with quantum parameter specialized to $\zeta \in \C$, and let $\Rep(U_\zeta)$ be the category of finite-dimensional $U_\zeta$-modules with a weight space decomposition.

For a negative or non-rational level $\kappa$, Kazhdan--Lusztig have defined a braided monoidal structure on the full subcategory $\mathcal{O}_\kappa^\mathrm{fl}$ of $\aff{g}$-modules of finite length in $\mathcal{O}_\kappa$ \cite{KL12,KL34}.
For $\zeta = \exp(\frac{\mathrm{i} \pi}{D\kappa})$ (where $D$ is the lacing number of $\mathfrak{g}$), they have further constructed a braided monoidal functor
\[ F_\kappa \colon \mathcal{O}_\kappa^\mathrm{fl} \longrightarrow \Rep(U_\zeta) , \]
which is an equivalence under mild restrictions on $\kappa$.
We say that $\kappa$ is \emph{KL-good} if $F_\kappa$ is an equivalence.

For a positive level $\kappa \in \Q_{>0}$, the questions whether there is a canonical monoidal structure on $\mathcal{O}_\kappa$ and how $\mathcal{O}_\kappa$ is related to representations of quantum groups have been studied in detail for $\mathfrak{g} = \mathfrak{sl}_2(\C)$ by McRae--Yang \cite{McRaeYang}, but little is known in the general case.
Using monoidal Ringel duality, we establish the following generalization of the results of McRae--Yang.

\begin{thmalph}
\label{thm:intropositivelevel}
    Let $\kappa \in \Q_{>0}$ and $\zeta = \exp(-\frac{\pi \mathrm{i}}{D \kappa})$.
    If $-\kappa$ is KL-good, then there is a braided monoidal structure on $\mathcal{O}_\kappa$ and an exact essentially surjective braided monoidal functor
    \[ G_\kappa \colon \mathcal{O}_\kappa \longrightarrow \Ind \, \Rep(U_\zeta) . \]
    Furthermore, the functor $G_\kappa$ admits a fully faithful right adjoint functor, which exhibits $\Ind \, \Rep(U_\zeta)$ as a reflective subcategory of $\mathcal{O}_\kappa$.
\end{thmalph}

The proof of \Cref{thm:intropositivelevel} relies on the fact (due to Arkhipov \cite{ArkhipovSemiInfiniteAssociative} and Soergel \cite{SoergelCharakterformeln}) that for a positive level $\kappa \in \Q_{>0}$, the category $\mathcal{O}_\kappa$ is an upper finite highest weight category whose Ringel dual is the lower finite highest weight category $\mathcal{O}_{-\kappa}$.
Therefore, monoidal Ringel duality (\Cref{thm:introlowertoupper} below) allows us to construct a monoidal structure on $\mathcal{O}_\kappa$ from the monoidal structure on $\mathcal{O}_{-\kappa}^\mathrm{fl} \simeq \Rep(U_\zeta)$.
For more details we refer the reader to \Cref{thm:positivelevelmonoidalstructure}.

\subsection*{Monoidal Ringel duality}

As explained before, the main tool that we use to establish Theorems~\ref{thm:introinterpolation} and \ref{thm:intropositivelevel} is \emph{monoidal Ringel duality}, that is, a monoidal enhancement of Brundan--Stroppel's semi-infinite Ringel duality \cite{BrundanStroppelSemiInfiniteHighestWeight}.
In order to discuss this formalism, we first need to introduce some additional terminology.
A highest weight category is (loosely speaking) an abelian category $\cC$ with a collection of standard objects $\Delta(\lambda)$, labeled by a poset $(\Lambda,\leq)$, such that every standard object has a unique simple quotient $L(\lambda)$ and the simple objects $L(\lambda)$ with $\lambda \in \Lambda$ form a set of representatives for the isomorphism classes of simple objects in $\cC$.
There is also a dual notion of costandard objects $\nabla(\lambda)$ for $\lambda \in \Lambda$, and an object of $\cC$ is called a tilting object if it admits two filtrations whose subquotients are all isomorphic to standard objects, or to costandard objects, respectively.

The characteristics of a highest weight category $\cC$ strongly depend on the structure of its weight poset $(\Lambda,\leq)$.
For instance, if $(\Lambda,\leq)$ is upper finite (that is, for any $\lambda \in \Lambda$, there are finitely many $\mu \in \Lambda$ with $\lambda \leq \mu$), then $\cC$ has enough projective objects, but may fail to have tilting objects.
In contrast, if $(\Lambda,\leq)$ is lower finite, then for all $\lambda \in \Lambda$, there is an indecomposable tilting object $T(\lambda)$ of highest weight $\lambda$, but $\cC$ may fail to have projective objects.
This situation is at least partly remedied by the fact (due to Ringel \cite{RingelAlmostSplit} for finite weight posets and generalized to infinite weight posets by Brundan--Stroppel \cite{BrundanStroppelSemiInfiniteHighestWeight}) that there is a bijective correspondence (called Ringel duality) between lower finite highest weight categories and upper finite highest weight categories, interchanging tilting objects and projective objects.
More precisely, for a lower finite highest weight category $\cC$ with weight poset $(\Lambda,\leq)$, there is an upper finite highest weight category $\cC^\vee$ with the opposite weight poset $(\Lambda,\leq^\mathrm{op})$ and a so-called Ringel duality functor
\[ R = R_\cC \colon \cC \longrightarrow \cC^\vee \]
that restricts to an equivalence between the full subcategory $\Tilt(\cC)$ of tilting objects in $\cC$ and the full subcategory $\Proj^\mathrm{fg}(\cC^\vee)$ of finitely generated projective objects in $\mathcal{D}$.
Furthermore, the Ringel duality functor extends to a functor $R \colon \cC^\mathrm{cc} \to \cC^{\vee,\mathrm{cc}}$ between suitable cocomplete versions of the categories $\cC$ and $\cC^\vee$, and the latter admits a left adjoint
\[ L = L_\cC \colon \cC^{\vee,\mathrm{cc}} \longrightarrow \cC^\mathrm{cc} , \]
which we call the opposite Ringel duality functor.
Finally, Ringel duality is a bijective correspondence, in that for every upper finite highest weight category, there is a (unique up to equivalence) lower finite highest weight category $\prescript{\vee}{}{\mathcal{D}}$ such that $\prescript{\vee}{}{(\cC^\vee)} \simeq \cC$ and $(\prescript{\vee}{}{\mathcal{D}})^\vee \simeq \mathcal{D}$.

Our monoidal enhancement of Ringel duality comes in two parts, the first of which consists of explaining how a monoidal structure on a lower finite highest weight category gives rise to a monoidal structure on the (cocompleted) Ringel dual upper finite highest weight category.

\begin{thmalph}
\label{thm:introlowertoupper}
    Let $\cC$ be a lower finite highest weight category with a monoidal structure $(\otimes,\mathbf{1},\ldots)$ such that $\otimes$ is right exact in both arguments and $\Tilt(\cC)$ is a monoidal subcategory.
    Then there is a unique (up to monoidal equivalence) monoidal structure $(\otimes^\vee,\mathbf{1}^\vee,\ldots)$ on the (cocompleted) Ringel dual $\cC^{\vee,\mathrm{cc}}$ such that $\otimes^\vee$ is cocontinuous in both arguments and the Ringel duality functor
    \[ R \colon \cC \longrightarrow \cC^\vee \]
    restricts to a monoidal equivalence $\Tilt(\cC) \simeq \Proj^\mathrm{fg}(\cC^\vee)$.
    Moreover, the opposite Ringel duality functor
    \[ L \colon \cC^{\vee,\mathrm{cc}} \longrightarrow \cC^\mathrm{cc} \]
    can be enhanced to a monoidal functor.
\end{thmalph}

\Cref{thm:introlowertoupper} is essentially proven by identifying $\cC^{\vee,\mathrm{cc}}$ with the category of presheaves on $\Tilt(\cC)$, which can be equipped with the Day convolution monoidal structure \cite{DayConvolution} using the monoidal structure on $\Tilt(\cC)$.
The fact that $L$ can be enhanced to a monoidal functor then follows from the universal property of Day convolution established by Im--Kelly \cite{ImKellyConvolution}.
For the precise statement and additional details, we refer the reader to \Cref{thm:monoidalRingelduality_lowertoupper}.

To prove a converse of \Cref{thm:introlowertoupper} (that is, to define a monoidal structure on a lower finite highest weight category given a monoidal structure on its Ringel dual) is a more difficult task.
For an upper finite highest weight category $\mathcal{D}$ and for $\cA = \Proj^\mathrm{fg}(\mathcal{D})$, our strategy here relies on the fact that the Ringel dual $\prescript{\vee}{}{\mathcal{D}}$ can be realized as the heart of a $t$-structure on the homotopy category $K^b(\cA)$ (see Subsection~\ref{subsec:Ringeldualityexceptionalsequences} below).
If $\mathcal{D}^\mathrm{cc}$ admits a monoidal structure such that $\cA$ is a monoidal subcategory, then there also is a canonical monoidal structure on $K^b( \cA )$.
Under some additional hypotheses, the aforementioned $t$-structure is compatible with this monoidal structure, so that we obtain a monoidal structure on the heart $\prescript{\vee}{}{\mathcal{D}}$.

\begin{thmalph}
\label{thm:introuppertolower}
    Let $\mathcal{D}$ be an upper finite highest weight category with a monoidal structure $(\otimes,\mathbf{1},\ldots)$ on $\mathcal{D}^\mathrm{cc}$ such that $\otimes$ is cocontinuous in both arguments and $\Proj^\mathrm{fg}(\mathcal{D})$ is a monoidal subcategory.
    Further suppose that $\mathcal{D}$ satisfies the condition \eqref{eq:Xtensor}.
    Then the Ringel dual $\cC = \prescript{\vee}{}{\mathcal{D}}$ admits a canonical monoidal structure $(\prescript{\vee}{}{\otimes},\prescript{\vee}{}{\mathbf{1}},\ldots)$ such that $\prescript{\vee}{}{\otimes}$ is right exact in both arguments, $\Tilt(\cC)$ is a monoidal subcategory and $\Tilt(\cC)$ is monoidally equivalent to $\Proj^\mathrm{fg}(\mathcal{D})$.
\end{thmalph}

For the precise statement and additional details, we refer the reader to \Cref{thm:monoidalRingelduality_uppertolower}.
Notably, under additional hypotheses on the monoidal structure $(\otimes,\mathbf{1},\ldots)$ on $\mathcal{D}^\mathrm{cc}$, we establish there variants of \Cref{thm:introuppertolower} that allow us to conclude that $\cC = \prescript{\vee}{}{\mathcal{D}}$ is braided, symmetric or rigid.

\subsection*{Structure of the article}

We conclude the introduction by giving a brief outline of the structure of the article.
In \Cref{sec:tensorcategories}, we discuss preliminaries on monoidal categories and on how monoidal structures interact with some elementary categorical constructions (such as Karoubi envelopes, homotopy categories and presheaf categories).
In \Cref{sec:highestweightcategories}, we recall Brundan--Stroppel's formalism of semi-infinite highest weight categories and semi-infinite Ringel duality, and we slightly reformulate semi-infinite Ringel duality in a way that is more well-suited to our monoidal Ringel duality.
The main results about monoidal Ringel duality are established in \Cref{sec:monoidalRingelduality}.
In \Cref{sec:triangularcategories}, we discuss Sam--Snowden's framework of monoidal triangular categories in relation to our monoidal Ringel duality, and \Cref{sec:tensorenvelopes} establishes that Knop's tensor envelops fit into said framework, thus establishing our main results about interpolation categories.
The final \Cref{sec:affineLiealgebras} is devoted to our results about tensor structures at positive levels for affine Lie algebras.
In \Cref{sec:relations}, we prove some results needed in \Cref{sec:tensorenvelopes} using a calculus of ($n$-ary) relations.

\subsection*{Acknowledgements}

The authors would like to thank Thorsten Heidersdorf, Robert McRae and Catharina Stroppel for helpful conversations, and Kevin Coulembier and Andrew Snowden for helpful conversations and valuable comments on earlier versions of this paper.
J.G.\ has received financial support from the SNSF via grant P500PT\_206751 and from the DFG via project 531430348.

\section{Monoidal categories}
\label{sec:tensorcategories}

In this section, we briefly recall the notions of monoidal categories, braidings and rigidity, and we discuss how these concepts interact with other categorical constructions (Karoubi envelopes, homotopy categories, presheaf categories, ind-completions) that will be used in the following sections.

\subsection{Monoidal categories}

Let us fix a field $\kk$ and write $\Vect_\kk$ for the category of $\kk$-vector spaces and $\Vect_\kk^\mathrm{fd}$ for the full subcategory of finite-dimensional $\kk$-vector spaces.
A \emph{$\kk$-linear category} is the same as a \emph{$\Vect_\kk$-enriched category} in the sense of \cite[Section 1.2]{KellyEnriched}, that is, a category where all $\Hom$-sets are $\kk$-vector spaces and where composition of homomorphisms is $\kk$-bilinear.
For $\kk$-linear categories $\cA$ and $\mathcal{B}$, a functor $F \colon \cA \to \mathcal{B}$ is called $\kk$-linear (or $\Vect_\kk$-enriched) if it is given by $\kk$-linear maps $\Hom_\cA(x,y) \to \Hom_\mathcal{B}\big( F(x) , F(y) \big)$ for all objects $x$ and $y$ of $\cA$.
We write $\mathrm{Fun}_\kk(\cA,\mathcal{B})$ for the $\kk$-linear category of $\kk$-linear functors from $\cA$ to $\mathcal{B}$, where homomorphisms are given by natural transformations.%
\footnote{There is no difference between the notion of a $\kk$-linear natural transformation and an ordinary natural transformation.}
We also write $\cA^\mathrm{op}$ for the ($\kk$-linear) opposite category of $\cA$, so $\cA^\mathrm{op}$ has the same objects as $\cA$ and $\Hom_{\cA^\mathrm{op}}(x,y) = \Hom_\cA(y,x)$ for all objects $x$ and $y$ of $\cA$.

A \emph{$\kk$-linear monoidal category} is a tuple $(\cA,\otimes,\mathbf{1},\alpha,\lambda,\rho)$ consisting of
\begin{itemize}
    \item a $\kk$-linear category $\cA$,
    \item a bifunctor $- \otimes - \colon \cA \times \cA \to \cA$ that is $\kk$-linear in both arguments (called the \emph{tensor product}),
    \item an object $\mathbf{1}$ of $\cA$ (called the \emph{unit object} or \emph{tensor unit}),
    \item a natural isomorphism $\alpha \colon - \otimes (- \otimes - ) \to (- \otimes -) \otimes -$ (called the associativity constraint),
    \item natural isomorphisms $\lambda \colon \mathbf{1} \otimes - \to \id_\cA$ and $\rho \colon - \otimes \mathbf{1} \to \id_\cA$ (called the unitors)
\end{itemize}
subject to the pentagon axiom and the triangle axiom from \cite[Definition 2.2.8]{EGNO}.
We also say that $(\otimes,\mathbf{1},\ldots)$ is a ($\kk$-linear) \emph{monoidal structure} on $\cA$ and occasionally refer to $\cA$ as a monoidal category without explicit reference to the additional data $(\otimes,\mathbf{1},\ldots)$.
The opposite category $\cA^\mathrm{op}$ of a monoidal category $(\cA,\otimes,\mathbf{1},\ldots)$ has a canonical monoidal structure with the same tensor product and unit object, and we define the reverse monoidal category $\cA^\mathrm{rev}$ as the monoidal category with the same underlying category $\cA$, but with the reverse tensor product $X \otimes^\mathrm{rev} Y = Y \otimes X$ for objects $X$ and $Y$ of $\cA$ (and similarly for homomorphisms).

A $\kk$-linear monoidal functor between $\kk$-linear monoidal categories $(\cA,\otimes,\mathbf{1},\ldots)$ and $(\cA',\otimes',\mathbf{1}',\ldots)$ is a pair $(F,\mu)$, where $F \colon \cA \to \cA'$ is a $\kk$-linear functor with $F(\mathbf{1}) \cong \mathbf{1}'$ and $\mu \colon F(-) \otimes' F(-) \to F(- \otimes -)$ is a natural isomorphism that is compatible with the associativity constraints of $\cA$ and $\cA'$ as in \cite[Definition 2.4.1]{EGNO}.
We often refer to $F$ as a monoidal functor without explicitly mentioning the natural isomorphism $\mu$, or we may say that a given functor $F \colon \cA \to \cA'$ can be enhanced to a monoidal functor if there is a natural isomorphism $\mu$ such that the pair $(F,\mu)$ is a monoidal functor.

A \emph{braiding} on a $\kk$-linear monoidal category $(\cA,\otimes,\mathbf{1},\ldots)$ is a natural isomorphism
\[ \beta \colon - \otimes - \to (- \otimes -) \circ \tau , \]
where $\tau \colon \cA \times \cA \to \cA \times \cA$ is given by $(X,Y) \mapsto (Y,X)$, subject to the hexagon axioms from \cite[Definition 8.1.1]{EGNO}.
A monoidal category with a braiding is called a \emph{braided monoidal category}, and a monoidal functor between braided monoidal categories is called a \emph{braided monoidal functor} if it is compatible with the braidings in the sense of \cite[Definition 8.1.7]{EGNO}.
The braiding $\beta$ is called \emph{symmetric} if $\beta_{X,Y} \circ \beta_{Y,X} = \id_{X \otimes Y}$ for all objects $X$ and $Y$ of $\cA$; we then also say that $\cA$ (with the braiding $\beta$) is a \emph{symmetric monoidal category} or that $\cA$ is \emph{symmetrically braided}.
A braided monoidal functor between symmetric monoidal categories is also called a \emph{symmetric monoidal functor}.

An object $X$ of a $\kk$-linear monoidal category $(\cA,\otimes,\mathbf{1},\ldots)$ is called \emph{rigid} if it admits a left dual $X^*$ and a right dual $\prescript{*}{}{X}$ as in \cite[Definition 2.10.1]{EGNO}. The $\kk$-linear monoidal category $\cA$ is called rigid if every object of $\cA$ is rigid.
As noted in the paragraph following Example 2.10.4 in \cite{EGNO}, taking left or right duals can then be enhanced to equivalences
\[ (-)^* \colon \cA^\mathrm{rev} \longrightarrow \cA^\mathrm{op} , \hspace{2cm} \prescript{*}{}{(-)} \colon \cA^\mathrm{rev} \longrightarrow \cA^\mathrm{op} . \]
If $\cA$ is a braided monoidal category, then left and right duals coincide, and if $\cA$ is a rigid braided monoidal category, then the left and right duality functors are naturally isomorphic.

\subsection{Karoubi envelopes}
\label{subsec:Karoubi}

For a $\kk$-linear (not necessarily additive) category $\cA$, we write $\mathrm{Add}(\cA)$ for the additive envelope of $\cA$, that is, the category whose objects are formal direct sums (or finite tuples) of objects of $\cA$ and whose homomorphisms are given by matrices of homomorphisms in $\cA$, as in \cite[Definition 3.2]{BarNatan-Khovanov-homology-tangles}.
Then $\mathrm{Add}(\cA)$ is a $\kk$-linear additive category.
If $\cA$ is a $\kk$-linear monoidal category, then there is a canonical monoidal structure on $\mathrm{Add}(\cA)$ such that the canonical functor $\cA \to \mathrm{Add}(\cA)$ can be enhanced to a monoidal functor.
Furthermore, a (symmetric) braiding on $\cA$ gives rise to a (symmetric) braiding on $\mathrm{Add}(\cA)$ such that the functor $\cA \to \mathrm{Add}(\cA)$ is braided monoidal, and if $\cA$ is rigid, then so is $\mathrm{Add}(\cA)$.

A $\kk$-linear additive category $\cA$ is called \emph{idempotent complete} if if all idempotents in $\cA$ split, that is, if for every object $A$ of $\cA$ with an idempotent endomorphism $e \in \End_\cA(A)$ (i.e.\ such that $e^2 = e$), there is an object $B$ of $\cA$ with homomorphisms $f \colon A \to B$ and $g \colon B \to A$ such that $e = g \circ f$ and $\id_B = f \circ g$.
The \emph{idempotent completion} $\mathrm{Split}(\cA)$ of $\cA$ is the category whose objects are pairs $(A,e)$ of an object $A$ of $\cA$ with an idempotent $e \in \End_\cA(A)$ and where the homomorphisms from $(A,e)$ to $(A',e')$ are given by the homomorphisms $f \colon A \to A'$ such that $e' \circ f = f \circ e$.
The idempotent completion $\mathrm{Split}(\cA)$ is idempotent complete and there is a canonical fully faithful functor $\imath \cA \to \mathrm{Split}(A)$ that is given on objects by $A \mapsto (A,\id_A)$ and on homomorphisms by $f \mapsto f$.
Furthermore, every $\kk$-linear functor $F \colon \cA \to \mathcal{B}$ to an idempotent complete category $\mathcal{B}$ extends uniquely (up to natural isomorphism) to a functor $\tilde F \colon \mathrm{Split}(\cA) \to \mathcal{B}$ such that $F \cong \tilde F \circ \imath$, cf. \cite[Section 2.B]{FreydAbelianCategories}.
Using this universal property, it is straightforward to see that a monoidal structure on $\cA$ gives rise to a monoidal structure on $\mathrm{Split}(\cA)$ such that $\imath$ can be enhanced to a monoidal functor, that a (symmetric) braiding on $\cA$ induces a (symmetric) braiding on $\mathrm{Split}(\cA)$ such that $\imath$ is a braided monoidal functor, and that $\mathrm{Split}(\cA)$ is rigid whenever $\cA$ is rigid.

The \emph{Karoubi envelope} of a $\kk$-linear category $\cA$ is
\[ \Kar(\cA) \coloneqq \mathrm{Split}\big( \mathrm{Add}(\cA) \big) ; \]
it is a $\kk$-linear additive and idempotent complete category with a canonical fully faithful functor $\cA \to \Kar(\cA)$ such that every object of $\Kar(\cA)$ is a direct sum of direct summands of objects of $\cA$.
If $\cA$ is a $\kk$-linear monoidal category, then by the above discussion, there is a canonical monoidal structure on $\Kar(\cA)$ such that the functor $\cA \to \Kar(\cA)$ can be enhanced to a monoidal functor.
Furthermore, a (symmetric) braiding on $\cA$ gives rise to a (symmetric) braiding on $\Kar(\cA)$ such that $\cA \to \Kar(\cA)$ is a braided monoidal functor, and if $\cA$ is rigid, then so is $\Kar(\cA)$.

\subsection{Tensor product of complexes}
\label{subsec:tensorproductofcomplexes}

Given a $\kk$-linear additive category $\cA$, we write $\mathrm{Ch}^b(\cA)$ for the category of bounded (chain) complexes
\[ A = (A_\bullet,d_\bullet) = (\cdots \to A_i \xrightarrow{\; d_i \;} A_{i+1} \to \cdots) \]
in $\cA$, with the usual shift functors $A \mapsto A[n]$ for $n \in \Z$, where $A[n]$ has terms $A[n]_i = A_{i+n}$ and differential $d[n]_i = (-1)^n \cdot d_{i+n}$.
The tautological functor $\cA \to \mathrm{Ch}^b(\cA)$ that sends objects of $\cA$ to complexes concentrated in homological degree zero will be denoted by
\[ X \longmapsto X_\mathrm{c} = ( 0 \to X \to 0 ) . \]
We also write $K^b(\cA)$ for the bounded homotopy category of $\cA$.
The objects of $K^b(\cA)$ are bounded complexes in $\cA$ and the homomorphisms are chain maps up to chain homotopy, as in \cite[Definition 11.2.4]{KashiwaraSchapiraCategoriesSheaves}.
Note that $K^b(\cA)$ is a triangulated category, with distinguished triangles given by the mapping cone triangles, cf. \cite[Theorem 11.3.8]{KashiwaraSchapiraCategoriesSheaves}.
If $\cA$ is abelian, then we further write $D^b(\cA)$ for the bounded derived category of $\cA$, that is, the Verdier localization of $K^b(\cA)$ at the class of quasi-isomorphisms, as in \cite[Definition 13.1.2]{KashiwaraSchapiraCategoriesSheaves}.

Now let $(\cA,\otimes,\mathbf{1},\ldots)$ be an additive $\kk$-linear monoidal category.
Then for complexes $A = (A_\bullet,d_\bullet)$ and $B = (B_\bullet,d'_\bullet)$ in $\cA$, we can form the tensor product double complex
\[ A \otimes_\mathrm{dc} B = \left( \begin{tikzcd}[baseline={([yshift=-axis_height]\tikzcdmatrixname)}]
    & \vdots \ar[d] & & \vdots \ar[d] & \\
    \cdots \ar[r] & A_i \otimes B_j \arrow[rr,"\id_{A_i} \otimes d'_i"] \ar[dd,swap,"d_i \otimes \id_{B_j}"] & & A_i \otimes B_{j+1} \ar[r] \ar[dd,"d_i \otimes \id_{B_{j+1}}"] & \cdots \\
    & & & & \\
    \cdots \ar[r] & A_{i+1} \otimes B_j \arrow[rr,swap,"\id_{A_{i+1}} \otimes d'_i"] \ar[d] & \phantom{AA} & A_{i+1} \otimes B_{j+1} \ar[r] \ar[d] & \cdots \\
    & \vdots & & \vdots & 
\end{tikzcd} \right) . \]
The tensor product complex $A \otimes_\mathrm{c} B$ is defined as the total complex (as in \cite[Section 11.5]{KashiwaraSchapiraCategoriesSheaves})
\[ A \otimes_\mathrm{c} B = \mathrm{tot}( A \otimes_\mathrm{dc} B ) . \]
The tensor product $- \otimes_\mathrm{c} -$ is functorial in both arguments and defines a monoidal structure on $\mathrm{Ch}^b(\cA)$, with unit object $\mathbf{1}_c = (0 \to \mathbf{1} \to 0)$ (where $\mathbf{1}$ is in homological degree zero) and with the obvious associators and unitors.
A (symmetric) braiding on $\cA$ induces a canonical (symmetric) braiding on $\mathrm{Ch}^b(\cA)$, and if $\cA$ is rigid, then so is $\mathrm{Ch}^b(\cA)$, with the left dual of a complex $A = ( \cdots A_i \to A_{i+1} \to \cdots )$ given by $A^* = ( \cdots \to A_{i+1}^* \to A_i^* \to \cdots )$, where $A_i^*$ is in homological degree $-i$.

The monoidal structure $(\otimes_\mathrm{c},\mathbf{1}_c,\ldots)$ on $\mathrm{Ch}^b(\cA)$ also descends to a monoidal structure on the homotopy category $K^b(\cA)$, which has a canonical (symmetric) braiding if $\cA$ is (symmetrically) braided, and $K^b(\cA)$ is rigid whenever $\cA$ is rigid.
Furthermore, the tensor product $\otimes_\mathrm{c}$ is exact in both arguments with respect to the triangulated structure of $K^b(\cA)$.
(This follows from the elementary fact that the functors $X\otimes_\mathrm{c}-$ and $- \otimes_\mathrm{c} X$ commute with taking mapping cones.)
If $\cA$ is abelian and $\otimes$ is exact in both arguments, then $\otimes_\mathrm{c}$ also gives rise to a monoidal structure on $D^b(\cA)$ and $\otimes_\mathrm{c}$ is exact in both arguments with respect to the triangulated structure of $D^b(\cA)$.

\subsection{Presheaves and Day convolution}
\label{subsec:functorcategories}

One of the key categorical constructions that we employ in this article is the Day convolution monoidal structure on categories of presheaves.
We first recall some generalities about presheaf categories and their universal property (see \Cref{lem:freecocompletion}).

\begin{Definition}
    A \emph{presheaf} on a $\kk$-linear category $\cA$ is a $\kk$-linear functor $F \colon \cA^\mathrm{op} \to \Vect_\kk$.
    We write
    \[ \PSh(\cA) \coloneqq \mathrm{Fun}_\kk( \cA^\mathrm{op} , \Vect_\kk ) . \]
    for the category of presheaves on $\cA$.
    The presheaf $F$ is called \emph{pointwise finite-dimensional} if $F(x)$ is finite-dimensional for all objects $x$ of $\cA$.
    We further write
    \[ \PSh^\mathrm{fd}(\cA) \coloneqq \mathrm{Fun}_\kk( \cA^\mathrm{op} , \Vect_\kk^\mathrm{fd} ) \]
    for the full subcategory of pointwise finite-dimensional presheaves.
\end{Definition}

For every $\kk$-linear category $\cA$ the \emph{Yoneda embedding} is the $\kk$-linear functor
\[ Y_\cA \colon \cA \longrightarrow \PSh(\cA) , \qquad x \mapsto \Hom_\cA(-,x) \eqqcolon h_x . \]
A presheaf $F \colon \cA^\mathrm{op} \to \Vect_\kk$ is called \emph{representable} if $F \cong h_x$ for an object $x$ of $\cA$.
For another $\kk$-linear category $\mathcal{B}$ and a $\kk$-linear functor $F \colon \cA \to \mathcal{B}$, we also consider the \emph{restricted Yoneda embedding}
\[ rY_F \colon \mathcal{B} \longrightarrow \PSh(\cA) , \qquad x \longmapsto \Hom_\cA\big( F(-) , x \big) = h_x \circ F . \]
The (enriched) Yoneda lemma can be stated by saying that there is a natural isomorphism
\[ rY_{Y_\cA} \cong \id_{\PSh(\cA)} \]
of endofunctors of $\PSh(\cA)$.
The component of this natural isomorphism at a presheaf $F$ on $\cA$ is an isomorphism
\[ \Hom_\mathrm{PSh(\cA)}(-,F) \circ Y_\cA \cong F \]
in $\PSh(\cA)$, and the latter is given by isomorphisms of $\kk$-vector spaces
\[ \Hom_\mathrm{PSh(\cA)}(h_x,F) \cong F(x) \]
for all objects $x$ of $\cA$.

Now let us fix a $\kk$-linear category $\cA$, and recall that a category is called \emph{cocomplete} if it has all small colimits and that a functor is called cocontinuous if it commutes with all small colimits.
The category $\PSh(\cA)$ is cocomplete (with colimits computed pointwise in $\Vect_\kk$), and it is the \emph{free $\kk$-linear cocompletion} of $\cA$, in that it satisfies the following universal property (see Theorem 4.51 in \cite{KellyEnriched}).

\begin{Lemma} \label{lem:freecocompletion}
	Let $\cC$ be a cocomplete $\kk$-linear category.
	For every $\kk$-linear functor $F \colon \cA \to \cC$, there is a unique (up to natural isomorphism) cocontinuous $\kk$-linear functor $\hat F \colon \PSh(\cA) \to \cC$ such that $F \cong \hat F \circ Y_\cA$.
	Furthermore, $\hat F$ is left adjoint to the restricted Yoneda embedding $rY_F \colon \cC \to \PSh(\cA)$.
\end{Lemma}

\begin{Definition}
    The functor $\hat{F} \colon \PSh(\cA) \to \cC$ in \Cref{lem:freecocompletion} is called the \emph{Yoneda extension} of $F$.
\end{Definition}

An object $X$ of an abelian category $\cC$ is called projective (or injective) if the functor $\Hom_\cC(X,-)$ (respectively $\Hom_\cC(-,X)$) is exact.
We write $\Proj(\cC)$ and $\mathrm{Inj}(\cC)$ for the full subcategories of projective objects and injective objects in $\cC$, respectively.
An object of $\cC$ is called \emph{compact} if $\Hom_\cC(X,-)$ preserves filtered colimits.
In representation theoretic contexts, compact projective objects often correspond to finitely generated projective modules (see Lemma 2.5 in \cite{BrundanStroppelSemiInfiniteHighestWeight}), and so we interchangeably use the terms \emph{finitely generated projective} and \emph{compact projective} and write $\Proj^\mathrm{fg}(\cC)$ for the full subcategory of compact projective objects in $\cC$.
We say that $\cC$ is generated by compact projective objects if for every non-zero homomorphism $X \to Y$ in $\cC$, there is a compact projective object $P$ and a homomorphism $P \to X$ such that the composition $P \to X \to Y$ is non-zero.

\begin{Lemma}
\label{lem:projectivepresheaves}
    Let $\cA$ be a $\kk$-linear category.
    Then the category $\PSh(\cA)$ is generated by compact projective objects and the Yoneda embedding $Y_\cA \colon \cA \to \PSh(\cA)$ induces an equivalence
    \[ \Kar(\cA) \simeq \Proj^\mathrm{fg}\big( \PSh(\cA) \big) . \]
\end{Lemma}
\begin{proof}
    This follows from \cite[Section 5.G]{FreydAbelianCategories}.
\end{proof}

The following variant of the Freyd--Mitchell embedding theorem provides an internal characterization of presheaf categories.

\begin{Lemma}
\label{lem:presheavesrecognition}
    Let $\cC$ be a cocomplete $\kk$-linear abelian category that is generated by compact projective objects.
    Further let $\cA = \Proj^\mathrm{fg}( \cC )$ and let $F \colon \cA \to \cC$ be the canonical embedding.
    Then the restricted Yoneda embedding $rY_F$ and the Yoneda extension $\hat F$ are quasi-inverse equivalences
    \[ rY_F \colon \cC \xrightarrow{~\sim~} \PSh(\cA) , \hspace{2cm} \hat F \colon \PSh(\cA) \xrightarrow{~\sim~} \cC . \]
\end{Lemma}
\begin{proof}
    The claim that $rY_F$ is an equivalence follows from \cite[Section 5.H]{FreydAbelianCategories}.
    Since $\hat F$ is left adjoint to $rY_F$ by \Cref{lem:freecocompletion}, this implies that $\hat F$ is a quasi-inverse of $rY_F$ (see \cite[Theorem IV.4.1]{MacLaneCategories}).
\end{proof}

Now additionally assume that $\cA$ admits a $\kk$-linear monoidal structure $(\otimes, \mathbf{1} , \ldots)$, that is, a monoidal structure such that the tensor product bifunctor $\otimes \colon \cA \times \cA \to \cA$ is bilinear at the level of homomorphisms.
Then $\PSh(\cA)$ has a canonical $\kk$-linear monoidal structure such that the Yoneda embedding $Y_\cA \colon \cA \to \PSh(\cA)$ admits the structure of a monoidal functor, with tensor product bifunctor
\[ \otimes_\mathrm{Day} \colon \PSh(\cA) \times \PSh(\cA) \longrightarrow \PSh(\cA) \]
given by the so-called \emph{Day convolution product} \cite{DayConvolution}.
The monoidal structure on $\PSh(\cA)$ is \emph{biclosed}, in that the functors $F \otimes_\mathrm{Day} -$ and $-\otimes_\mathrm{Day} F$ both admit right adjoints for every presheaf $F$ on $\cA$, and in particular, the tensor product bifunctor is cocontinuous in both arguments (because left adjoints are cocontinuous).
For the benefit of the reader, we sketch below a construction of the Day convolution tensor product using the universal property of $\PSh(\cA)$ from \Cref{lem:freecocompletion}.

\begin{Remark}
\label{rem:Dayconvolution}
	Let $(\cA,\otimes,\mathbf{1},\ldots)$ be a $\kk$-linear monoidal category.
	For every object $x$ of $\cA$, the functor $Y_\cA \circ (x \otimes -) \colon \cA \to \PSh(\cA)$ extends uniquely to a cocontinuous $\kk$-linear functor
	\[ x \mathop{\hat \otimes} - \colon \PSh(\cA) \to \PSh(\cA) \]
	by \Cref{lem:freecocompletion}, and this construction is natural in $x$.
	For a presheaf $G$ on $\cA$, the functor
	\[ - \mathop{\hat \otimes} G \colon \cA \to \PSh(\cA) \]
	extends uniquely to a cocontinuous $\kk$-linear functor
	\[ - \mathop{\tilde \otimes} G \colon \PSh(\cA) \to \PSh(\cA) , \]
	again by \Cref{lem:freecocompletion}, and this construction is natural in $G$.
    Now we can define
    \[ F \otimes_\mathrm{Day} G \coloneqq F \mathop{\tilde \otimes} G \]
    for presheaves $F$ and $G$ on $\cA$.
    Checking that $\otimes_\mathrm{Day}$ (together with the tensor unit $Y_\cA(\mathbf{1})$ and the canonical associator and unitors) defines a monoidal structure on $\PSh(\cA)$ is a straightforward exercise.
\end{Remark}

\begin{Remark}
\label{rem:Dayconvolutionbraided}
    If $\cA$ is a braided monoidal category, then $\PSh(\cA)$ admits a canonical braiding making the Yoneda embedding a braided monoidal functor (see \cite[Example 5.6]{JoyalStreetBraidedMonoidal}) and if $\cA$ is symmetric, then so is $\PSh(\cA)$ (see \cite[Theorem 3.6]{DayConvolution}).
\end{Remark}

For a $\kk$-linear monoidal category $(\cC,\otimes',\ldots)$, we say that $\cC$ is \emph{monoidally cocomplete} if $\cC$ is cocomplete and the tensor product bifunctor $\otimes'$ is cocontinuous in both arguments.
In particular, $(\PSh(\cA),\otimes_\mathrm{Day},\ldots)$ is monoidally cocomplete.
Furthermore, the Day convolution monoidal structure makes $\PSh(\cA)$ the \emph{free $\kk$-linear monoidal cocompletion} of $\cA$, in that the following universal property holds \cite[Theorem 5.1]{ImKellyConvolution}.

\begin{Lemma} \label{lem:freemonoidalcocompletion}
	Let $(\cC,\otimes',\ldots)$ be a $\kk$-linear monoidal and monoidally cocomplete category.
	For every $\kk$-linear monoidal functor $F \colon \cA \to \cC$, there is a unique (up to monoidal natural isomorphism) cocontinuous $\kk$-linear monoidal functor $\hat F \colon \PSh(\cA) \to \cC$ such that $F \cong \hat F \circ Y_\cA$ as monoidal functors.
\end{Lemma}

\begin{Remark} \label{rem:freemonoidalcocompletion}
    \begin{enumerate}
        \item \Cref{lem:freemonoidalcocompletion} implies that Day convolution gives rise to the \emph{unique} (up to monoidal equivalence) monoidal structure on $\PSh(\cA)$ such that $\PSh(\cA)$ is monoidally cocomplete and $Y_\cA$ can be enhanced to a monoidal functor.
        \item Another way of stating \Cref{lem:freemonoidalcocompletion} is that the Yoneda extension $\hat F \colon \PSh(\cA) \to \cC$ from \Cref{lem:freecocompletion} (viewed as an ordinary functor) can be enhanced (uniquely up to monoidal natural isomorphism) to a monoidal functor.
        \item If $\cA$ is a braided (or symmetric) monoidal category, $\cC$ is a $\kk$-linear braided (or symmetric) monoidal and monoidally cocomplete category and $F \colon \cA \to \cC$ is a $\kk$-linear braided (or symmetric) monoidal functor, then the monoidal functor $\hat F \colon \PSh(\cA) \to \cC$ from \Cref{lem:freemonoidalcocompletion} is a braided (or symmetric) monoidal functor, with respect to the canonical (symmetric) braiding on $\PSh(\cA)$ that makes the Yoneda embedding a braided monoidal functor.
        In the case of symmetric monoidal categories and functors, this is explicitly stated in \cite[Theorem 5.1]{ImKellyConvolution}, and the case of braided monoidal categories and functors is analogous.
    \end{enumerate}
\end{Remark}

\subsection{Ind-completions}
\label{subsec:indcompletion}

The ind-completion of a $\kk$-linear category $\cA$ is the full subcategory $\Ind(\cA)$ of $\PSh(\cA)$ whose objects are filtered colimits of representable presheaves.
It is the free completion of $\cA$ with respect to filtered colimits, in the sense that every functor from $\cA$ to a category admitting filtered colimits factors uniquely (up to isomorphism) through the canonical embedding $\cA \to \Ind(\cA)$ (see \cite[Corollary 6.3.2]{KashiwaraSchapiraCategoriesSheaves}).
If $\cA$ is a $\kk$-linear abelian category, then $\Ind(\cA)$ is abelian and cocomplete by \cite[Theorem 8.6.5]{KashiwaraSchapiraCategoriesSheaves}.
For a $\kk$-linear (left or right) exact functor $F \colon \cA \to \cA'$ between $\kk$-linear abelian categories, the corresponding functor $\Ind(F) \colon \Ind(\cA) \to \Ind(\cA')$ is (left or right) exact by \cite[Corollary 8.6.8]{KashiwaraSchapiraCategoriesSheaves}.

Now suppose that $(\cA,\otimes,\mathbf{1},\ldots)$ is a $\kk$-linear monoidal category.
Since the Day convolution tensor product $\otimes_\mathrm{Day}$ on $\PSh(\cA)$ is cocontinuous in both arguments (and hence commutes with filtered colimits in both arguments) and since tensor products of representable presheaves are representable (because the Yoneda embedding is a monoidal functor), Day convolution restricts to a bifunctor
\[ - \otimes - \colon \quad \Ind(\cA) \times \Ind(\cA) \longrightarrow \Ind(\cA) , \]
which endows $\Ind(\cA)$ with a monoidal structure $(\otimes,\mathbf{1},\ldots)$ such that the tensor product $\otimes$ on $\Ind(\cA)$ commutes with filtered colimits in both arguments and the canonical embedding $\cA \to \Ind(\cA)$ can be enhanced to a monoidal functor.%
\footnote{Alternatively, the monoidal structure on $\Ind(\cA)$ can be constructed using the universal property of the ind-completion and the canonical equivalence $\Ind(\cA \times \cA) \simeq \Ind(\cA) \times \Ind(\cA)$, see Propositions 6.1.9 and 6.1.12 in \cite{KashiwaraSchapiraCategoriesSheaves}.}
(Note that we use the same symbols $\otimes$ and $\mathbf{1}$ for the tensor product bifunctors and tensor units in $\cC$ and $\Ind(\cC)$.)
If $\cA$ is (symmetrically) braided, then there is a canonical (symmetric) braiding on $\Ind(\cA)$ such that the canonical embedding $\cA \to \Ind(\cA)$ is a braided monoidal functor.
Finally, if $\cA$ is a $\kk$-linear abelian category with a monoidal structure $(\otimes,\mathbf{1},\ldots)$ such that $\otimes$ is right exact (or exact) in both arguments, then the tensor product $\otimes$ on $\Ind(\cA)$ is right exact (or exact) in both arguments by \cite[Corollary 8.6.8]{KashiwaraSchapiraCategoriesSheaves}.
In that case, it also follows that $\otimes$ is cocontinuous in both arguments (so that $\Ind(\cC)$ is monoidally cocomplete), because every functor that is right exact (hence commutes with finite colimits) and commutes with filtered colimits (hence commutes with small coproducts) is cocontinuous; see the paragraph below Lemma 3.2.9 in \cite{KashiwaraSchapiraCategoriesSheaves}.

\section{Highest weight categories}
\label{sec:highestweightcategories}

In this section, we recall Brundan--Stroppel's formalism of semi-infinite highest weight categories and semi-infinite Ringel duality \cite{BrundanStroppelSemiInfiniteHighestWeight}.
Along the way, we also establish some preliminary results that will be needed later and (more importantly) reformulate semi-infinite Ringel duality in a way that we find more suitable to our discussion of monoidal Ringel duality in \Cref{sec:monoidalRingelduality} below.
More precisely, we rephrase ``lower to upper'' Ringel duality in terms of presheaf categories (see Subsection~\ref{subsec:Ringeldualitypresheaves}) and ``upper to lower'' Ringel duality in terms of exceptional sequences (see Subsection~\ref{subsec:Ringeldualityexceptionalsequences}).
\medskip

For the rest of the article, we fix an algebraically closed field $\kk$.%
\footnote{The assumption that $\kk$ be algebraically closed is due to our reliance on results from \cite{BrundanStroppelSemiInfiniteHighestWeight,SamSnowdenTriangular}.
We expect that most of the results of the article remain valid over fields that are not algebraically closed, at least as long as simple objects in the categories under consideration are absolutely simple, see \cite[Remark~4.3]{SamSnowdenTriangular}.
}
A \emph{$\kk$-linear category} is a category enriched over the category $\Vect_\kk$ of $\kk$-vector spaces.
All categories, functors and equivalences that appear below are $\kk$-linear, though we may sometimes omit spelling this out explicitly.

\subsection{Preliminaries on highest weight categories}
\label{subsec:highestweightcategories}

Before defining highest weight categories, we need to discuss some finiteness conditions on $\kk$-linear abelian categories, following \cite[Section 2]{BrundanStroppelSemiInfiniteHighestWeight}.

\begin{Definition}
    A $\kk$-linear abelian category $\cC$ is \emph{locally finite} if all objects of $\cC$ have finite composition series and all $\Hom$-spaces in $\cC$ are finite-dimensional.    
\end{Definition}

By a result of Takeuchi \cite{TakeuchiMoritaComodules}, a $\kk$-linear abelian category $\cC$ is locally finite if and only if there is a $\kk$-coalgebra $C$ such that $\cC$ is equivalent to the category $C\text{-}\mathrm{comod}_\mathrm{fd}$ of finite-dimensional $C$-comodules.
Then as explained in \cite[Section 2.1]{BrundanStroppelSemiInfiniteHighestWeight}, the ind-completion $\Ind(\cC)$ is equivalent to the category $C\text{-}\mathrm{comod}$ of all $C$-comodules, and this implies that $\cC$ can equivalently be described as the full subcategory of compact objects in $\Ind(\cC)$ or the full subcategory of objects of finite length in $\Ind(\cC)$.
Furthermore, it follows that $\Ind(\cC)$ has enough injectives.

\begin{Definition}
\label{def:Schurian}
    A $\kk$-linear abelian category $\cC$ is \emph{Schurian} if $\cC \simeq \PSh^\mathrm{fd}(\cA)$ for an essentially small $\kk$-linear category $\cA$ with finite-dimensional $\Hom$-spaces.
\end{Definition}

\begin{Remark}
    A \emph{locally finite-dimensional locally unital $\kk$-algebra} is a (not necessarily unital) $\kk$-algebra $A$ with a distinguished collection of idempotents $\{ e_i \mid i \in I \}$ such that $A = \bigoplus_{i,j \in I} e_i A e_j$ and the subspaces $e_i A e_j$ are finite-dimensional for all $i,j \in I$.
    A \emph{locally finite-dimensional} $A$-module is an $A$-module $V$ such that $e_i V$ is finite-dimensional for all $i \in I$.
    We write $A\text{-}\mathrm{mod}_\mathrm{lfd}$ for the category of locally finite-dimensional $A$-modules.
    In \cite[Section 2.3]{BrundanStroppelSemiInfiniteHighestWeight}, a $\kk$-linear abelian category is defined to be Schurian if $\cC \simeq A\text{-} \mathrm{mod}_\mathrm{lfd}$ for some locally finite-dimensional locally unital algebra $A$.
    This matches \Cref{def:Schurian} by \cite[Remark 2.3]{BrundanStroppelSemiInfiniteHighestWeight}.
    We prefer to work with categories of (pointwise finite-dimensional) presheaves because this allows us to more directly use Yoneda extensions and Day convolution to define functors and monoidal structures on (categories related to) Schurian categories, see \Cref{thm:monoidalRingelduality_lowertoupper} below.
\end{Remark}

For a $\kk$-linear abelian category $\cC$ that is locally finite or Schurian, it will often be necessary to pass to a larger cocomplete category, denoted by $\cC^\mathrm{cc}$.
If $\cC$ is locally finite, then we define
\[ \cC^\mathrm{cc} \coloneqq \Ind(\cC) . \]
If $\cC$ is Schurian, then we write $\cC_c$ for the (possibly non-abelian) full subcategory of compact objects in $\cC$, and we define
\[ \cC^\mathrm{cc} \coloneqq \Ind(\cC_c) . \]
Note that these definitions are unambiguous:
If $\cC$ is Schurian and locally finite, then $\cC$ is equivalent to the category of finite-dimensional $A$-modules for a finite-dimensional $\kk$-algebra $A$ by Lemma 2.21 in \cite{BrundanStroppelSemiInfiniteHighestWeight}.
All finite-dimensional $A$-modules are finitely-presented, hence compact by Lemma 2.5 in \cite{BrundanStroppelSemiInfiniteHighestWeight}, and it follows that $\cC = \cC_c$ and the two definitions of $\cC^\mathrm{cc}$ coincide.

\begin{Remark}
\label{rem:presheavesonprojectives}
	Let $\mathcal{D}$ be a Schurian category.
    Further let $\cA = \Proj^\mathrm{fg}(\mathcal{D})$ and let $F \colon \cA \to \mathcal{D}^\mathrm{cc}$ be the canonical embedding.
	Then by Lemmas~\ref{lem:projectivepresheaves} and \ref{lem:presheavesrecognition}, the Yoneda extension and the restricted Yoneda embedding are quasi-inverse equivalences
	\[ \hat F \colon \PSh(\cA) \xrightarrow{~\sim~} \mathcal{D}^\mathrm{cc} \hspace{2cm} rY_F \colon \mathcal{D}^\mathrm{cc} \xrightarrow{~\sim~} \PSh(\cA) . \]
	Furthermore, $\hat F$ and $rY_F$ restrict to equivalences $\mathcal{D} \simeq \PSh(\cA)^\mathrm{fd}$.
\end{Remark}

A \emph{weight datum} for $\cC$ is a triple $(\Lambda,\leq,L)$, where $(\Lambda,\leq)$ is a poset and $L \colon \Lambda \to \mathrm{Ob}(\cC)$ is a function indexing a set $\{ L(\lambda) \mid \lambda \in \Lambda \}$ of representatives for the isomorphism classes of simple objects in $\cC$.
We call $L(\lambda)$ the simple object of \emph{highest weight} $\lambda$.
If $\cC$ is Schurian, then we further denote by $P(\lambda)$ the projective cover of $L(\lambda)$, and if $\cC$ is locally finite abelian, then we write $I(\lambda)$ for the injective hull of $L(\lambda)$ in the ind-completion $\Ind(\cC)$.
For every subset $\Gamma \subseteq \Lambda$, we denote by $\cC_\Gamma$ the Serre subcategory of $\cC$ generated by the simple objects $L(\lambda)$ with $\lambda \in \Gamma$, and we abbreviate
\[ \cC_{\leq \lambda} = \cC_{ \{ \mu \in \Lambda \mid \mu \leq \lambda \} } , \hspace{2cm} \cC_{< \lambda} = \cC_{ \{ \mu \in \Lambda \mid \mu < \lambda \} } . \]
The weight datum $(\Lambda,\leq,L)$ is called \emph{admissible} if the Serre quotient categories $\cC_\lambda \coloneqq \cC_{\leq \lambda} / \cC_{< \lambda}$ are equivalent to the category $\Vect_\kk^\mathrm{fd}$ of finite-dimensional vector spaces for all $\lambda \in \Lambda$.
Note that a weight datum $(\Lambda,\leq,L)$ on $\cC$ restricts to a weight datum $(\Gamma,\leq,L)$ on $\cC_\Gamma$ for every order ideal $\Gamma \subseteq \Lambda$, and that the latter is admissible if the former is admissible.
If $(\Lambda,\leq,L)$ is admissible and $\cC$ is Schurian or locally finite abelian, then by Lemmas 2.24 and 2.26 in \cite{BrundanStroppelSemiInfiniteHighestWeight}, the quotient functor $j^\lambda \colon \cC_{\leq \lambda} \to \cC_\lambda$ has a left adjoint $j^\lambda_!$ and a right adjoint $j^\lambda_*$, and we define the \emph{standard object} and the \emph{costandard object} of highest weight $\lambda$ via
\[ \Delta(\lambda) \coloneqq j^\lambda_! j^\lambda L(\lambda) , \hspace{2cm} \nabla(\lambda) \coloneqq j^\lambda_* j^\lambda L(\lambda) . \]
The following properties of standard objects and costandard objects are consequences of Lemma 3.1, Corollary 3.2 and Lemma 3.4 in \cite{BrundanStroppelSemiInfiniteHighestWeight}.
\begin{enumerate}
	\item the standard object $\Delta(\lambda)$ is a projective cover of $L(\lambda)$ in $\cC_{\leq \lambda}$;
	\item the costandard object $\nabla(\lambda)$ is an injective hull of $L(\lambda)$ in $\cC_{\leq \lambda}$;
	\item the unique maximal subobject of $\Delta(\lambda)$ belongs to $\cC_{< \lambda}$;
	\item the unique maximal quotient of $\nabla(\lambda)$ belongs to $\cC_{< \lambda}$;
	\item we have $\Hom_\cC\big( \Delta(\lambda) , \nabla(\lambda) \big) \cong \kk$ and $\Hom_\cC\big( \Delta(\lambda) , \nabla(\mu) \big) = 0$ for $\lambda,\mu \in \Lambda$ with $\lambda \neq \mu$.
\end{enumerate}
A \emph{standard filtration} of an object $M$ of $\cC$ is a filtration $0 = M_0 \subseteq M_1 \subseteq \cdots \subseteq M_r = M$ such that $M_i / M_{i-1} \cong \Delta(\lambda_i)$ for some $\lambda_i \in \Lambda$, for $i=1,\ldots,r$.
Costandard filtrations are defined analogously, and we write $\cC_\Delta$ and $\cC_\nabla$ for the full subcategories of $\cC$ whose objects are the objects of $\cC$ that admit a standard filtration or a costandard filtration, respectively.

For a Schurian (or finite abelian) category $\cC$ with an admissible weight datum $(\Lambda,\leq,L)$, we consider the following condition:
\begin{itemize}
	\item[($P\Delta$)] For all $\lambda \in \Lambda$, the projective cover $P(\lambda)$ of $L(\lambda)$ admits a standard filtration
	\[ 0 = M_0 \subseteq M_1 \subseteq \cdots \subseteq M_r = P(\lambda) \]
	with $M_r / M_{r-1} \cong \Delta(\lambda)$ and $M_i / M_{i-1} \cong \Delta(\mu_i)$ for $\mu_i \in \Lambda$ with $\mu_i > \lambda$ for $i = 1 , \ldots , r-1$.
\end{itemize}

\begin{Definition}
	\begin{enumerate}
		\item A \emph{finite highest weight category} is a finite abelian category $\cC$ with an admissible weight datum $(\Lambda,\leq,L)$ such that the condition $(P\Delta)$ is satisfied.
		\item An \emph{upper finite highest weight category} is a Schurian category $\cC$ with an admissible weight datum $(\Lambda,\leq,L)$ such that the poset $(\Lambda,\leq)$ is upper finite and the condition $(P\Delta)$ is satisfied.
		\item A \emph{lower finite highest weight category} is a locally finite abelian category $\cC$ with an admissible weight datum $(\Lambda,\leq,L)$ such that the poset $(\Lambda,\leq)$ is lower finite and for every finite order ideal $\Gamma \subseteq \Lambda$, the subcategory $\cC_\Gamma$ is a finite highest weight category with respect to the restricted weight datum $(\Gamma,\leq,L)$.
	\end{enumerate}
\end{Definition}

\begin{Remark}
	A lower finite highest weight category can be equivalently defined as a locally finite abelian category $\cC$ with an admissible weight datum $(\Lambda,\leq,L)$ such that the poset $(\Lambda,\leq)$ is lower finite and we have
	\[ \Ext_\cC^1\big( \Delta(\lambda) , \nabla(\mu) \big) = 0 , \hspace{2cm} \Ext_\cC^2\big( \Delta(\lambda) , \nabla(\mu) \big) = 0 \]
	for $\lambda,\mu \in \Lambda$; see Corollary 3.64 in \cite{BrundanStroppelSemiInfiniteHighestWeight}.
	There is yet another equivalent characterization of lower finite highest weight categories in terms of (possibly infinite) ascending costandard filtrations for the indecomposable injective objects in $\Ind(\cC)$; we do not recall the details here and refer the reader to Corollaries 3.57 and 3.61 in \cite{BrundanStroppelSemiInfiniteHighestWeight} or Theorem 3.1.3 in \cite{CoulembierIndCompletionHighestWeight}.
\end{Remark}

\begin{Remark}
    Let $\cC$ be a lower finite (or upper finite) highest weight category with weight datum $(\Lambda,\leq,L)$.
    Then the opposite category $\cC^\mathrm{op}$ is a lower finite (or upper finite) highest weight category with the same weight datum, but with the standard objects and costandard objects interchanged, that is $\Delta^\mathrm{op}(\lambda) = \nabla(\lambda)$ and $\nabla^\mathrm{op}(\lambda) = \Delta(\lambda)$.
    See Theorem 3.9 and the paragraphs below Definition 3.34 and Remark 3.51 in \cite{BrundanStroppelSemiInfiniteHighestWeight}.
\end{Remark}

For applications in Subsection~\ref{subsec:monoidalRingelduality_uppertolower}, we will need the following explicit description of the opposite category of an upper finite highest weight category.

\begin{Remark}
\label{rem:oppositeupperfinite}
	Let $\mathcal{D}$ be an upper finite highest weight category.
    Further let $\cA = \Proj^\mathrm{fg}(\mathcal{D})$ and let $F \colon \cA \to \mathcal{D}$ be the canonical embedding.
	Then the restricted Yoneda embedding $rY_F \colon \mathcal{D} \to \PSh^\mathrm{fd}(\cA)$ is an equivalence by \Cref{rem:presheavesonprojectives}, and this gives rise to an equivalence
	\begin{equation} \label{eq:presheavesoppositedual}
	    \mathcal{D}^\mathrm{op} \simeq \PSh^\mathrm{fd}(\cA)^\mathrm{op} \xrightarrow{~\sim~} \PSh^\mathrm{fd}(\cA^\mathrm{op}) , \qquad G \longmapsto (-)^* \circ G ,
	\end{equation}
	where $(-)^* \colon \Vect_\kk^\mathrm{fd} \to \Vect_\kk^{\mathrm{fd},\mathrm{op}}$ denotes the usual duality functor on finite-dimensional vector spaces.
	The equivalence \eqref{eq:presheavesoppositedual} further gives rise to equivalences $\mathcal{D}^{\mathrm{op},\mathrm{cc}} \simeq \PSh(\cA^\mathrm{op})$ and $\Proj^\mathrm{fg}(\mathcal{D}^\mathrm{op}) \simeq \cA^\mathrm{op}$, the latter by \Cref{lem:projectivepresheaves}.
\end{Remark}

For later use, we record three elementary (and well-known) results on highest weight categories.

\begin{Lemma} \label{lem:projectiveresolutionDelta}
	Let $\cC$ be a finite or upper finite highest weight category with weight datum $(\Lambda,\leq,L)$.
	Then for all $\lambda \in \Lambda$, the standard object $\Delta(\lambda)$ has a finite projective resolution
	\[ \cdots \to P_{-2} \to P_{-1} \to P_0 \to \Delta(\lambda) \to 0 \]
	such that $P_0 \cong P(\lambda)$ and $P_i$ is a finite direct sum of certain $P(\mu)$ with $\mu > \lambda$ for all $i<0$.
\end{Lemma}
\begin{proof}
	This is an easy consequence of the condition ($P\Delta$) and the horseshoe lemma, using induction and the fact that $(\Lambda,\leq)$ is (upper) finite.
\end{proof}

\begin{Lemma} \label{lem:Extvanishingexceptional}
	Let $\cC$ be an upper or lower finite highest weight category with weight datum $(\Lambda,\leq,L)$.
	For $\lambda,\mu \in \Lambda$ and $i \geq 0$, we have
	\begin{enumerate}
		\item $\Ext_\cC^i\big( L(\lambda) , \nabla(\mu) \big) = 0$ and $\Ext_\cC^i\big( \nabla(\lambda) , \nabla(\mu) \big) = 0$ unless $\lambda > \mu$ or $\lambda=\mu$ and $i=0$;
		\item $\Ext_\cC^i\big( \Delta(\lambda) , L(\mu) \big) = 0$ and $\Ext_\cC^i\big( \Delta(\lambda) , \Delta(\mu) \big) = 0$ unless $\lambda < \mu$ or $\lambda=\mu$ and $i=0$.
\end{enumerate}
\end{Lemma}
\begin{proof}
	If $\cC$ is a finite or upper finite highest weight category, then the $\Ext$-groups in part (2) of the lemma can be computed using a projective resolution of $\Delta(\lambda)$ as in \Cref{lem:projectiveresolutionDelta}.
	Then the $\Ext$-vanishing properties in (2) are obvious, and (1) follows by passing to the opposite category $\cC^\mathrm{op}$.
	If $\cC$ is a lower finite highest weight category, then the $\Ext$-groups in (1) and (2) can be computed in the finite highest weight category $\cC_\Gamma$ for any finite order ideal $\Gamma \subseteq \Lambda$ containing $\lambda$ and $\mu$, and so the claims reduce to the finite case, which was treated above.
\end{proof}

\begin{Lemma} \label{lem:Extvanishingstandardcostandard}
	Let $\cC$ be an upper or lower finite highest weight category with weight datum $(\Lambda,\leq,L)$.
	For $\lambda,\mu \in \Lambda$ and $i \geq 0$, we have
	\[ \Ext_\cC^i\big( \Delta(\lambda) , \nabla(\mu) \big) \cong \begin{cases} \kk & \text{if } \lambda = \mu \text{ and } i=0 , \\ 0 & \text{otherwise} . \end{cases} \]
\end{Lemma}
\begin{proof}
	As in the proof of \Cref{lem:Extvanishingexceptional}, we see that $ \Ext_\cC^i\big( \Delta(\lambda) , \nabla(\mu) \big) = 0$ unless $\lambda<\mu$ or $\lambda = \mu$ and $i=0$.
	By passing to the opposite category $\cC^\mathrm{op}$, we further obtain that $ \Ext_\cC^i\big( \Delta(\lambda) , \nabla(\mu) \big) = 0$ unless $\lambda>\mu$ or $\lambda = \mu$ and $i=0$, and the claim is immediate.
\end{proof}

\subsection{Tilting objects and Ringel duality}
\label{subsec:tiltingobjectsRingelduality}

Let $\cC$ be a lower finite highest weight category with weight poset $(\Lambda,\leq)$.
An object of $\cC$ is called a \emph{tilting object} if it admits both a standard filtration and a costandard filtration.
We write $\Tilt(\cC)$ for the full subcategory of tilting objects in $\cC$.
The following classification of tilting objects is originally due to Ringel \cite{RingelAlmostSplit}; we refer the reader to \cite[Theorem 4.2]{BrundanStroppelSemiInfiniteHighestWeight} or \cite[Theorem 7.14]{RicheHabilitation} for a proof in our setting.

\begin{Proposition}
    Let $\cC$ be a lower finite highest weight category with weight poset $(\Lambda,\leq)$.
    \begin{enumerate}
        \item For all $\lambda \in \Lambda$, there is a unique (up to isomorphism) indecomposable tilting object $T(\lambda)$ such that $T(\lambda)$ belongs to $\cC_{\leq \lambda}$ and $[T(\lambda) : L(\lambda)] = 1$.
        \item Every indecomposable tilting object is isomorphic to $T(\lambda)$ for some $\lambda \in \Lambda$.
        \item Every tilting object is isomorphic to a finite direct sum of indecomposable tilting objects.
    \end{enumerate}
\end{Proposition}

Now let $\cC$ be an upper finite highest weight category with weight poset $(\Lambda,\leq)$.
Following \cite[Section 4.3]{BrundanStroppelSemiInfiniteHighestWeight}, we say that an object of $\cC$ is a \emph{tilting object} if it admits a (possibly infinite) ascending standard filtration and a (possibly infinite) descending costandard filtration.
In this setting, we have the following classification of tilting objects; see Theorem 4.18 and Corollary 4.19 in \cite{BrundanStroppelSemiInfiniteHighestWeight}.

\begin{Proposition}
    Let $\cC$ be an upper finite highest weight category with weight poset $(\Lambda,\leq)$.
    \begin{enumerate}
        \item For all $\lambda \in \Lambda$, there is a unique (up to isomorphism) indecomposable tilting object $T(\lambda)$ such that $T(\lambda)$ belongs to $\cC_{\leq \lambda}$ and $[T(\lambda):L(\lambda)] = 1$.
        \item For all $(n_\lambda)_{\lambda \in \Lambda} \in \N^\Lambda$, the direct sum $\bigoplus_{\lambda \in \Lambda} T(\lambda)^{\oplus n_\lambda}$ is a tilting object.
        \item Every tilting object is isomorphic to a direct sum
        $\bigoplus_{\lambda \in \Lambda} T(\lambda)^{\oplus n_\lambda} $
        for uniquely determined multiplicities $(n_\lambda)_{\lambda \in \Lambda} \in \N^\Lambda$.
    \end{enumerate}
\end{Proposition}

Lower finite highest weight categories and upper finite highest weight categories are related via \emph{Ringel duality}, originally due to Ringel \cite[Section 6]{RingelAlmostSplit} in the finite setting and generalized to the semi-infinite setting we consider here by Brundan--Stroppel \cite[Section 4]{BrundanStroppelSemiInfiniteHighestWeight}.
In the following, we recall the approach to Ringel duality from \cite{BrundanStroppelSemiInfiniteHighestWeight} and reformulate it using free cocompletions.

Let $\cC$ be a lower finite highest weight category with weight datum $(\Lambda,\leq,L)$ and let $(T_i)_{i \in I}$ be a collection of tilting objects with the property that for all $\lambda \in \Lambda$, there is an element $i \in I$ such that $T(\lambda)$ is a direct summand of $T_i$.
Then we consider the locally unital algebra
\begin{equation} \label{eq:Ringeldualalgebra}
	A = \Big( \bigoplus_{i,j \in I} \Hom_\cC(T_i,T_j) \Big)^\mathrm{op}
\end{equation}
and call the category
\[ \cC^\vee = A \text{-} \mathrm{mod}_\mathrm{lfd} \]
of locally finite-dimensional $A$-modules the \emph{Ringel dual} of $\cC$ with respect to $(T_i)_{i \in I}$.
We also write
\[ \cC^{\vee,\mathrm{cc}} = A \text{-} \mathrm{mod} \]
for the (cocomplete) category of all $A$-modules and note that $\cC^{\vee,\mathrm{cc}}$ is canonically equivalent to the ind-completion $\Ind(\cC^\vee_c)$ of the full subcategory of compact objects in $\cC^\vee$, as pointed out in \cite[Section 2.3]{BrundanStroppelSemiInfiniteHighestWeight}.
The Ringel duality functor with respect to $(T_i)_{i \in I}$ is defined via
\[ R = R_\cC \coloneqq \bigoplus_{i \in I} \Hom_\cC( T_i , - ) \colon \quad \cC^\mathrm{cc} \longrightarrow \cC^{\vee,\mathrm{cc}} , \]
and it restricts to a functor from $\cC$ to $\cC^\vee$.
The following result is Theorem 4.25 in \cite{BrundanStroppelSemiInfiniteHighestWeight}.

\begin{Theorem} \label{thm:Ringelduality}
	In the setup discussed above, $\cC^\vee$ is an upper finite highest weight category with weight datum $(L^\vee,\Lambda,\leq^\mathrm{op})$, where the distinguished objects are given by
	\[ \Delta^\vee(\lambda) = R \nabla(\lambda) , \qquad P^\vee(\lambda) = R T(\lambda) , \qquad L^\vee(\lambda) = \mathrm{hd} P^\vee(\lambda) , \qquad T^\vee(\lambda) = R I(\lambda) \]
	for all $\lambda \in \Lambda$, where $\mathrm{hd}$ denotes the head (or top) of an object.
	Furthermore, the functor $R$ restricts to equivalences $\cC_\nabla \simeq \cC^\vee_\Delta$ and $\Tilt(\cC) \simeq \Proj^\mathrm{fg}(\cC^\vee)$ and sends short exact sequences in $\cC_\nabla$ to short exact sequences in $\cC^\vee_\Delta$.
\end{Theorem}

By Theorem 4.27 and Corollary 4.29 in \cite{BrundanStroppelSemiInfiniteHighestWeight}, the Ringel duality functor $R \colon \cC^\mathrm{cc} \to \cC^{\vee,\mathrm{cc}}$ admits a left adjoint functor
\[ L = L_\cC \colon \cC^{\vee,\mathrm{cc}} \longrightarrow \cC^\mathrm{cc} \]
which satisfies $L \Delta^\vee(\lambda) \cong \nabla(\lambda)$ and $L P^\vee(\lambda) \cong T(\lambda)$ for all $\lambda \in \Lambda$ and restricts to equivalences $\cC^\vee_\Delta \cong \cC_\nabla$ and $\Proj^\mathrm{fg}(\cC^\vee) \cong \Tilt(\cC)$.
Furthermore, $L$ sends exact sequences in $\cC^\vee_\Delta$ to exact sequences in $\cC$ by the proof of \cite[Theorem 4.27]{BrundanStroppelSemiInfiniteHighestWeight}. 
We call $L$ the \emph{(opposite) Ringel duality functor}.

\subsection{Lower to upper Ringel duality and functor categories}
\label{subsec:Ringeldualitypresheaves}

For the applications we want to discuss below, it will be convenient to rephrase Ringel duality in a way that does not rely on a choice of tilting objects $(T_i)_{i \in I}$ in $\cC$ and an explicit realization of $\cC^\vee$ as a category of modules over an explicit locally unital algebra.
This is the subject of the following proposition.
Recall that for a lower finite highest weight category $\cC$, we write $\cC^\vee$ for the Ringel dual of $\cC$ and $R = R_\cC \colon \cC^\mathrm{cc} \to \cC^{\vee,\mathrm{cc}}$ and $L = L_\cC \colon \cC^{\vee,\mathrm{cc}} \to \cC^\mathrm{cc}$ for the Ringel duality functors.
Also recall the restricted Yoneda embedding and the Yoneda extension from Subsection~\ref{subsec:functorcategories}.

\begin{Proposition} \label{prop:Ringeldualpresheaves}
	Let $\cC$ be a lower finite highest weight category.
    Further let $\cA = \Tilt(\cC)$ and write $F \colon \cA \to \cC$ for the canonical embedding.
	Then there is an equivalence
	\[ \psi \colon \cC^{\vee,\mathrm{cc}} \xrightarrow{~\sim~} \PSh(\cA) \]
	such that $\psi \circ R \cong rY_F$ and $L \cong \hat F \circ \psi$.
	Furthermore, $\psi$ restricts to an equivalence $\cC^\vee \xrightarrow{ ~ \sim ~ } \PSh^\mathrm{fd}(\cA)$.
\end{Proposition}
\begin{proof}
    Consider the functor $G = R \circ F \colon \cA \to \cC^{\vee,\mathrm{cc}}$ and recall that $G$ induces an equivalence $\cA \simeq \Proj^\mathrm{fg}(\cC^\vee)$ by \Cref{thm:Ringelduality}.
    As $\cC^{\vee,\mathrm{cc}}$ is cocomplete and $\Proj^\mathrm{fg}(\cC^\vee)$ is a generating set of compact projective objects in $\cC^{\vee,\mathrm{cc}}$ (cf.\ \Cref{lem:projectivepresheaves}), the restricted Yoneda embedding
    \[ \psi = rY_G \colon \cC^{\vee,\mathrm{cc}} \longrightarrow \PSh(\cA) . \]
    is an equivalence by \Cref{lem:presheavesrecognition}, with quasi-inverse given by the Yoneda extension $\hat G \colon \PSh(\cA) \to \cC^{\vee,\mathrm{cc}}$.
    By \Cref{lem:freecocompletion} and \Cref{thm:Ringelduality}, we have natural isomorphisms
    \[ L \circ \hat G \circ Y_\cA \cong L \circ G = L \circ R \circ F \cong F \cong \hat F \circ Y_\cA , \]
    and the uniqueness statement in \Cref{lem:freecocompletion} implies that $\hat F \cong L \circ \hat G$ and $L \cong \hat F \circ \psi$.
    The natural isomorphism $\psi \circ R \cong rY_F$ follows from the adjunctions $L \dashv R$ and $\hat F \dashv rY_F$ (see \Cref{thm:Ringelduality} and \Cref{lem:freecocompletion}).
    Finally, note that $\psi$ clearly restricts to a fully faithful functor $\cC^\vee \to \PSh(\cA)^\mathrm{fd}$.
    For an object $x$ of $\cA$ and an object $y$ of $\cC^{\vee,\mathrm{cc}}$ such that $\psi(x)$ belongs to $\PSh(\cA)^\mathrm{fd}$, we have
    \[ \Hom_{\cC^\vee}\big( R(x) , y \big) \cong \Hom_{\PSh(\cA)}\big( \psi \circ R(x) , \psi(y) \big) \cong \Hom_{\PSh(\cA)}\big( rY_F(x) , \psi(y) \big) \cong \psi(y)(x) \]
    by the Yoneda lemma, hence all composition multiplicities in $y$ are finite and $y$ belongs to $\cC^\vee$ by Lemma 2.13 in \cite{BrundanStroppelSemiInfiniteHighestWeight}.
    Hence $\psi$ restricts to an equivalence $\cC^\vee \xrightarrow{ ~ \sim ~ } \PSh^\mathrm{fd}(\cA)$, as claimed.
\end{proof}

As an immediate consequence of \Cref{prop:Ringeldualpresheaves}, we obtain the following criterion for when an upper finite highest weight category is the Ringel dual of a given lower finite highest weight category.

\begin{Remark} \label{rem:tiltingequivalentprojectiveimpliesRingeldual}
    Suppose that we are given a lower finite highest weight category $\cC$ with weight poset $(\Lambda,\leq)$ and an upper finite highest weight category $\mathcal{D}$ with weight poset $(\Lambda,\leq^\mathrm{op})$.
    Further suppose that there is a $\kk$-linear equivalence $\psi \colon \Tilt(\cC) \xrightarrow{~\sim~} \Proj^\mathrm{fg}(\mathcal{D})$ such that $\psi T_\cC(\lambda)\cong P_\mathcal{D}(\lambda)$ for all $\lambda \in \Lambda$.
    Then there are equivalences
    \begin{gather*}
        \tilde \psi \colon \quad \cC^{\vee,\mathrm{cc}} \simeq \PSh\big( \Tilt(\cC) \big) \simeq \PSh\big( \Proj^\mathrm{fg}( \mathcal{D} ) \big) \simeq \mathcal{D}^\mathrm{cc} , \phantom{\quad \colon \tilde \psi} \\
        \cC^\vee \simeq \PSh^\mathrm{fd}\big( \Tilt(\cC) \big) \simeq \PSh^\mathrm{fd}\big( \Proj^\mathrm{fg}( \mathcal{D} ) \big) \simeq \mathcal{D} 
    \end{gather*}
    by \Cref{prop:Ringeldualpresheaves} and \Cref{rem:presheavesonprojectives}.
    By construction, we have
    \[ \tilde \psi L^\vee(\lambda) \cong L_\mathcal{D}(\lambda) , \hspace{2cm} \tilde \psi \Delta^\vee(\lambda) \cong \Delta_\mathcal{D}(\lambda) , \hspace{2cm} \tilde \psi P^\vee(\lambda) \cong P_\mathcal{D}(\lambda) \]
    for all $\lambda \in \Lambda$, and for the Ringel duality functor $R = R_\cC \colon \cC \to \cC^\vee$, there is a natural isomorphism
    \[ \tilde \psi \circ R\vert_{\Tilt(\cC)} \cong \psi \]
    of functors from $\Tilt(\cC)$ to $\Proj^\mathrm{fg}(\mathcal{D})$.
\end{Remark}

As a further consequence of \Cref{prop:Ringeldualpresheaves}, we explain how a duality on a lower finite highest weight category gives rise to a duality on its Ringel dual.

\begin{Remark} \label{rem:Chevalleydualitylowertoupper}
    Let $\cC$ be a lower finite highest weight category with weight poset $(\Lambda,\leq)$ and suppose that there is an equivalence $\tau \colon \cC \to \cC^\mathrm{op}$ and an order automorphism $\sigma \colon \Lambda \to \Lambda$ such that $\tau\big( L(\lambda) \big) \cong L( \sigma \lambda )$ for all $\lambda \in \Lambda$.
    Then it follows that
    \[ \tau\big( \Delta(\lambda) \big) \cong \nabla(\sigma \lambda) , \hspace{2cm} \tau\big( \nabla(\lambda) \big) \cong \Delta(\sigma\lambda) , \hspace{2cm} \tau\big( T(\lambda) \big) \cong T(\sigma\lambda) \]
    for all $\lambda \in \Lambda$, and in particular $\tau$ restricts to an equivalence $\Tilt(\cC) \cong \Tilt(\cC)^\mathrm{op} = \Tilt(\cC^\mathrm{op})$.
    This in turn gives rise to an equivalence
    \[ \tau^\vee \colon \quad \cC^\vee \simeq \PSh^\mathrm{fd}\big( \Tilt(\cC) \big) \simeq \PSh^\mathrm{fd}\big( \Tilt(\cC)^\mathrm{op} \big) \simeq \PSh^\mathrm{fd}\big( \Tilt(\cC) \big)^\mathrm{op} \simeq \cC^{\vee,\mathrm{op}} , \]
    where the first and last equivalence are given by \Cref{prop:Ringeldualpresheaves}, the second equivalence is obtained by composition with $\tau$, and the third equivalence is obtained by composition with $(-)^* \colon \Vect_\kk^\mathrm{fd} \to \Vect_\kk^{\mathrm{fd},\mathrm{op}}$, the usual duality funcor on finite-dimensional vector spaces.
    By construction, the equivalence $\tau^\vee$ satisfies $\tau^\vee\big( L^\vee(\lambda) \big) \cong L^\vee(\sigma\lambda)$ and therefore also
    \[ \tau^\vee\big( \Delta^\vee(\lambda) \big) \cong \nabla^\vee(\sigma \lambda) , \hspace{2cm} \tau^\vee\big( \nabla^\vee(\lambda) \big) \cong \Delta^\vee(\sigma \lambda) , \hspace{2cm} \tau^\vee\big( T^\vee(\lambda) \big) \cong T^\vee(\sigma \lambda) \]
    for all $\lambda \in \Lambda$.
\end{Remark}

A full subcategory $\mathcal{R}$ in a category $\mathcal{S}$ is called a \emph{reflective subcategory} if the embedding $\mathcal{R} \to \mathcal{S}$ has a left adjoint functor.
Given a pair of functors $L \colon \mathcal{R} \to \mathcal{S}$ and $R \colon \mathcal{S} \to \mathcal{R}$ with an adjunction $L \dashv R$, the functor $R$ exhibits $\mathcal{R}$ as a reflective subcategory of $\mathcal{S}$ if and only if $L$ is fully faithful.
For later use, we provide a sufficient condition for when the Ringel duality functor exhibits a lower finite highest weight category $\cC$ as a reflective subcategory of its Ringel dual $\cC^\vee$.
(This situation occurs if $\cC$ is the category of representations of a quantum group at a root of unity; see \Cref{sec:affineLiealgebras} below.)

\begin{Proposition} \label{prop:RingelDualReflective}
    Let $\cC$ be a lower finite highest weight category.
    Suppose that $\cC$ has enough projective objects and that all projective objects are tilting objects.
    Then the Ringel duality functor
    \[ R = R_\cC \colon \cC^\mathrm{cc} \longrightarrow \cC^{\vee,\mathrm{cc}} \]
    is fully faithful; hence it exhibits $\cC^\mathrm{cc}$ as a reflective subcategory of $\cC^{\vee,\mathrm{cc}}$.
\end{Proposition}
\begin{proof}
    Let us write $\cA = \Tilt(\cC)$ and $F \colon \cA \to \cC^\mathrm{cc}$ for the canonical embedding.
    By \Cref{prop:Ringeldualpresheaves}, it suffices to show that the restricted Yoneda embedding $rY_F \colon \cC^\mathrm{cc} \to \PSh(\cA)$ is fully faithful.
    As all objects of $\cC$ (and hence of $\cA$) are compact in $\cC^\mathrm{cc}$, the functor $rY_F$ commutes with filtered colimits.
    Furthermore, $rY_F$ restricts to a functor $rY_F^\circ \colon \cC \to \PSh^\mathrm{fd}(\cA)$, and $rY_F$ is the unique extension of $rY_F^\circ$ to $\cC^\mathrm{cc} = \Ind(\cC)$ that commutes with filtered colimits (see \cite[Corollary 6.3.2]{KashiwaraSchapiraCategoriesSheaves}).
    Again using the fact that all objects of $\cC$ are compact in $\cC^\mathrm{cc}$, we get from \cite[Proposition 6.3.4]{KashiwaraSchapiraCategoriesSheaves} that $rY_F$ is fully faithful provided that $rY_F^\circ$ is fully faithful.
    This is what we will prove below.
    
    First note that for a homomorphism $f \colon X \to Y$ in $\cC$ and for a projective cover $\pi_X \colon P_X \to X$, we have $f \circ \pi_X= 0$ if and only if $f=0$.
    Since $P_X$ belongs to $\cA$, this implies that $rY_F^\circ$ is faithful.

    Now let $X$ and $Y$ be objects of $\cC$ and let
    \[ \xi \colon \quad rY_F^\circ(X) = \Hom_{\cC}\big( F(-) , X \big) \longrightarrow rY_F^\circ(Y) = \Hom_{\cC}\big( F(-) , Y \big) \]
    be a homomorphism in $\PSh(\cA)$ (i.e.\ a natural transformation).
    We claim that $\xi$ can be lifted to a natural transformation
    \[ \hat \xi \colon \Hom_\cC(-,X) \longrightarrow \Hom_\cC( - , Y ) \]
    such that $\hat \xi_T = \xi_T$ for all tilting objects $T$ of $\cC$.
    To that end, observe that the functor
    \[ \Hom_\cC( - , X ) = \Ext_\cC^0( - , X ) \colon \quad \cC \longrightarrow \Vect_\kk^\mathrm{op} \]
    can be recovered from $rY_F^\circ(X)$ as the composition of the functors
    \[ \cC \xrightarrow{~\imath~} D^-(\cC) \simeq K^-\big( \Proj(\cC) \big) \xrightarrow{~ r \tilde Y_F(X) ~} K^-( \Vect_\kk^\mathrm{op} ) \longrightarrow D^-( \Vect_\kk^\mathrm{op} ) \xrightarrow{H^0(-)} \Vect_\kk^\mathrm{op} , \]
    where 
    \begin{itemize}
        \item $\imath$ denotes the canonical embedding of $\cC$ into the bounded above derived category $D^-(\cC)$,
        \item the equivalence $D^-(\cC) \simeq K^-\big( \Proj(\cC) \big)$ is the canonical one (see \cite[Theorem 10.4.8]{WeibelHomAlg}),
        \item the functor $r \tilde Y_F(X)$ is induced by the restriction to $\Proj(\cC)$ of $rY_F(X) \colon \cA \to \Vect_\kk^\mathrm{op}$,
        \item $H^0(-)$ denotes taking cohomology in degree zero.
    \end{itemize}
    Analogously, the functor $\Hom_\cC(-,Y)$ can be recovered from $rY_F(Y)$, and since $\xi$ canonically lifts to a natural transformation between the functors $r \tilde Y_F(X)$ and $r \tilde Y_F(Y)$, we see that $\xi$ also induces a natural transformation
    \[ \hat \xi \colon \Hom_\cC(-,X) \longrightarrow \Hom_\cC( - , Y ) \]
    such that $\hat \xi_P = \xi_P$ for all projective objects $P$ of $\cC$.
    For an arbitrary tilting object $T$ of $\cC$, we can choose an exact sequence $P' \to P \to T \to 0$ with $P$ and $P'$ projective, and using the corresponding exact sequences of $\Hom$-spaces and the universal property of the kernel, we further obtain $\hat \xi_T = \xi_T$.
    Now by the Yoneda lemma, there is a homomorphism $f \colon X \to Y$ such that $\hat \xi = Y_\cC(f)$, and consequently $\xi = rY_F^\circ(f)$.
    We conclude that $rY_F^\circ$ is full, as required.
\end{proof}

\subsection{Upper to lower Ringel duality and exceptional sequences}
\label{subsec:Ringeldualityexceptionalsequences}

Let $\mathcal{D}$ be an upper finite highest weight category with weight datum $(\Lambda,\leq,L')$.
Then by Theorem 4.27 and Corollary 4.30 in \cite{BrundanStroppelSemiInfiniteHighestWeight}, there exists a lower finite highest weight category $\cC = \prescript{\vee}{}{\mathcal{D}}$ with weight datum $(\Lambda,\leq^\mathrm{op},L)$ such that $\mathcal{D} \simeq \cC^\vee$ is the Ringel dual of $\cC$ (in the sense of Subsection~\ref{subsec:tiltingobjectsRingelduality}).
The category $\cC$ is defined explicitly as a category of finite-dimensional comodules over the coend coalgebra of a tilting generator for $\mathcal{D}$ (with possibly infinite standard and costandard filtrations) in \cite[Definition 4.26]{BrundanStroppelSemiInfiniteHighestWeight}.
For the purposes of this article, the explicit choices required for the construction from \cite{BrundanStroppelSemiInfiniteHighestWeight} can somewhat obscure the interactions between Ringel duality and additional structures that may be present on $\mathcal{D}$ (such as a monoidal structure, for instance).
In this subsection, we discuss how $\cC = \prescript{\vee}{}{\mathcal{D}}$ can be constructed from $\mathcal{D}$ without choosing a tilting generator.

The first important observation for our construction is that (for any highest weight category $\cC$), the canonical functor
\begin{equation} \label{eq:tiltingequivalence}
	K^b\big( \Tilt(\cC) \big) \longrightarrow D^b(\cC)
\end{equation}
from the bounded homotopy category of $\Tilt(\cC)$ to the bounded derived category of $\cC$ is an equivalence of categories (cf.\ \cite[Proposition 7.17]{RicheHabilitation}).
If $\mathcal{D} = \cC^\vee$ is the Ringel dual of $\cC$, then we have $\Tilt(\cC) \simeq \Proj^\mathrm{fg}(\mathcal{D})$ by \Cref{thm:Ringelduality}, and so the natural $t$-structure on $D^b(\cC)$ (with heart $\cC$) corresponds to a $t$-structure on $K^b\big( \Proj^\mathrm{fg}(\mathcal{D}) \big)$ (with heart $\cC$) under the equivalence
\[ D^b(\cC) \simeq K^b\big( \Tilt(\cC) \big) \simeq K^b\big( \Proj^\mathrm{fg}(\mathcal{D}) \big) . \]
The second important observation is that the costandard objects of $\cC$ form an \emph{exceptional collection} (see \Cref{def:exceptionalcollection} below) in $D^b(\cC)$, which uniquely determines the natural $t$-structure.
As costandard objects in $\cC$ correspond to standard objects in $\mathcal{D}$ under Ringel duality, the standard objects in $\mathcal{D}$ form an exceptional collection in $K^b\big( \Proj^\mathrm{fg}(\mathcal{D}) \big)$ such that the heart of the corresponding $t$-structure is $\cC$.

In order to make this construction precise, we first recall some results (presumably well known to experts) about exceptional collections in triangulated categories, following \cite{BezrukavnikovQuasiExceptional}.
Let $\cT$ be $\kk$-linear triangulated category with shift functor $X \mapsto X[1]$, and for objects $X$ and $Y$ of $\cT$, consider the $\Z$-graded $\kk$-vector space
\[ \Hom_\cT^\bullet(X,Y) = \bigoplus_{k \in \Z} \Hom_\cT(X,Y[k]) . \]
We say that $\cT$ is of \emph{finite type} if $\Hom_\cT^\bullet(X,Y)$ is finite-dimensional for all objects $X$ and $Y$ of $\cT$.
For a set of objects $\mathcal{X} \subseteq \mathrm{Ob}(\cT)$, we write $\langle \mathcal{X} \rangle$ for the full triangulated subcategory of $\cT$ generated by $\mathcal{X}$ (under shifts and extensions) and $\langle \mathcal{X} \rangle_\mathrm{ext}$ for the full subcategory of $\cT$ generated by $\mathcal{X}$ under extensions.

\begin{Definition}
\label{def:exceptionalcollection}
    An \emph{exceptional collection} in a $\kk$-linear triangulated category $\cT$ is a collection of objects $\{ Y_\lambda \mid \lambda \in \Lambda \}$, indexed by a partially ordered set $(\Lambda,\leq)$, such that for all $\lambda,\mu \in \Lambda$, we have
    \begin{enumerate}
        \item $\Hom_\cT^\bullet(Y_\lambda,Y_\mu) = 0$ unless $\lambda \geq \mu$;
        \item $\Hom_\cT^\bullet( Y_\lambda , Y_\lambda ) \cong \kk$.
    \end{enumerate}
    Given an exceptional collection $\{ Y_\lambda \mid \lambda \in \Lambda \}$ in $\cT$, we write $\cT_{<\lambda} = \langle Y_\mu \mid \mu \in \Lambda , \mu < \lambda \rangle$.
    A collection of objects $\{ X_\lambda \mid \lambda \in \Lambda \}$ of $\cT$ is called a \emph{dual collection} of $\{ Y_\lambda \mid \lambda \in \Lambda \}$ if we have
    \begin{enumerate}
        \item $\Hom_\cT^\bullet( X_\lambda , Y_\mu ) = 0$ if $\lambda>\mu$;
        \item $X_\lambda$ and $Y_\lambda$ are isomorphic in the Verdier quotient $\cT / \cT_{<\lambda}$.
    \end{enumerate}
\end{Definition}

The dual collection of an exceptional collection is uniquely determined up to isomorphism if it exists, by the following result from Lemma 2(e) in \cite{BezrukavnikovQuasiExceptional}.

\begin{Lemma} \label{lem:dualcollectionunique}
	Let $\cT$ be a $\kk$-linear triangulated category and let $\{ Y_\lambda \mid \lambda \in \Lambda \}$ be an exceptional collection in $\cT$, indexed by a poset $(\Lambda,\leq)$.
	If a dual collection $\{ X_\lambda \mid \lambda \in \Lambda \}$ of $\{ Y_\lambda \mid \lambda \in \Lambda \}$ exists, then it is uniquely determined up to isomorphism.
\end{Lemma}

\begin{Remark} \label{rem:dualcollection}
	Let $\cT$ be a $\kk$-linear triangulated category of finite type and let $\{ Y_\lambda \mid \lambda \in \Lambda \}$ be an exceptional collection in $\cT$, indexed by a \emph{lower finite} poset $(\Lambda,\leq)$.
	Then a dual collection $\{ X_\lambda \mid \lambda \in \Lambda \}$ of $\{ Y_\lambda \mid \lambda \in \Lambda \}$ exists and can be constructed via \emph{mutation} of exceptional collections as explained in \cite[Sections 2--3]{BondalAssociativeAlgebrasSheaves}.
	Indeed, for all $\lambda \in \Lambda$, we can fix a total order $\preceq$ on the finite set $\{ \mu \in \Lambda \mid \mu < \lambda \}$ refining the partial order $\leq$, and the objects $\{ Y_\mu \mid \mu \in \Lambda , \mu \leq \lambda \}$ still form an exceptional collection with respect to $\preceq$.
	By Theorem 3.2 in \cite{BondalAssociativeAlgebrasSheaves}, the subcategory $\cT_{<\lambda}$ of $\cT$ is \emph{admissible}, that is, the embedding functor $\cT_{<\lambda} \to \cT$ admits a left adjoint and a right adjoint.
	Let us write $\prescript{\perp}{}{\cT}_{<\lambda}$ for the left-orthogonal complement of $\cT_{<\lambda}$ (i.e.\ the full subcategory of $\cT$ whose objects are the objects $X$ of $\cT$ such that $\Hom_\cT(X,Y) = 0$ for all objects $Y$ of $\cT_{<\lambda}$), and note that the embedding functor $\prescript{\perp}{}{\cT}_{<\lambda} \to \cT$ also admits a right adjoint $R_{<\lambda} \colon \cT \to \prescript{\perp}{}{\cT}_{<\lambda}$ by Lemma 3.1 in \cite{BondalAssociativeAlgebrasSheaves}.
	Now we can define
	\[ X_\lambda = R_{<\lambda} Y_\lambda . \]
	Indeed, the orthogonality property (1) is ovious from the construction.
	Furthermore, the cone of the adjunction counit $R_{<\lambda} X \to X$ belongs to $(\prescript{\perp}{}{\cT}_{<\lambda})^\perp = \cT_{<\lambda}$ (see Lemma 1.7 in \cite{BondalKapranovRepresentableFunctors}), and it follows that the adjunction counit $X_\lambda = j^!_{<\lambda} Y_\lambda \to Y_\lambda$ descends to an isomorphism in $\cT / \cT_{<\lambda}$.
\end{Remark}

Now let us return to the setup of highest weight categories.

\begin{Lemma} \label{lem:exceptionalcollectionlowerfinite}
	Let $\cC$ be lower finite highest weight category with weight datum $(\Lambda,\leq,L)$.
	\begin{enumerate}
		\item The costandard objects $\{ \nabla(\lambda) \mid \lambda \in \Lambda \}$ form an exceptional collection in $\cT = D^b(\cC)$ with respect to $(\Lambda,\leq)$, and the standard objects $\{ \Delta(\lambda) \mid \lambda \in \Lambda \}$ form a dual collection.
		\item The natural $t$-structure $(\cT^{\leq 0},\cT^{\geq 0})$ on $\cT = D^b(\cC)$ with heart $\cC$ can be described in terms of the standard objects and costandard objects via
		\begin{align*}
			\cT^{\geq 0} & = \big\langle \nabla(\lambda)[d] \mathrel{\big|} \lambda \in \Lambda , d \leq 0 \big\rangle_\mathrm{ext} , \\
			\cT^{\leq 0} & = \big\langle \Delta(\lambda)[d] \mathrel{\big|} \lambda \in \Lambda , d \geq 0 \big\rangle_\mathrm{ext} .
		\end{align*}
	\end{enumerate}
\end{Lemma}
\begin{proof}
	The costandard objects $\{ \nabla(\lambda) \mid \lambda \in \Lambda \}$ form an exceptional collection in $\cT$ by \Cref{lem:Extvanishingexceptional}.
	Furthermore, we have $\Hom_\cT^\bullet\big( \Delta(\lambda) , \nabla(\mu) \big) = 0$ for $\lambda,\mu \in \Lambda$ with $\lambda \neq \mu$ by \Cref{lem:Extvanishingstandardcostandard}, and using the lower finiteness of $(\Lambda,\leq)$ it is straightforward to see that
	\[ \cT_{< \lambda} = \langle \nabla(\mu) \mid \mu \in \Lambda, \mu < \Lambda \rangle = \langle L(\mu) \mid \mu \in \Lambda, \mu < \Lambda \rangle , \]
	hence $\Delta(\lambda) \cong L(\lambda) \cong \nabla(\lambda)$ in the Verdier quotient $\cT / \cT_{<\lambda}$.
	This shows that the standard objects form a dual collection of the costandard objects.
	
	In order to prove (2), recall that $\cT^{\geq 0}$ is the full subcategory of $\cT$ whose objects are the bounded complexes $X$ in $\cC$ such that $H^i(X) = 0$ for all $i<0$.
	The shifted complexes $\nabla(\lambda)[d]$ obviously satisfy this property for $\lambda \in \Lambda$ and $d \leq 0$, and so the full subcategory $\langle \nabla(\lambda)[d] \mathrel{\big|} \lambda \in \Lambda , d \leq 0 \rangle_\mathrm{ext}$ is contained in $\cT^{\geq 0}$.
	For the reverse inclusion, it suffices to show that every object $M$ of $\cC$ admits a finite resolution $0 \to M \to M_0 \to \cdots \to M_r \to 0$ such that $M_i$ has a costandard filtration for $i = 0 , \ldots , r$.
	The existence of such a resolution is a well-known property of lower finite highest weight categories; see for instance Definition 2.2 and Lemma 2.9 in \cite{ParkerGFD}.
	The claim about $\cT^{\leq 0}$ can be proven analogously. 
\end{proof}

In order to check whether an object of $\cT = D^b(\cC)$ belongs to the co-aisle $\cT^{\geq 0}$ of the $t$-structure defined in Proposition \ref{lem:exceptionalcollectionlowerfinite}, we can use the following result, whose proof is completely analogous to \cite[Proposition 3.2]{GruberMinimalTilting}.

\begin{Proposition}
\label{prop:costandardgeneratedegrees}
    Let $\cC$ be a lower finite highest weight category with weight poset $(\Lambda,\leq)$, and let $X$ be a complex in $\cC$.
    Then there is sequences $n_0,\ldots,n_r$ of integers, $\lambda_0,\ldots,\lambda_r \in \Lambda$ of weights and $X_0,\ldots,X_{r+1}$ of complexes in $\cC$, such that $X_{r+1} \cong 0$ in $D^b(\cC)$, there are distinguished triangles
    \[ \nabla(\lambda_i)[n_i] \longrightarrow X_i \longrightarrow X_{i+1} \longrightarrow \nabla(\lambda_i)[i+1] \]
    in $D^b(\cC)$ for $i=0,\ldots,r$, and we have
    \[ \sum_{i=0}^r \delta_{\lambda,\lambda_i} \cdot v^{n_i} = \sum_{i \in \Z} \dim \Hom_{D^b(\cC)}\big( \Delta(\lambda) , X[i] \big) \cdot v^i . \]
    In particular, the complex $X$ belongs to the subcategory $\big\langle \nabla(\lambda)[d] \mathop{\big|} \lambda \in \Lambda , d \leq 0 \big\rangle_\mathrm{ext}$ of $D^b(\cC)$ if and only if $\Hom_{D^b(\cC)}\big( \Delta(\lambda) , X[i] \big) = 0$ for all $\lambda \in \Lambda$ and $i>0$.
\end{Proposition}

\begin{Definition}
\label{def:orderpreservingduality}
    Let $\cA$ be a $\kk$-linear Krull--Schmidt category with a poset $(\Lambda,\leq)$ and a map
    $\Lambda \to \mathrm{Ob}(\cA) , ~ \lambda \mapsto X_\lambda$
    indexing the isomorphism classes of indecomposable objects in $\cA$.
    An \emph{order preserving duality} on $\cA$ is a $\kk$-linear equivalence $\tau \colon \cA \to \cA^\mathrm{op}$ with an order automorphism $\sigma \colon \Lambda \to \Lambda$ such that $\tau (X_\lambda) \cong X_{\sigma(\lambda)}$ for all $\lambda \in \Lambda$.
\end{Definition}

If $\mathcal{D}$ is an upper finite highest weight category and $\cA = \Proj(\mathcal{D})$, or if $\cC$ is a lower finite highest weigh category and $\cA = \Tilt(\cC)$, then clearly the isomorphism classes of indecomposable objects in $\cA$ are canonically indexed by the weight poset.
In that case, we will refer to order-preserving dualities without explicit reference to the poset.

\begin{Theorem}
\label{thm:Ringeldualityexceptioncollection}
	Let $\mathcal{D}$ be an upper finite highest weight category with weight poset $(\Lambda,\leq)$.
    Further let $\cA = \Proj^\mathrm{fg}(\mathcal{D})$ and consider the triangulated category $\cT = K^b(\cA)$.
	\begin{enumerate}
		\item For all $\lambda \in \Lambda$, fix a finite projective resolution
	\[ \cdots \to P_{-2}^\lambda \to P_{-1}^\lambda \to P_0^\lambda \to \Delta_\mathcal{D}(\lambda) \to 0 . \]
	Then the complexes
	\[ Y_\lambda \coloneqq \big( \cdots \to P_{-2}^\lambda \to P_{-1}^\lambda \to P_0^\lambda \to 0 \big) \]
	form a full exceptional collection in $\cT$ with respect to the opposite poset $(\Lambda,\leq^\mathrm{op})$.
		\item The exceptional collection $\{ Y_\lambda \mid \lambda \in \Lambda \}$ admits a dual collection $\{ X_\lambda \mid \lambda \in \Lambda \}$.
		\item There is a $t$-structure $(\cT^{\leq 0},\cT^{\geq 0})$ on $\cT$ such that
		\begin{align*}
			\cT^{\geq 0} & = \big\langle Y_\lambda[d] \mathrel{\big|} \lambda \in \Lambda , d \leq 0 \big\rangle_\mathrm{ext} , \\
			\cT^{\leq 0} & = \big\langle X_\lambda[d] \mathrel{\big|} \lambda \in \Lambda , d \geq 0 \big\rangle_\mathrm{ext} .
		\end{align*}
		\item The heart $\cC = \cT^{\leq 0} \cap \cT^{\geq 0}$ of the $t$-structure $(\cT^{\leq 0},\cT^{\geq 0})$ is a lower finite highest weight category with weight poset $(\Lambda,\leq^\mathrm{op})$.
        We have $\Tilt(\cC) = \cA = \Proj^\mathrm{fg}(\mathcal{D})$, the subcategory of $\cT = K^b(\cA)$ of complexes concentrated in degree zero.
        Furthermore, the standard objects, costandard objects and tilting objects in $\cC$ are given by
        \[ \Delta_\cC(\lambda) = X_\lambda , \qquad \nabla_\cC(\lambda) = Y_\lambda , \qquad T_\cC(\lambda) = P_\mathcal{D}(\lambda) \]
        for $\lambda \in \Lambda$.
        \item Let $\cC^\vee$ be the Ringel dual of $\cC$ and let $R = R_\cC \colon \cC \to \cC^\vee$ be the Ringel duality functor.
        Then there is an equivalence $\psi \colon \cC^\vee \xrightarrow{~\sim~} \mathcal{D}$ such that $\psi L^\vee(\lambda) \cong L_\mathcal{D}(\lambda)$ for all $\lambda \in \Lambda$ and such that there is a natural isomorphism
        \[ H^0(-) \circ \imath \cong \psi \circ R , \]
        where  $\imath \colon \cC \to \cT$ is the canonical embedding and $H^0(-) \colon \cT = K^b\big( \Proj^\mathrm{fg}(\mathcal{D}) \big) \to \mathcal{D}$ is the degree zero cohomology functor.
        \item If there is an order-preserving duality $\tau \colon \cA \to \cA^\mathrm{op}$, then the induced equivalence on $\cT \simeq \cT^\mathrm{op}$ restricts to an equivalence $\cC \simeq \cC^\mathrm{op}$, which in turn restricts to the duality $\tau$ on $\Tilt(\cC) = \cA$.
	\end{enumerate}
\end{Theorem}
\begin{proof}
	Part (1) of the proposition follows from \Cref{lem:Extvanishingexceptional} and the fact that the canonical functor from $\cT = K^b\big( \Proj(\mathcal{D}) \big)$ to $D^b(\mathcal{D})$ is fully faithful, cf.\ \cite[Corollary 10.4.7]{WeibelHomAlg}.
    Part (2) follows from \Cref{rem:dualcollection}.
	In order to prove parts (3) and (4), recall that by Corollary 4.30 in \cite{BrundanStroppelSemiInfiniteHighestWeight}, there is a lower finite highest weight category $\cC_0$ with weight poset $(\Lambda,\leq^\mathrm{op})$ such that $\mathcal{D} \simeq \cC_0^\vee$ is equivalent to the Ringel dual of $\cC_0$.
	The opposite Ringel duality functor $L = L_{\cC_0} \colon \mathcal{D}^\mathrm{cc} \to \cC_0^\mathrm{cc}$ restricts to an equivalence $\Proj^\mathrm{fg}(\mathcal{D}) \simeq \Tilt(\cC_0)$ and therefore induces an exact equivalence
	\[ \mathfrak{L} \colon \cT = K^b\big( \Proj^\mathrm{fg}(\mathcal{D}) \big) \longrightarrow K^b\big( \Tilt(\cC_0) \big) \simeq D^b(\cC_0) \]
	via the equivalence in \eqref{eq:tiltingequivalence}.
	As $L$ sends exact sequences in $\mathcal{D}_\Delta$ to exact sequences in $(\cC_0)_\nabla$, a projective resolution of $\Delta_\mathcal{D}(\lambda)$ in $\mathcal{D}$ is sent by $L$ to a resolution of $\nabla_{\cC_0}(\lambda) \cong L\Delta_\mathcal{D}(\lambda)$ by tilting objects in $\cC_0$ for all $\lambda \in \Lambda$, and it follows that
    \[ \mathfrak{L} Y_\lambda \cong \nabla_{\cC_0}(\lambda) \]
    in $D^b(\cC_0)$.
	Now \Cref{lem:exceptionalcollectionlowerfinite} implies that the costandard objects $\{ \nabla_{\cC_0}(\lambda) \mid \lambda \in \Lambda \}$ form an exceptional collection in $D^b(\cC_0)$ with respect to $(\Lambda,\leq^\mathrm{op})$, with dual collection given by $\{ \Delta_{\cC_0}(\lambda) \mid \lambda \in \Lambda \}$, and by uniqueness of the dual collection (see \Cref{lem:dualcollectionunique}) it follows that
    \[ \mathfrak{L} X_\lambda \cong \Delta_{\cC_0}(\lambda) \]
    for all $\lambda \in \Lambda$.
	Again by \Cref{lem:exceptionalcollectionlowerfinite}, the natural $t$-structure on $D^b(\cC_0)$ with heart $\cC_0$ is given by
	\begin{align*}
			D^b(\cC_0)^{\geq 0} & = \big\langle \nabla_{\cC_0}(\lambda)[d] \mathrel{\big|} \lambda \in \Lambda , d \leq 0 \big\rangle_\mathrm{ext} = \big\langle \mathfrak{L} Y_\lambda [d] \mathrel{\big|} \lambda \in \Lambda , d \leq 0 \big\rangle_\mathrm{ext} , \\
			D^b(\cC_0)^{\leq 0} & = \big\langle \Delta_{\cC_0}(\lambda)[d] \mathrel{\big|} \lambda \in \Lambda , d \geq 0 \big\rangle_\mathrm{ext} = \big\langle \mathfrak{L} X_\lambda [d] \mathrel{\big|} \lambda \in \Lambda , d \geq 0 \big\rangle_\mathrm{ext} ,
	\end{align*}
	and the claims (3)--(4) follow because $\mathfrak{L} \colon \cT \to D^b(\cC_0)$ is an exact equivalence.

    In order to prove (5), note that by \Cref{rem:tiltingequivalentprojectiveimpliesRingeldual}, the equality $\Tilt(\cC) = \cA = \Proj^\mathrm{fg}(\cC)$ gives rise to an equivalence $\psi \colon \cC^\vee \to \mathcal{D}$ such that $\psi L^\vee(\lambda) \cong L_\mathcal{D}(\lambda)$ for all $\lambda \in \Lambda$ and there is a natural isomorphism
    \[ \psi \circ R\vert_\cA \cong \mathrm{id}_\cA \]
    of functors from $\cA = \Tilt(\cC)$ to $\cA = \Proj^\mathrm{fg}(\mathcal{D})$.
    In order to show that $\psi \circ R \cong H^0(-) \circ \imath$, observe that the subcategory $\cC_\nabla$ of $\cC$ is injective with respect to $\mathcal{R} \colon \cC \to \cC^\vee$ in the sense of \cite[Definition 13.3.4]{KashiwaraSchapiraCategoriesSheaves}, by \cite[Corollary 13.3.8]{KashiwaraSchapiraCategoriesSheaves} together with \cite[Corollary 3.58]{BrundanStroppelSemiInfiniteHighestWeight} and \Cref{thm:Ringelduality}.
    By \cite[Proposition 13.3.5]{KashiwaraSchapiraCategoriesSheaves}, this implies that the composition
    \[ \mathbb{R}R \colon D^b( \cC ) \xrightarrow{ ~ \eqref{eq:tiltingequivalence} ~ } K^b(\cA) \xrightarrow{~\hat R~} K^b\big( \Proj^\mathrm{fg}(\cC^\vee) \big) \to D^b( \cC^\vee )  \]
    is the total derived functor of $R$, where $\hat R$ is induced by the restriction of $R$ to $\cA = \Tilt(\cC)$.
    In particular, we have $R \cong H^0(-) \circ \hat R \circ \imath$ (the zeroth right derived functor of $R$).
    As $\psi \circ R\vert_\cA \cong \id_\cA$, this implies that $\psi \circ R \cong H^0(-) \circ \imath$, as claimed.
    
    Finally, suppose that $\cA = \Proj^\mathrm{fg}(\mathcal{D})$ admits an order-preserving duality $\tau$, and let $\sigma \colon \Lambda \to \Lambda$ be the order automorphism with $\tau P_\mathcal{D}(\lambda) \cong P_\mathcal{D}(\sigma\lambda)$.
    Then $\tau$ lifts to a duality $\tau \colon K^b(\cA) \to K^b(\cA)^\mathrm{op}$ with
    \[ \tau\big( \cdots \to P_i \to P_{i+1} \to \cdots \big) = \big( \cdots \to \tau(P_{i+1}) \to \tau(P_i) \to \cdots  \big) , \]
    where $\tau(P_i)$ is in homological degree $-i$, and we claim that $X_{\sigma\lambda} \cong \tau Y_\lambda$ for all $\lambda \in \Lambda$.
    
    In order to prove the claim, note that by \Cref{lem:projectiveresolutionDelta}, we may assume that $P_0^\lambda \cong P_\mathcal{D}(\lambda)$ and $P_i^\lambda$ is a direct sum of certain $P_\mathcal{D}(\mu)$ with $\mu > \lambda$ for $i<0$, possibly after replacing $Y_\lambda$ by a homotopy equivalent complex in $\cT = K^b(\cA)$.
    This implies that we have
    \[ \cT_{<^\mathrm{op}\lambda} = \langle Y_\mu \mid \mu <^\mathrm{op} \lambda \rangle = \langle P_\mathcal{D}(\mu) \mid \mu > \lambda \rangle , \]
    and therefore $Y_\lambda \cong P_\mathcal{D}(\lambda)$ in $\cT / \cT_{<^\mathrm{op}\lambda}$ and
    \[ \tau Y_\lambda \cong \tau P_\mathcal{D}(\lambda) \cong P_\mathcal{D}(\sigma\lambda) \cong Y_{\sigma\lambda} \qquad \text{in} \quad \cT / \cT_{<^\mathrm{op}\sigma\lambda} . \]
    Furthermore, we have
    \[ \Hom_\cT^\bullet( \tau Y_\lambda , Y_{\sigma \mu} ) \cong \Hom_\cT^\bullet( \tau^{-1} Y_{\sigma\mu} , Y_\lambda ) = 0 \qquad \text{for} \quad \lambda>^\mathrm{op}\mu \]
    because $\tau^{-1} Y_{\sigma\mu}$ belongs to $\cT_{<^\mathrm{op}\lambda}$ and because $\{ Y_\nu \mid \nu \in \Lambda \}$ is an exceptional collection in $\cT$ with respect to $(\Lambda,\leq^\mathrm{op})$.
    Now the uniqueness of the dual collection (\Cref{lem:dualcollectionunique}) implies that
    \[ X_{\sigma\lambda} \cong \tau Y_\lambda \]
    for all $\lambda \in \Lambda$, as claimed.
    As the complexes $\{ \tau Y_\lambda \mid \lambda \in \Lambda \} = \{ X_{\sigma\lambda} \mid \lambda \in \Lambda \}$ form an exceptional collection in $\cT^\mathrm{op}$ with two dual collections $\{ \tau X_\lambda \mid \lambda \in \Lambda \}$ and $\{ Y_{\sigma \lambda} \mid \lambda \in \Lambda \}$, we further conclude that
    \[ Y_{\sigma\lambda} \cong \tau X_\lambda  \]
    for all $\lambda \in \Lambda$.
    In particular, we have $\tau \cT^{\geq 0} = \cT^{\leq 0}$ and $\tau \cT^{\leq 0} = \cT^{\geq 0}$, hence $\tau$ restricts to a duality on $\cC = \cT^{\geq 0} \cap \cT^{\leq 0}$, as required.
\end{proof}

\section{Monoidal Ringel duality}
\label{sec:monoidalRingelduality}

In this section, we explain how Brundan--Stroppel's semi-infinite Ringel duality interacts with monoidal structures on highest weight categories.
We first explain how a monoidal structure on a lower finite highest weight category induces a monoidal structure on the Ringel dual upper finite highest weight category (Subsection~\ref{subsec:monoidalRingelduality_lowertoupper}) and then discuss how a monoidal structure on an upper finite highest weight category gives rise to a monoidal structure on the Ringel dual lower finite highest weight category (Subsection~\ref{subsec:monoidalRingelduality_uppertolower}).
In Subsection~\ref{subsec:abelian-envelopes}, we discuss monoidal abelian envelopes in the context of monoidal highest weight categories and prove that a lower finite highest weight category with a compatible monoidal structure is often a monoidal abelian envelope of its full subcategory of tilting objects.

\subsection{Lower to upper monoidal Ringel duality}
\label{subsec:monoidalRingelduality_lowertoupper}

The first main result of this section is the following theorem, which explains how a monoidal structure on a lower finite highest weight category gives rise to a monoidal structure on its Ringel dual.

\begin{Theorem}[Lower to upper monoidal Ringel duality]
\label{thm:monoidalRingelduality_lowertoupper}
	Let $\cC$ be a lower finite highest weight category with a monoidal structure $(\otimes,\mathbf{1},\ldots)$ such that $\cA = \Tilt(\cC)$ is a monoidal subcategory and $\otimes$ is right exact in both arguments.
	Then the following assertions hold.
	\begin{enumerate}
		\item The Ringel dual $\cC^{\vee,\mathrm{cc}}$ admits a unique (up to monoidal equivalence) monoidal structure $(\otimes^\vee,\mathbf{1}^\vee,\ldots)$ such that the embedding $R\vert_\cA \colon \cA \to \cC^{\vee,\mathrm{cc}}$ can be enhanced to a monoidal functor and $\otimes^\vee$ is cocontinuous in both arguments.
		\item The finitely generated objects $\cC^{\vee,\mathrm{fg}}$ form a monoidal subcategory of $\cC^{\vee,\mathrm{cc}}$, and $\Proj^\mathrm{fg}( \cC^\vee )$ is a monoidal subcategory of $\cC^{\vee,\mathrm{fg}}$.
        \item If $\cC$ is braided (or symmetric), then $\cC^{\vee,\mathrm{cc}}$ admits a canonical (symmetric) braiding, making the embedding $R\vert_\cA \colon \cA \to \cC^{\vee,\mathrm{cc}}$ a braided monoidal functor.
        \item If $\otimes$ is right exact in both arguments, then the opposite Ringel duality functor
        \[ L = L_\cC \colon \cC^{\vee,\mathrm{cc}} \longrightarrow \cC^\mathrm{cc} \]
        can be enhanced (uniquely up to monoidal natural isomorphism) to a monoidal functor.
        If $\cC$ is braided, then $L$ is a braided monoidal functor.
		\item All rigid objects in $\cC^{\vee,\mathrm{cc}}$ are projective.
		If $\cA$ is rigid, then $\Proj(\cC^\vee)$ is the subcategory of rigid objects in $\cC^{\vee,\mathrm{fg}}$.
	\end{enumerate}
\end{Theorem}
\begin{proof}
    By \Cref{prop:Ringeldualpresheaves}, we can identify the Ringel dual $\cC^{\vee,\mathrm{cc}}$ with the presheaf category $\PSh(\cA)$.
    Then the embedding $R\vert_\cA \colon \cA \to \cC^{\vee,\mathrm{cc}}$ identifies with the Yoneda embedding $Y_\cA \colon \cA \to \PSh(\cA)$ and the opposite Ringel duality functor $L \colon \cC^{\vee,\mathrm{cc}} \to \cC^\mathrm{cc}$ identifies with the Yoneda extension $\hat F \colon \PSh(\cA) \to \cC^\mathrm{cc}$ of the embedding $F \colon \cA \to \cC^\mathrm{cc}$.
    As $\cA = \Tilt(\cC)$ is closed under tensor products in $\cC$, \Cref{rem:freemonoidalcocompletion} implies that Day convolution gives rise to the unique (up to monoidal equivalence) monoidal structure $(\otimes^\vee,\mathbf{1}^\vee,\ldots)$ on $\cC^{\vee,\mathrm{cc}} = \PSh(\cA)$ such that the embedding $R\vert_\cA = Y_\cA \colon \cA \to \cC^{\vee,\mathrm{cc}}$ can be enhanced to a monoidal functor and $\otimes^\vee$ is cocontinuous in both arguments.
    
    For the second claim, first note that $\Proj^\mathrm{fg}(\cC^\vee)$ is closed under tensor products in $\cC^{\vee,\mathrm{cc}}$ because the functor $R\vert_\cA \colon \cA \to \cC^{\vee,\mathrm{cc}}$ is monoidal and restricts to an equivalence $\Tilt(\cC) \simeq \Proj^\mathrm{fg}(\cC^\vee)$.
    Since every finitely generated object is a quotient of a finitely generated projective object and since the tensor product $\otimes^\vee$ is cocontinuous (and therefore right exact) in both arguments, it follows that $\cC^{\vee,\mathrm{fg}}$ is closed under tensor products in $\cC^{\vee,\mathrm{cc}}$.
    
    If $\cC$ is braided (or symmetric), then so is $\Tilt(\cC)$, and by \Cref{rem:Dayconvolutionbraided}, there is a canonical (symmetric) braiding on $\cC^{\vee,\mathrm{cc}}$ making $\mathcal{R}\vert_\cA$ a braided monoidal functor.
    
    Next, if $\otimes$ is right exact in both arguments, then the monoidal structure on $\cC$ extends to a (unique up to monoidal equivalence) monoidal structure on $\cC^\mathrm{cc}$ such that the tensor product is cocontinuous in both arguments (see Subsection~\ref{subsec:indcompletion}), making $\cC^\mathrm{cc}$ a $\kk$-linear monoidal and monoidally cocomplete category.
    Again by \Cref{rem:freemonoidalcocompletion}, the opposite Ringel duality functor $L = \hat F \colon \cC^{\vee,\mathrm{cc}} \to \cC^\mathrm{cc}$ can be enhanced (uniquely up to monoidal natural isomorphism) to a monoidal functor, and the latter is braided if $\cC$ is braided by \Cref{rem:freemonoidalcocompletion}.
    
    In order to prove the final claim, first observe that the unit object $\mathbf{1}^\vee \cong R(\mathbf{1})$ is projective in $\cC^\vee$, because $\mathbf{1}$ is a tilting object in $\cC$ by assumption.
    It is straightforward to see that every tensor product of a projective object by a rigid object is projective, and so every rigid object $x \cong \mathbf{1}^\vee \otimes^\vee x$ of $\cC^\vee$ is projective.
    If $\cA$ is rigid, then so is $\Proj^\mathrm{fg}(\cC^\vee)$, because the embedding $R \vert_\cA$ restricts to a monoidal equivalence $\cA \simeq \Proj^\mathrm{fg}(\cC^\vee)$, and we conclude that the rigid objects in $\cC^{\vee,\mathrm{fg}}$ are precisely the projective objects in $\cC^\vee$, as claimed.
\end{proof}

For later use, we note that for an upper finite highest weight category $\mathcal{D}$, a monoidal structure on $\mathcal{D}^\mathrm{cc}$ with the properties of the monoidal structure on $\cC^{\vee,\mathrm{cc}}$ in \Cref{thm:monoidalRingelduality_lowertoupper} also gives rise to a monoidal structure on $\mathcal{D}^\mathrm{op,cc}$.

\begin{Remark}
\label{rem:monoidalstructureupperfiniteopposite}
    Let $\mathcal{D}$ be an upper finite highest weight category with a monoidal structure $(\otimes,\mathbf{1},\ldots)$ on $\mathcal{D}^\mathrm{cc}$ such that $\otimes$ is cocontinuous in both arguments and $\cA = \Proj^\mathrm{fg}(\mathcal{D})$ is a monoidal subcategory.
    The Yoneda extension of the monoidal embedding $\cA \to \mathcal{D}^\mathrm{cc}$ gives rise to a monoidal equivalence $\mathcal{D}^\mathrm{cc} \to \PSh(\cA)$ by Remarks~\ref{rem:presheavesonprojectives} and \ref{rem:freemonoidalcocompletion}, where the monoidal structure on $\PSh(\cA)$ is given by Day convolution. 
    Recall from \Cref{rem:oppositeupperfinite} that
    \[ \mathcal{D}^\mathrm{op} \simeq \PSh^\mathrm{fd}(\cA^\mathrm{op}) , \qquad \mathcal{D}^\mathrm{op,cc} \simeq \PSh(\cA^\mathrm{op}) , \qquad \Proj^\mathrm{fg}(\mathcal{D}^\mathrm{op}) \simeq \cA^\mathrm{op} . \]
    As the monoidal structure on $\cA$ also defines a monoidal structure on $\cA^\mathrm{op}$, Day concolution gives rise to a unique (up to monoidal equivalence) monoidal structure $(\otimes',\mathbf{1},\ldots)$ on $\mathcal{D}^\mathrm{op,cc}$ such that $\otimes'$ is cocontinuous in both arguments, $\Proj^\mathrm{fg}(\mathcal{D}^\mathrm{op})$ is a monoidal subcategory and $\Proj^\mathrm{fg}(\mathcal{D}^\mathrm{op})$ is monoidally equivalent to $\cA^\mathrm{op}$ (see \Cref{rem:freemonoidalcocompletion}).
    
    Analogously, there is a unique (up to monoidal equivalence) monoidal structure $(\otimes'',\mathbf{1},\ldots)$ on $\mathcal{D}^\mathrm{cc}$ such that $\otimes''$ is cocontinuous in both arguments, $\Proj^\mathrm{fg}(\mathcal{D})$ is a monoidal subcategory and $\Proj^\mathrm{fg}(\mathcal{D})$ is monoidally equivalent to $\cA^\mathrm{rev}$.
    Uniqueness of the monoidal structure implies that $(\mathcal{D}^\mathrm{cc},\otimes'',\mathbf{1},\ldots)$ is monoidally equivalent to $\mathcal{D}^\mathrm{cc,rev}$ (that is, to $(\mathcal{D}^\mathrm{cc},\otimes^\mathrm{rev},\mathbf{1},\ldots)$).
\end{Remark}

In order to prove a converse of \Cref{thm:monoidalRingelduality_lowertoupper}, we will consider monoidal upper finite highest weight categories that satisfy certain additional conditions \eqref{eq:Xtensor} or \eqref{eq:Ytensor}.
The latter are partly motivated by the following observation.

\begin{Proposition}
\label{prop:Ringeldualtensorcondition}
    Let $\cC$ be a lower finite highest weight category with weight poset $(\Lambda,\leq)$ and with a monoidal structure $(\otimes,\mathbf{1},\ldots)$ such that $\cA = \Tilt(\cC)$ is a monoidal subcategory and $\otimes$ is exact in both arguments.
    Consider the Ringel dual $\cC^\vee$ with the monoidal structure $( \otimes^\vee , \mathbf{1}^\vee , \ldots )$ on $\cC^{\vee,\mathrm{cc}}$ from \Cref{thm:monoidalRingelduality_lowertoupper}, and note that $\otimes^\vee$ gives rise to a monoidal structure $( \otimes^\vee_\mathrm{c} , \mathbf{1}^\vee_\mathrm{c} , \ldots )$ on the homotopy category $\cT = K^b\big( \Proj^\mathrm{fg}(\cC^\vee) \big)$ via the tensor product of complexes from Subsection~\ref{subsec:tensorproductofcomplexes}.
    Further consider the exceptional collection $\{ Y_\lambda \mid \lambda \in \Lambda \}$ in $\cT$ with dual collection $\{ X_\lambda \mid \lambda \in \Lambda \}$ as in \Cref{thm:Ringeldualityexceptioncollection}.
    Then the subcategories
    \begin{align*}
			\mathcal{Y} & = \big\langle Y_\lambda[d] \mathrel{\big|} \lambda \in \Lambda , d \leq 0 \big\rangle_\mathrm{ext} , \\
			\mathcal{X} & = \big\langle X_\lambda[d] \mathrel{\big|} \lambda \in \Lambda , d \geq 0 \big\rangle_\mathrm{ext}
	\end{align*}
    are monoidal subcateories of $\cT$ with respect to $(\otimes^\vee_\mathrm{c},\mathbf{1}^\vee_\mathrm{c},\ldots)$.
\end{Proposition}

\begin{proof}
    By \cite[Theorem~4.8]{BrundanStroppelSemiInfiniteHighestWeight}, an object of $\cC$ admits a $\nabla$-filtration (or $\Delta$-filtration) if and only if it admits a finite resolution (respectively coresolution) by tilting objects.
    For all $\lambda \in \Lambda$, we fix such a finite resolution (or coresolution)
    \[ \cdots \to T_1^\lambda \to T_0^\lambda \to \nabla(\lambda) \to 0 , \hspace{2cm} 0 \to \Delta(\lambda) \to Q_0^\lambda \to Q_1^\lambda \to \cdots \]
    and consider the complexes
    \[ X_\lambda' = ( 0 \to Q_0^\lambda \to Q_1^\lambda \to \cdots ) , \hspace{2cm} Y_\lambda' = ( \cdots \to T_1^\lambda \to T_0^\lambda \to 0 ) \]
    in $\Tilt(\cC)$ (with $Q_0^\lambda$ and $T_0^\lambda$ in homological degree zero).
    The monoidal equivalence
    \[ \Tilt(\cC) \simeq \Proj^\mathrm{fg}(\cC^\vee) \]
    from \Cref{thm:monoidalRingelduality_lowertoupper} gives rise to a monoidal equivalence of triangulated categories
    \[ \mathcal{L} \colon K^b\big( \Proj^\mathrm{fg}(\cC^\vee) \big) \xrightarrow{~\sim~} K^b\big( \Tilt(\cC) \big) , \]
    and by \Cref{thm:Ringeldualityexceptioncollection} and its proof, we have $\mathcal{L} Y_\lambda \cong Y_\lambda'$ and $\mathcal{L} X_\lambda \cong X_\lambda'$ for all $\lambda \in \Lambda$.
    Further observe that since the tensor product $\otimes$ on $\cC$ is exact in both arguments, the tensor product of complexes gives rise to a monoidal structure $(\otimes_\mathrm{c},\mathbf{1}_\mathrm{c},\ldots)$ on the derived category $D^b(\cC)$ such that $\otimes_\mathrm{c}$ is exact in both arguments.
    The equivalence from \eqref{eq:tiltingequivalence} becomes a monoidal equivalence of triangulated categories
    \[ \mathcal{E} \colon K^b\big( \Tilt(\cC) \big) \xrightarrow{~\sim~} D^b(\cC) \]
    with $\mathcal{E} X_\lambda' \cong \Delta(\lambda)_\mathrm{c}$ and $\mathcal{E} Y_\lambda' \cong \nabla(\lambda)_\mathrm{c}$ for all $\lambda \in \Lambda$ (where $(-)_\mathrm{c}$ denotes complexes concentrated in homological degree zero, as in Subsection~\ref{subsec:tensorproductofcomplexes}).
    Therefore, it suffices to prove that the subcategories
    \begin{align*}
		\mathcal{X}' & = \big\langle \Delta(\lambda)[d] \mathrel{\big|} \lambda \in \Lambda , d \geq 0 \big\rangle_\mathrm{ext} , \\
        \mathcal{Y}' & = \big\langle \nabla(\lambda)[d] \mathrel{\big|} \lambda \in \Lambda , d \leq 0 \big\rangle_\mathrm{ext}
	\end{align*}
    are monoidal subcategories in $D^b\big( \cC)$, with respect to the monoidal structure $(\otimes_\mathrm{c},\mathbf{1}_\mathrm{c},\ldots)$ afforded by the tensor product of complexes.
    Since the tensor product $\otimes_\mathrm{c}$ is exact in both arguments and commutes with the shift functor, it further suffices to show that $\mathcal{X}'$ contains $\mathbf{1}_\mathrm{c}$ and all tensor products of the form $\Delta(\lambda)_\mathrm{c} \otimes_\mathrm{c} \Delta(\mu)_\mathrm{c}$ for $\lambda,\mu \in \Lambda$, and similarly for $\mathcal{Y}'$.

    Since $\mathbf{1}$ is a tilting object in $\cC$ by assumption, it is clear that $\mathbf{1}_c$ belongs to $\mathcal{X}'$ and $\mathcal{Y}'$.
    For $\lambda,\mu \in \Lambda$, there are isomorphisms
    \[ \big( \Delta(\lambda) \otimes \Delta(\mu) \big)_\mathrm{c} \cong \Delta(\lambda)_\mathrm{c} \otimes_\mathrm{c} \Delta(\mu)_\mathrm{c} \cong X_\lambda' \otimes_\mathrm{c} X_\mu' \]
    in $D^b(\cC)$, and this implies that the complex $X_\lambda' \otimes_\mathrm{c} X_\mu'$ provides a finite coresolution of $\Delta(\lambda) \otimes \Delta(\mu)$ by tilting objects in $\cC$.
    Again using \cite[Theorem 4.8]{BrundanStroppelSemiInfiniteHighestWeight}, it follows that $\Delta(\lambda) \otimes \Delta(\mu)$ admits a $\Delta$-filtration, and so $\Delta(\lambda)_\mathrm{c} \otimes_\mathrm{c} \Delta(\mu)_\mathrm{c}$ belongs to $\mathcal{X}'$.
    Analogously, the tensor product $\nabla(\lambda)_\mathrm{c} \otimes_\mathrm{c} \nabla(\mu)_\mathrm{c}$ belongs to $\mathcal{Y}'$, as required.
\end{proof}

\subsection{Upper to lower monoidal Ringel duality}
\label{subsec:monoidalRingelduality_uppertolower}

In this subsection, we explain how a monoidal structure on an upper finite highest weight category $\mathcal{D}$ gives rise to a monoidal structure on the Ringel dual lower finite highest weight category $\prescript{\vee}{}{\mathcal{D}}$ (see \Cref{thm:monoidalRingelduality_uppertolower}).
To be more precise, we will consider monoidal structures on the cocomplete category $\mathcal{D}^\mathrm{cc}$ instead of $\mathcal{D}$ (this is motivated and justified by \Cref{thm:monoidalRingelduality_lowertoupper}), and we will impose one or both of the conditions \eqref{eq:Xtensor} and \eqref{eq:Ytensor} defined below.

\begin{Definition}
\label{def:XtensorYtensor}
    Let $\mathcal{D}$ be an upper finite highest weight category with weight poset $(\Lambda,\leq)$ and with a monoidal structure $(\otimes,\mathbf{1},\ldots)$ on $\mathcal{D}^\mathrm{cc}$ such that $\cA = \Proj^\mathrm{fg}(\mathcal{D})$ is a monoidal subcategory and $\otimes$ is cocontinuous in both arguments.
    For all $\lambda \in \Lambda$, fix a finite projective resolution
	\[ \cdots \to P_{-2}^\lambda \to P_{-1}^\lambda \to P_0^\lambda \to \Delta(\lambda) \to 0 , \]
	and consider the complexes
	\[ Y_\lambda \coloneqq \big( \cdots \to P_{-2}^\lambda \to P_{-1}^\lambda \to P_0^\lambda \to 0 \big) \]
    in $\cA$.
    Further recall from Subsection~\ref{subsec:tensorproductofcomplexes} that the monoidal structure on $\cA$ gives rise to a monoidal structure $(\otimes_\mathrm{c},\mathbf{1}_\mathrm{c},\ldots)$ on $K^b(\cA)$.
    We consider the following condition:
    \begin{flalign} \label{eq:Ytensor}
        \text{The subcategory }\big\langle Y_\nu[d] \mathop{\big|} \nu \in \Lambda , d \leq 0 \big\rangle_\mathrm{ext} \text{ is a monoidal subcategory of } K^b(\cA) .
        && \tag{$Y\otimes$}
	\end{flalign}
    Recall from \Cref{thm:Ringeldualityexceptioncollection} that the complexes $\{ Y_\lambda \mid \lambda \in \Lambda \}$ form an exceptional collection in $K^b(\cA)$, and let $\{ X_\lambda \mid \lambda \in \Lambda \}$ be a dual collection.
    We further consider the following condition:
    \begin{flalign} \label{eq:Xtensor}
        \text{The subcategory }\big\langle X_\nu[d] \mathop{\big|} \nu \in \Lambda , d \geq 0 \big\rangle_\mathrm{ext} \text{ is a monoidal subcategory of } K^b(\cA) .
        && \tag{$X\otimes$}
	\end{flalign}
\end{Definition}

\begin{Remark}
\label{rem:XtensorYtensordependence}
    Let $\mathcal{D}$ be an upper finite highest weight category with weight poset $(\Lambda,\leq)$ and with a monoidal structure $(\otimes,\mathbf{1},\ldots)$ on $\mathcal{D}^\mathrm{cc}$ as in \Cref{def:XtensorYtensor}.
    Further let $\cA = \Proj^\mathrm{fg}(\mathcal{D})$.
    \begin{enumerate}
        \item The definition of \eqref{eq:Ytensor} is independent of the choice of the complexes $\{ Y_\lambda \mid \lambda \in \Lambda \}$, because a projective resolution of $\Delta(\lambda)$ in $\mathcal{D}$ is unique up to isomorphism in $K^b(\cA)$ and because $\otimes_\mathrm{c}$ is a bifunctor.
        The definition of \eqref{eq:Xtensor} is independent of the choice of the complexes $\{ X_\lambda \mid \lambda \in \Lambda \}$ because the dual collection of $\{ Y_\lambda \mid \lambda \in \Lambda \}$ is unique up to isomorphism by \Cref{lem:dualcollectionunique}.
        In other words, the conditions \eqref{eq:Xtensor} and \eqref{eq:Ytensor} only depend on $\mathcal{D}$, the weight datum $(\Lambda,\leq,L)$ and the monoidal structure on $\mathcal{D}^\mathrm{cc}$.
        In the following, we will usually say that \emph{$\mathcal{D}$ satisfies \eqref{eq:Xtensor} or \eqref{eq:Ytensor}}, without explicit reference to the weight datum or the monoidal structure, if the latter are clear from context.
        \item Since the tensor product $\otimes_\mathrm{c}$ on $K^b(\cA)$ is exact and commutes with the shift functor, $\mathcal{D}$ satisfies \eqref{eq:Ytensor} if and only if the subcategory $\big\langle Y_\nu[d] \mathop{\big|} \nu \in \Lambda , d \leq 0 \big\rangle_\mathrm{ext}$ contains the tensor unit $\mathbf{1}_\mathrm{c}$ and all tensor products of the form $Y_\lambda \otimes_\mathrm{c} Y_\mu$ for $\lambda,\mu \in \Lambda$.
    \end{enumerate}
\end{Remark}

\begin{Remark}
    Let $\cC$ be a lower finite highest weight category with a monoidal structure $(\otimes,\mathbf{1},\ldots)$ such that $\Tilt(\cC)$ is a monoidal subcategory and $\otimes$ is exact in both arguments.
    Then another way of stating \Cref{prop:Ringeldualtensorcondition} is to say that the Ringel dual $\cC^\vee$ (equipped with the monoidal structure from \Cref{thm:monoidalRingelduality_lowertoupper} on $\cC^{\vee,\mathrm{cc}}$) satisfies the conditions \eqref{eq:Xtensor} and \eqref{eq:Ytensor}.
\end{Remark}

In practice, the condition \eqref{eq:Ytensor} is easier to verify explicitly (see Propositions~\ref{prop:torvanishingYtensor} below), but the condition \eqref{eq:Xtensor} is more useful in applications (see \Cref{thm:monoidalRingelduality_uppertolower}).
Therefore, it is helpful to note that these conditions are in fact opposite to each other, in the following sense.

\begin{Proposition}
\label{prop:oppositeXtensorYtensor}
    Let $\mathcal{D}$ be an upper finite highest weight category with a $\kk$-linear monoidal structure $(\otimes,\mathbf{1},\ldots)$ on $\mathcal{D}^\mathrm{cc}$ such that $\Proj^\mathrm{fg}(\mathcal{D})$ is a monoidal subcategory and $\otimes$ is cocontinuous in both arguments.
    Consider the opposite category $\mathcal{D}^\mathrm{op}$, with the monoidal structure $(\otimes',\mathbf{1},\ldots)$ on $\mathcal{D}^\mathrm{op,cc}$ from \Cref{rem:monoidalstructureupperfiniteopposite}.
    Then $\mathcal{D}$ satisfies \eqref{eq:Xtensor} if and only if $\mathcal{D}^\mathrm{op}$ satisfies \eqref{eq:Ytensor}.
\end{Proposition}
\begin{proof}
    Let us write $(\Lambda,\leq)$ for the weight poset of $\mathcal{D}$ and $\cA = \Proj(\mathcal{D})$, and observe that there is a canonical equivalence $\cA^\mathrm{op} \simeq \Proj^\mathrm{fg}(\mathcal{D}^\mathrm{op})$ by \Cref{rem:oppositeupperfinite}.
    To avoid confusion, we will denote the standard objects in $\mathcal{D}$ and $\mathcal{D}^\mathrm{op}$ by $\Delta^\mathcal{D}(\lambda)$ and $\Delta^{\mathcal{D}^\mathrm{op}}(\lambda)$, respectively, for $\lambda \in \Lambda$.
    For all $\lambda \in \Lambda$, we fix a finite projective resolution
	\[ \cdots \to P_{-2}^\lambda \to P_{-1}^\lambda \to P_0^\lambda \to \Delta^\mathcal{D}(\lambda) \to 0 , \]
	in $\mathcal{D}$ and consider the complexes
	\[ Y_\lambda^\mathcal{D} \coloneqq \big( \cdots \to P_{-2}^\lambda \to P_{-1}^\lambda \to P_0^\lambda \to 0 \big) \]
    in $\cA$.
    We also consider the complexes $Y_\lambda^{\mathcal{D}^\mathrm{op}}$ in $\cA^\mathrm{op}$ that are defined via finite projective resolutions of the standard objects $\Delta^{\mathcal{D}^\mathrm{op}}(\lambda)$ in $\mathcal{D}^\mathrm{op}$ as above.
    By \Cref{thm:Ringeldualityexceptioncollection}, the complexes $\{ Y_\lambda^\mathcal{D} \mid \lambda \in \Lambda \}$ form an exceptional collection in the triangulated category $\cT = K^b(\cA)$, and $\{ Y_\lambda^\mathcal{D} \mid \lambda \in \Lambda \}$ admits a dual collection $\{ X_\lambda^\mathcal{D} \mid \lambda \in \Lambda \}$.
    Furthermore, the Ringel dual $\cC = \prescript{\vee}{}{\mathcal{D}}$ of $\mathcal{D}$ is the heart of the $t$-structure $(\cT^{\geq 0},\cT^{\leq 0}) = (\mathcal{X},\mathcal{Y})$ on $\cT$ with
    \begin{align*}
        \mathcal{X} & = \big\langle X_\lambda^\mathcal{D}[d] \mathop{\big|} \lambda \in \Lambda , d \leq 0 \big\rangle_\mathrm{ext} , \\
        \mathcal{Y} & = \big\langle Y_\lambda^\mathcal{D}[d] \mathop{\big|} \lambda \in \Lambda , d \geq 0 \big\rangle_\mathrm{ext} ,
    \end{align*}
    the standard and costandard objects in $\cC$ are given by
    \[ \nabla^\cC(\lambda) = Y_\lambda^\mathcal{D} , \hspace{2cm} \Delta^\cC(\lambda) = X_\lambda^\mathcal{D} \]
    for $\lambda \in \Lambda$, and the tilting objects in $\cC$ are $\Tilt(\cC) = \cA$ (embedded in $K^b(\cA)$ as complexes concentrated in degree zero).
    The opposite category $\cC^\mathrm{op}$ is the heart of the $t$-structure $(\mathcal{Y}^\mathrm{op},\mathcal{X}^\mathrm{op})$ on $\cT^\mathrm{op}$, and $\cC^\mathrm{op}$ is a lower finite highest weight category with costandard objects
    \begin{equation} \label{eq:costandardobjectopposite}
        \nabla^{\cC^\mathrm{op}}(\lambda) = \Delta^\cC(\lambda) = X_\lambda^\mathcal{D} ,
    \end{equation}
    for $\lambda \in \Lambda$, and with tilting objects $\Tilt(\cC^\mathrm{op}) = \cA^\mathrm{op}$.
    Next observe that there is a monoidal equivalence
    \[ \mathrm{op} \colon \cT^\mathrm{op} \xrightarrow{~\sim~} K^b(\cA^\mathrm{op}) , \]
    with respect to the tensor product of complexes on $\cT$ and $K^b(\cA^\mathrm{op})$, which sends a complex in $\cA$ to the corresponding complex in $\cA^\mathrm{op}$ with the grading and differentials reversed.
    The equality \eqref{eq:costandardobjectopposite} implies that the complex $\mathrm{op}(X_\lambda^\mathcal{D})$ is isomorphic in $K^b(\cA^\mathrm{op})$ to a resolution of $\nabla^{\cC^\mathrm{op}}(\lambda)$ by tilting objects in $\cC^\mathrm{op}$, and so Ringel duality implies that $\mathrm{op}(X_\lambda^\mathcal{D})$ (now viewed as a complex of projective objects in $\mathcal{D}^\mathrm{op}$) is also isomorphic to a projective resolution of $\Delta^{\mathcal{D}^\mathrm{op}}(\lambda)$ in $\mathcal{D}^\mathrm{op}$.
    As projective resolutions are unique up to isomorphism in the homotopy category, we conclude that $Y_\lambda^{\mathcal{D}^\mathrm{op}} \cong \mathrm{op}( X_\lambda^\mathcal{D} )$ in $K^b(\cA^\mathrm{op})$.
    Now the claim is immediate from the definitions of \eqref{eq:Ytensor} and \eqref{eq:Xtensor} and the fact that $\mathrm{op}$ is monoidal.
\end{proof}

For an upper finite highest weight category $\mathcal{D}$ with a monoidal structure on $\mathcal{D}^\mathrm{cc}$ as in \Cref{prop:oppositeXtensorYtensor}, the subcategory $\Proj^\mathrm{fg}(\mathcal{D})$ often admits a monoidal duality, which simplifies the tasks of checking the conditions \eqref{eq:Xtensor} and \eqref{eq:Ytensor}, as we explain below.
The following definition is a monoidal analogue of \Cref{def:orderpreservingduality}; recall that $\cA^\mathrm{rev}$ denotes the reverse monoidal category of a monoidal category $\cA$, with the same underlying category, but with the reverse tensor product $X \otimes^\mathrm{rev}Y = Y \otimes X$.

\begin{Definition}
\label{def:orderpreservingmonoidalduality}
    Let $\cA$ be a $\kk$-linear monoidal Krull--Schmidt category with a poset $(\Lambda,\leq)$ and a map
    $\Lambda \to \mathrm{Ob}(\cA) , ~ \lambda \mapsto X_\lambda$
    indexing the isomorphism classes of indecomposable objects in $\cA$.
    An \emph{order preserving monoidal duality} on $\cA$ is a $\kk$-linear monoidal equivalence $\tau \colon \cA^\mathrm{rev} \to \cA^\mathrm{op}$ with an order automorphism $\sigma \colon \Lambda \to \Lambda$ such that $\tau (X_\lambda) \cong X_{\sigma(\lambda)}$ for all $\lambda \in \Lambda$.
\end{Definition}

\begin{Corollary}
\label{cor:dualityXtensorYtensor}
    Let $\mathcal{D}$ be an upper finite highest weight category with a monoidal structure $(\otimes,\mathbf{1},\ldots)$ on $\mathcal{D}^\mathrm{cc}$ such that $\Proj^\mathrm{fg}(\mathcal{D})$ is a monoidal subcategory and $\otimes$ is cocontinuous in both arguments.
    Further suppose that there is an order preserving monoidal duality
    \[ \tau \colon \Proj^\mathrm{fg}(\mathcal{D})^\mathrm{rev} \longrightarrow \Proj^\mathrm{fg}(\mathcal{D})^\mathrm{op} . \]
    If $\mathcal{D}$ satisfies one of the conditions \eqref{eq:Xtensor} or \eqref{eq:Ytensor}, then $\mathcal{D}$ satisfies both \eqref{eq:Xtensor} and \eqref{eq:Ytensor}.
\end{Corollary}
\begin{proof}
    Let us write $(\Lambda,\leq)$ for the weight poset of $\mathcal{D}$ and $\sigma \colon \Lambda \to \Lambda$ for the order automorphism such that $\tau P(\lambda) \cong P( \sigma \lambda )$.
    Further let $\cA = \Proj^\mathrm{fg}(\mathcal{D})$ and recall from Remarks~\ref{rem:presheavesonprojectives} and \ref{rem:oppositeupperfinite} that there are equivalences $\mathcal{D} \simeq \PSh^\mathrm{fd}(\cA)$ and $\mathcal{D}^\mathrm{op} \simeq \PSh^\mathrm{fd}(\cA^\mathrm{op})$.
    Composing with $\tau$ gives rise to an equivalence
    \[ \tilde \tau \colon \mathcal{D}^\mathrm{cc} \simeq \PSh(\cA) \longrightarrow \PSh(\cA^\mathrm{op}) \simeq \mathcal{D}^\mathrm{op,cc} , \qquad F \longmapsto F \circ \tau , \]
    which restricts to an equivalence $\mathcal{D} \simeq \mathcal{D}^\mathrm{op}$.
    Furthermore, the equivalence $\tilde \tau$ is compatible with the highest weight structures on $\mathcal{D}$ and $\mathcal{D}^\mathrm{op}$, in that $\tilde \tau L(\lambda) \cong L( \sigma \lambda )$ for $\lambda \in \Lambda$.
    
    Next recall from \Cref{rem:monoidalstructureupperfiniteopposite} that the monoidal structure on $\cA$ gives rise to a monoidal structure on $\mathcal{D}^{\mathrm{op},\mathrm{cc}} \simeq \PSh(\cA^\mathrm{op})$ such that the composition
    \[ \cA^\mathrm{rev} \xrightarrow{~\tau~} \cA^\mathrm{op} \xrightarrow{~Y_{\cA^\mathrm{op}}~} \PSh(\cA^\mathrm{op}) \simeq \mathcal{D}^{\mathrm{op,cc}} \]
    is monoidal.
    The equivalence $\tilde \tau$ identifies with the Yoneda extension of $Y_{\cA^\mathrm{op}} \circ \tau$, and so $\tilde \tau$ can be enhanced to a monoidal equivalence $\mathcal{D}^\mathrm{cc,rev} \simeq \mathcal{D}^\mathrm{op,cc}$ by Remarks~\ref{rem:freemonoidalcocompletion} and \ref{rem:monoidalstructureupperfiniteopposite}.
    Now if $\mathcal{D}$ satisfies \eqref{eq:Ytensor} with respect to the given monoidal structure $(\otimes,\mathbf{1},\ldots)$ on $\mathcal{D}^\mathrm{cc}$, then $\mathcal{D}$ also satisfies \eqref{eq:Ytensor} with respect to $(\otimes^\mathrm{rev},\mathbf{1},\ldots)$, because $\mathcal{D}^\mathrm{cc}$ and $\mathcal{D}^\mathrm{cc,rev}$ have the same monoidal subcategories.
    As $\tilde \tau \colon \mathcal{D}^\mathrm{cc,rev} \to \mathcal{D}^\mathrm{op,cc}$ is a monoidal equivalence with $\tilde \tau L(\lambda) \cong L(\sigma \lambda)$ for all $\lambda \in \Lambda$, it then follows from \Cref{rem:XtensorYtensordependence} that $\mathcal{D}^\mathrm{op}$ also satisfies \eqref{eq:Ytensor}, and so $\mathcal{D}$ satisfies \eqref{eq:Xtensor} by \Cref{prop:oppositeXtensorYtensor}.
    Analogously, if $\mathcal{D}$ satisfies \eqref{eq:Xtensor}, then so does $\mathcal{D}^\mathrm{op}$, and it follows that $\mathcal{D}$ satisfies \eqref{eq:Ytensor}.
\end{proof}

The following proposition provides a more explicit criterion for checking whether a given upper finite highest weight category satisfies the condition \eqref{eq:Ytensor}.
It will be used, for instance, to show that \eqref{eq:Ytensor} is satisfied for certain categories of presheaves on interpolation categories, in \Cref{sec:tensorenvelopes} below.

\begin{Proposition} \label{prop:torvanishingYtensor}
    Let $\mathcal{D}$ be an upper finite highest weight category and let $(\otimes,\mathbf{1},\ldots)$ be a monoidal structure on $\mathcal{D}^\mathrm{cc}$ such that $\Proj^\mathrm{fg}(\mathcal{D})$ is a monoidal subcategory and $\otimes$ is cocontinuous in both arguments.
    Further suppose that
    \begin{flalign} \label{eq:Deltatensor}
		& \mathcal{D}_\Delta \text{ is closed under tensor products in } \mathcal{D}^\mathrm{cc}; && \tag{$\Delta\otimes$} \\
        \label{eq:DeltaTor}
		& \text{we have } \mathrm{Tor}_i^\mathcal{D}(M,N) = 0 \text{ for all } i>0 \text{ and all objects  } M \text{ and } N \text{ of } \mathcal{D}_\Delta . && \tag{$\Delta\mathrm{Tor}$}
	\end{flalign}
    Then $\mathcal{D}$ satisfies \eqref{eq:Ytensor}.
\end{Proposition}
\begin{proof}
    Let us write $(\Lambda,\leq)$ for the weight poset of $\mathcal{D}$.
    For all $\lambda \in \Lambda$, we fix a finite projective resolution
    \[ \cdots \to P_{-2}^\lambda \to P_{-1}^\lambda \to P_0^\lambda \to \Delta(\lambda) \to 0 , \]
	of $\Delta(\lambda)$ in $\mathcal{D}$ and consider the complex
	\[ Y_\lambda \coloneqq \big( \cdots \to P_{-2}^\lambda \to P_{-1}^\lambda \to P_0^\lambda \to 0 \big) , \]
    as in \Cref{def:XtensorYtensor}.
    By \Cref{rem:XtensorYtensordependence}, it suffices to show that the subcategory
    \[ \mathcal{Y} = \big\langle Y_\nu[d] \mathop{\big|} \nu \in \Lambda , d \leq 0 \big\rangle_\mathrm{ext} \]
    of $K^b\big( \Proj^\mathrm{fg}(\mathcal{D}) \big)$ contains $\mathbf{1}$ and $Y_\lambda \otimes_\mathrm{c} Y_\mu$ for all $\lambda,\mu \in \Lambda$, where $\otimes_\mathrm{c}$ denotes the tensor product of complexes.
    To that end, first observe that the canonical functor $K^b\big( \Proj^\mathrm{fg}(\mathcal{D}) \big) \longrightarrow D^b(\mathcal{D})$ is fully faithful (see Corollary 10.4.7 in \cite{WeibelHomAlg}) and sends $Y_\lambda$ to $\Delta(\lambda)$ (considered as a complex concentrated in degree zero) for all $\lambda \in \Lambda$.
    Since $\Proj^\mathrm{fg}(\mathcal{D})$ is a monoidal subcategory of $\mathcal{D}^\mathrm{cc}$, the tensor unit $\mathbf{1}$ is finitely generated projective, and since $\mathcal{D}$ satisfies the condition $(P\Delta)$, we conclude that $\mathbf{1}$ belongs to $\mathcal{Y}$, as required.
    Next, we claim that
    \[ Y_\lambda \otimes_\mathrm{c} Y_\mu \cong \Delta(\lambda) \otimes \Delta(\mu) \]
    in $D^b(\mathcal{D})$, where $\Delta(\lambda) \otimes \Delta(\mu)$ is viewed as a complex concentrated in degree zero.
    In order to prove the claim, we will first compute the cohomology groups of $Y_\lambda \otimes_\mathrm{c} Y_\mu$.

    Recall from Subsection~\ref{subsec:tensorproductofcomplexes} that $Y_\lambda \otimes_\mathrm{c} Y_\mu$ is the total complex of the double complex $P_{\bullet,\bullet} = Y_\lambda \otimes_\mathrm{dc} Y_\mu$ with terms $P_{i,j} = P_i^\lambda \otimes P_j^\mu$.
    The cohomology of the total complex is the limit of the spectral sequence of the double complex
    \[ E_2^{p,q} = H^i_\mathrm{hor}\big( H^j_\mathrm{ver}(P_{\bullet,\bullet}) \big) \Longrightarrow H^{i+j}( Y_\lambda \otimes_\mathrm{c} Y_\mu ) , \]
    where $H^j_\mathrm{ver}(P_{\bullet,\bullet})$ denotes the complex whose $i$-th terms is the $j$-th ``vertical'' cohomology group of the complex $P_{i,\bullet}$ and $H^i_\mathrm{hor}\big( H^j_\mathrm{ver}(P_{\bullet,\bullet}) \big)$ is the $i$-th ``horizontal'' cohomology group of $H^j_\mathrm{ver}(P_{\bullet,\bullet})$.
    Note that for all $i \leq 0$, we have $P_{i,\bullet} = P_i^\lambda \otimes Y_\mu$ and so
    \[ H^j( P_{i,\bullet} ) = H^j( P_i^\lambda \otimes Y_\mu ) = \mathrm{Tor}^\mathcal{D}_{-j}\big( P_i^\lambda , \Delta(\mu) \big) = \begin{cases} P_i^\lambda \otimes \Delta(\mu) & \text{if } j=0 , \\ 0 & \text{otherwise}
    \end{cases} \]
    for all $j \leq 0$ by \eqref{eq:DeltaTor}.
    This implies that $H^j_\mathrm{ver}(P_{\bullet,\bullet}) = 0$ for $j \neq 0$ and $H^0_\mathrm{ver}(P_{\bullet,\bullet}) \cong Y_\lambda \otimes \Delta(\mu)$, and as before, we obtain
    \[ H^i_\mathrm{hor}\big( H^0_\mathrm{ver}(P_{\bullet,\bullet}) \big) = H^i( Y_\lambda \otimes \Delta(\mu) ) = \mathrm{Tor}^\mathcal{D}_{-i}\big( \Delta(\lambda) , \Delta(\mu) \big) \cong \begin{cases} \Delta(\lambda) \otimes \Delta(\mu) & \text{if } i = 0 \\ 0 & \text{otherwise} \end{cases} \]
    for all $i \leq 0$ by \eqref{eq:DeltaTor}.
    Thus, the spectral sequence degenerates and we obtain
    \[ H^i( Y_\lambda \otimes_\mathrm{c} Y_\mu ) \cong \begin{cases} \Delta(\lambda) \otimes \Delta(\mu) & \text{if } i=0 , \\ 0 & \text{otherwise} . \end{cases} \]
    Therefore, we have $Y_\lambda \otimes_\mathrm{c} Y_\mu \cong \Delta(\lambda) \otimes \Delta(\mu)$ in $D^b(\mathcal{D})$, as claimed.
    
    Now as $\Delta(\lambda) \otimes \Delta(\mu)$ belongs to $\mathcal{D}_\Delta$ by \eqref{eq:Deltatensor} and as the canonical functor $K^b\big( \Proj^\mathrm{fg}(\mathcal{D}) \big) \to D^b(\mathcal{D})$ is fully faithful and sends $Y_\nu$ to $\Delta(\nu)$ for all $\nu \in \Lambda$, we conclude that $Y_\lambda \otimes Y_\mu$ belongs to $\mathcal{Y}$, as required.
\end{proof}

Now we are ready to prove the main result of this subsection, providing a converse to \Cref{thm:monoidalRingelduality_lowertoupper}.

\begin{Theorem}[Upper to lower monoidal Ringel duality]
\label{thm:monoidalRingelduality_uppertolower}
    Let $\mathcal{D}$ be an upper finite highest weight category with a monoidal structure $(\otimes_\mathcal{D},\mathbf{1}_\mathcal{D},\ldots)$ on $\mathcal{D}^\mathrm{cc}$ such that $\Proj^\mathrm{fg}(\mathcal{D})$ is a monoidal subcategory and $\otimes_\mathcal{D}$ is cocontinuous in both arguments.
    Let $\cC = \prescript{\vee}{}{\mathcal{D}}$ be the Ringel dual of $\cC$.
    \begin{enumerate}
        \item Suppose that $\mathcal{D}$ satisfies the condition \eqref{eq:Xtensor}.
        Then $\cC$ admits a
        monoidal structure $(\otimes_\cC,\mathbf{1}_\cC,\ldots)$ such that $\otimes_\cC$ is right exact in both arguments, $\Tilt(\cC)$ is a monoidal subcategory and the Ringel duality functor $R = R_\cC \colon \cC \to \cC^\vee \simeq \mathcal{D}$ restricts to a monoidal equivalence $\Tilt(\cC) \simeq \Proj^\mathrm{fg}( \mathcal{D} )$.
        
        The opposite Ringel duality functor $L \colon \mathcal{D}^\mathrm{cc} \simeq \cC^{\vee,\mathrm{cc}} \to \cC^\mathrm{cc}$ can be enhanced to a monoidal functor (with respect to the canonical extension of $\otimes_\cC$ to $\cC^\mathrm{cc} = \Ind(\cC)$).
        
        If $\mathcal{D}^\mathrm{cc}$ has a (symmetric) braiding, then $\cC$ admits a canonical (symmetric) braiding such that $L \colon \mathcal{D}^\mathrm{cc} \to \cC^\mathrm{cc}$ is a braided monoidal functor (with respect to the induced braiding on $\cC^\mathrm{cc}$).
        \item Suppose that $\mathcal{D}$ satisfies the conditions \eqref{eq:Xtensor} and \eqref{eq:Ytensor}.
        Then $\otimes_\cC$ is exact in both arguments.
        \item Suppose that $\mathcal{D}$ satisfies one of the conditions \eqref{eq:Xtensor} or \eqref{eq:Ytensor}, that $\Proj^\mathrm{fg}(\mathcal{D})$ is rigid and that the duality functors
        \[ (-)^* \colon \Proj^\mathrm{fg}(\mathcal{D}) \longrightarrow \Proj^\mathrm{fg}(\mathcal{D})^\mathrm{op} , \qquad \prescript{*}{}{(-)} \colon \Proj^\mathrm{fg}(\mathcal{D}) \longrightarrow \Proj^\mathrm{fg}(\mathcal{D})^\mathrm{op} \]
        are order preserving dualities in the sense of \Cref{def:orderpreservingduality}.
        Then $(\cC,\otimes_\cC,\mathbf{1}_\cC,\ldots)$ is rigid, the monoidal subcategory $\Tilt(\cC)$ is closed under taking duals and $\mathcal{D}$ satisfies \eqref{eq:Xtensor} and \eqref{eq:Ytensor}.
        In particular, $\otimes_\cC$ is exact in both arguments.
    \end{enumerate}
\end{Theorem}

\begin{proof}
    Let us write $(\Lambda,\leq)$ for the weight poset of $\mathcal{D}$ and $\cA = \Proj^\mathrm{fg}(\mathcal{D})$.
    By \Cref{thm:Ringeldualityexceptioncollection}, the Ringel dual of $\mathcal{D}$ can be defined as the heart of a $t$-structure on the triangulated category $\cT = K^b(\cA)$; more precisely, there are complexes $\{ Y_\lambda \mid \lambda \in \Lambda \}$ and $\{ X_\lambda \mid \lambda \in \Lambda \}$ in $\cA$ such that the subcategories
    \begin{align*}
		\cT^{\geq 0} & = \big\langle Y_\lambda[d] \mathrel{\big|} \lambda \in \Lambda , d \leq 0 \big\rangle_\mathrm{ext} , \\
		\cT^{\leq 0} & = \big\langle X_\lambda[d] \mathrel{\big|} \lambda \in \Lambda , d \geq 0 \big\rangle_\mathrm{ext} .
	\end{align*}
    form a $t$-structure on $\cT$ whose heart $\cC = \cT^{\leq 0} \cap \cT^{\geq 0}$ is a lower finite highest weight category with weight poset $(\Lambda,\leq^\mathrm{op})$ and Ringel dual $\cC^\vee \simeq \mathcal{D}$.
    Furthermore, we have $\Tilt(\cC) = \cA = \Proj^\mathrm{fg}(\mathcal{D})$, the subcategory of complexes concentrated in degree zero in $\cT = K^b(\cA)$, and for the Ringel duality functor $R \colon \cC \to \cC^\vee \simeq \mathcal{D}$, there is a natural isomorphism $R \cong H^0(-) \circ \imath$, where $\imath \colon \cC \to \cT$ is the canonical embedding and $H^0(-) \colon \cT \to \mathcal{D}$ is the degree zero cohomology functor.

    Now suppose that $\mathcal{D}$ satisfies \eqref{eq:Xtensor}, i.e.\ that the subcategory $\cT^{\leq 0}$ is a monoidal subcategory of $\cT$ with respect to the tensor product of complexes $\otimes_\mathrm{c}$.
    Further let $\tau_{\geq 0} \colon \cT \to \cT^{\geq 0}$ for the truncation functor with respect to the $t$-structure $(\cT^{\leq 0},\cT^{\geq 0})$, as in \cite[Proposition 1.3.3]{FaisceauxPervers}.
    Then there is a unique monoidal structure on $\cC = \cT^{\geq 0} \cap \cT^{\leq 0}$ such that the restriction $\tau_{\geq 0} \colon \cT^{\leq 0} \to \cC$ can be enhanced to a monoidal functor, with tensor product given by
    \[ X \otimes_\cC Y = \tau_{\geq 0}( X \otimes_\mathrm{c} Y ) \]
    for objects $X$ and $Y$ of $\cC$ (see for instance Section 1.2.4 in \cite{BondarkoDeglise}, where this result is discussed using homological instead of cohomological indexing for $t$-structures).
    The tensor unit is $\mathbf{1}_\cC = \tau_{\geq 0} \mathbf{1}_\mathcal{D} = \mathbf{1}_\mathcal{D}$, where the second equality follows from the fact that $\mathbf{1}_\mathcal{D}$ belongs to $\Proj^\mathrm{fg}(\mathcal{D}) = \cA = \Tilt(\cC)$.
    The tensor product $\otimes_\cC$ is right exact in both arguments by the general results on right $t$-exact functors discussed in \cite[Section 1.3.16]{FaisceauxPervers}.
    Furthermore, since $\Tilt(\cC) = \cA = \Proj^\mathrm{fg}(\mathcal{D})$ is closed under tensor products in $\cT$, we have
    \[ X \otimes_\cC Y = X \otimes_\mathrm{c} Y = X \otimes_\mathcal{D} X \]
    for all objects $X$ and $Y$ of $\Tilt(\cC) = \cA$, and since $H^0(-) \circ \imath$ restricts to the identity functor on $\Tilt(\cC) = \cA = \Proj^\mathrm{fg}(\mathcal{D})$, we conclude that the Ringel duality functor $R \cong H^0(-) \circ \imath$ restricts to a monoidal equivalence $\Tilt(\cC) \simeq \Proj^\mathrm{fg}(\mathcal{D})$.

    Since $\otimes_\cC$ is right exact in both arguments, the monoidal structure on $\cC$ uniquely extends to a monoidal structure on $\cC^\mathrm{cc} = \Ind(\cC)$ with a tensor product that is cocontinuous in both arguments.
    With respect to this monoidal structure on $\cC^\mathrm{cc}$, the opposite Ringel duality functor $L \colon \mathcal{D}^\mathrm{cc} \to \cC^\mathrm{cc}$ can be enhanced to a monoidal functor by part (4) of \Cref{thm:monoidalRingelduality_lowertoupper}, because $\mathcal{D} \simeq \cC^\vee$ is the Ringel dual of $\cC$ and the monoidal structure on $\mathcal{D}^\mathrm{cc}$ is unique (up to monoidal equivalence) with the properties that $\otimes_\mathcal{D}$ is cocontinuous in both arguments and the Ringel duality functor $R \colon \cC \to \mathcal{D}$ restricts to a monoidal equivalence $\Tilt(\cC) \simeq \Proj^\mathrm{fg}(\mathcal{D})$.

    If $\mathcal{D}$ admits a (symmetric) braiding, then the latter restricts to a (symmetric) braiding on the monoidal subcategory $\cA = \Proj^\mathrm{fg}(\mathcal{D})$, which extends to a (symmetric) braiding on $\cT = K^b(\cA)$.
    By definition of the tensor product $\otimes_\cC$, the (symmetric) braiding on $\cT$ induces a canonical (symmetric) braiding on $\cC$ such that the Ringel duality functor $R \colon \cC \to \mathcal{D}$ restricts to a braided monoidal equivalence $\Tilt(\cC) \simeq \Proj^\mathrm{fg}(\mathcal{D})$.
    The claim that $L$ is braided follows from part (4) of \Cref{thm:monoidalRingelduality_lowertoupper} and (as before) the uniqueness of the monoidal structure on $\mathcal{D}^\mathrm{cc}$.

    Now suppose that $\mathcal{D}$ satisfies \eqref{eq:Xtensor} and \eqref{eq:Ytensor}.
    Then both $\cT^{\leq 0}$ and $\cT^{\geq 0}$ are closed under tensor products in $\cT$, hence so is $\cC = \cT^{\geq 0} \cap \cT^{\leq 0}$, and we have
    \[ X \otimes_\cC Y = \tau_{\geq 0}( X \otimes_\mathrm{c} Y ) = X \otimes_\mathrm{c} Y \]
    for all objects $X$ and $Y$ of $\cC$.
    The tensor product $\otimes_\cC$ is exact in both arguments by the general results on $t$-exact functors discussed in \cite[Section 1.3.16]{FaisceauxPervers}.

    Finally, suppose that $\mathcal{D}$ satisfies one of the conditions \eqref{eq:Xtensor} or \eqref{eq:Ytensor}, that $\cA = \Proj^\mathrm{fg}(\mathcal{D})$ is rigid, and that the duality functors are order preserving dualities in the sense of \Cref{def:orderpreservingduality}.
    Then the left duality functor is automatically an order preserving monoidal duality $(-)^* \colon \cA^\mathrm{rev} \to \cA^\mathrm{op}$ in the sense of \Cref{def:orderpreservingmonoidalduality}, hence $\mathcal{D}$ satisfies both of the conditions \eqref{eq:Xtensor} and \eqref{eq:Ytensor} by \Cref{cor:dualityXtensorYtensor}, and $\otimes_\cC$ is exact in both arguments by part (2) of the theorem.
    Furthermore, the tensor product of complexes $\otimes_\mathrm{c}$ makes $\cT = K^b(\cA)$ a rigid monoidal category, with left duality functor $(-)^* \colon \cT \to \cT^\mathrm{op}$ induced by the left duality functor on $\cA$, as explained in Subsection~\ref{subsec:tensorproductofcomplexes}.
    By part (6) of \Cref{thm:Ringeldualityexceptioncollection}, the left duality functor on $\cT$ restricts to a functor $(-)^* \colon \cC \to \cC^\mathrm{op}$, which in turn restricts to the left duality functor on $\cA = \Tilt(\cC)$ (and similarly for the right duality functor).
    In particular, the monoidal subcategory $\cC$ is closed under taking duals in the rigid monoidal category $\cT$, and so $\cC$ is rigid and $\Tilt(\cC)$ is closed under taking duals, as claimed.
\end{proof}

\subsection{Monoidal abelian envelopes} \label{subsec:abelian-envelopes}

Following \cite[Section~1.3]{CoulembierMonAbEnv}, a \emph{pseudo-tensor category} is an essentially small $\kk$-linear additive idempotent complete rigid monoidal category whose tensor unit has a one-dimensional endomorphism space.
A $\kk$-linear monoidal functor between pseudo-tensor categories is called a \emph{tensor functor}, and a pseudo-tensor category that is abelian is called a \emph{tensor category}.
Given a pseudo-tensor category $\cA$, one may ask whether $\cA$ embeds into a (universal) tensor category.
This motivates the following definition (see \cite[Definition 2.1.1]{CEOPAbEnvQuotProp}).

\begin{Definition}
    Let $\cA$ be a pseudo-tensor category.
    A \emph{weak monoidal abelian envelope} of $\cA$ is a pair $(\cT,e)$,
    where $\cT$ is a tensor category and $e \colon \cA \to \cT$ is a
    faithful tensor functor, such that restriction along $e$ induces an equivalence of categories
    \begin{equation} \label{eq:mon-ab-env-universal-property}
        \Tens^{\text{exact}}(\cT,\cT') \xrightarrow{~~\sim~~} \Tens^{\text{faith}}(\cA,\cT') ,
    \end{equation}
    for every tensor category $\cT'$.
    Here, $\Tens^{\text{faith}}(\cA,\cT')$ and $\Tens^{\text{exact}}(\cT,\cT')$ denote the categories of tensor functors that are faithful or exact, respectively, with homomorphisms given by monoidal natural transformations.
    A \emph{monoidal abelian envelope} is a weak monoidal abelian envelope $(\cT,e)$ such that the functor $e \colon \cA \to \cT$ is fully faithful.
\end{Definition}

For a pseudo-tensor category $\cA$ with monoidal abelian envelope $(\cT,e)$, the universal property \eqref{eq:mon-ab-env-universal-property} implies that for every tensor category $\cT'$ with a faithful tensor functor $\imath \colon \cA \to \cT'$, there is a unique (up to monoidal natural isomorphism) exact tensor functor $F \colon \cT \to \cT'$ such that $\imath \cong F \circ e$ as monoidal functors.
Note that every exact tensor functor $F \colon \cT \to \cT'$ is automatically faithful by \cite[Theorem~2.4.1]{CEOPAbEnvQuotProp}, hence $F \circ e$ is also faithful.

Now let $\cC$ be a tensor category with a lower finite highest weight structure such that $\Tilt(\cC)$ is a monoidal subcategory closed under taking duals (so that $\Tilt(\cC)$ is a pseudo-tensor category).
In the context of monoidal Ringel duality, it would be desirable to understand whether the canonical embedding $\imath \colon \Tilt(\cC) \to \cC$ exhibits $\cC$ as a monoidal abelian envelope of $\Tilt(\cC)$.
While this need not be true in general (see \Cref{rem:abelian-envelope} below), we can prove that $\cC$ is indeed a monoidal abelian envelope of $\Tilt(\cC)$ provided that $\Tilt(\cC)$ admits a monoidal abelian envelope.

\begin{Theorem} \label{thm:abelian-envelope}
    Let $\cC$ be a tensor category with a lower finite highest weight structure such that $\Tilt(\cC)$ is a monoidal subcategory closed under taking duals.
     If $\Tilt(\cC)$ admits a weak monoidal abelian envelope, then the canonical embedding $\imath \colon \Tilt(\cC) \to \cC$ is a monoidal abelian envelope.
\end{Theorem}
\begin{proof}
    Let us write $\cA = \Tilt(\cC)$ and let $e \colon \cA \to \cT$ be a weak monoidal abelian envelope of $\cA$ (so $\cT$ is a tensor category).
    We write $(\otimes',\mathbf{1}',\ldots)$ for the monoidal structure on $\cT$.
    By the universal property of monoidal abelian envelopes, there is a unique (up to monoidal natural isomorphism) exact tensor functor $F \colon \cT \to \cC$ such that $F \circ e \cong \imath$ as monoidal functors.
    Further observe that $F$ is faithful by \cite[Theorem~2.4.1]{CEOPAbEnvQuotProp}.
    
    In order to define a quasi-inverse of $F$, first observe that $\imath \colon \cA \to \cC$ gives rise to a functor
    \[ \imath_\mathrm{c} \colon K^b(\cA) \longrightarrow D^b(\cC) \]
    that sends a complex (or homomorphism) to itself, but now considered as an object (or homomorphism) in $D^b(\cC)$.
    The functor $\imath_\mathrm{c}$ can be canonically enhanced to a monoidal functor with respect to the tensor product of complexes $\otimes_\mathrm{c}$ on $K^b(\cA)$ and $D^b(\cC)$ (see Subsection~\ref{subsec:tensorproductofcomplexes}), and $\imath_\mathrm{c}$ is an equivalence of triangulated categories by \eqref{eq:tiltingequivalence}.
    We fix a monoidal quasi-inverse functor
    \[ \psi \colon D^b(\cC) \longrightarrow K^b(\cA) \]
    of $\imath_\mathrm{c}$ and consider the monoidal ``tilting complex'' functor
    \[ \mathrm{tc} \colon \cC \xrightarrow{~\eta~} D^b(\cC) \xrightarrow{~\psi~} K^b(\cA) , \]
    where $\eta \colon \cC \to D^b(\cC)$ is the canonical functor that sends an object of $\cC$ to itself (considered as a complex concentrated in homological degree zero).
    By construction, we have $H^0(-) \circ \mathrm{tc} \cong \id_\cC$ and $H^i(-) \circ \mathrm{tc} = 0$ for $i \neq 0$.
    Next observe that the functor $e \colon \cA \to \cT$ gives rise to a triangulated functor
    \[ e_\mathrm{c} \colon K^b(\cA) \longrightarrow D^b(\cT) \]
    that sends a complex $( \cdots \to T_i \to T_{i+1} \to \cdots)$ to the complex $( \cdots \to e(T_i) \to e(T_{i+1}) \to \cdots )$, and that $e_\mathrm{c}$ can be enhanced to a monoidal functor with respect to the tensor product of complexes $\otimes_\mathrm{c}$ on $K^b(\cA)$ and $\otimes'_\mathrm{c}$ on $D^b(\cT)$.
    Analogously, the exact tensor functor $F \colon \cT \to \cC$ gives rise to a monoidal functor
    \[ F_\mathrm{c} \colon D^b(\cT) \longrightarrow D^b(\cC) \]
    such that $H^i(-) \circ F_\mathrm{c} \cong F \circ H^i(-)$ for all $i \in \Z$.
    Thus, we obtain isomorphisms of functors
    \[ F \circ H^i(-) \circ e_\mathrm{c} \circ \mathrm{tc} \cong H^i(-) \circ F_\mathrm{c} \circ e_\mathrm{c} \circ \mathrm{tc} \cong H^i(-) \circ \imath_\mathrm{c} \circ \mathrm{tc} \cong H^i(-) \circ \eta \cong \begin{cases}
        \id_\cC & \text{if } i=0 , \\ 0 & \text{otherwise} ,
    \end{cases} \]
    where in the second isomorphism, we use that $F_\mathrm{c} \circ e_\mathrm{c} \cong \imath_\mathrm{c}$ because $F \circ e \cong \imath$.
    Since $F$ is faithful, this implies that $H^i(-) \circ e_\mathrm{c} \circ \mathrm{tc} = 0$ for all $i \neq 0$.
    We now define
    \[ G \coloneqq H^0(-) \circ e_\mathrm{c} \circ \mathrm{tc} \colon \quad \cC \longrightarrow \cT , \]
    so that $F \circ G \cong \id_\cC$ by the above computation.
    Furthermore, since $e_\mathrm{c} \circ \mathrm{tc}$ sends every short exact sequence in $\cC$ to a distinguished triangle in $D^b(\cT)$, we can use the long exact sequence in cohomology and the fact that $H^i(-) \circ e_\mathrm{c} \circ \mathrm{tc} = 0$ for $i \neq 0$ to conclude that $G$ is an exact functor.
    Using the Künneth formula (cf.\ \cite[Theorem 4.1]{Biglari-Künneth}), and again the fact that $H^i(-) \circ e_\mathrm{c} \circ \mathrm{tc} = 0$ for $i \neq 0$, we further see that $G$ can be enhanced to a monoidal functor (and we even have $F \circ G \cong \id_\cC$ as monoidal functors).
    Now since
    \[ G \circ F \circ e \cong G \circ \imath \cong e \]
    as monoidal functors, the universal property of the monoidal abelian envelope yields $G \circ F \cong \id_\cT$ as monoidal functors.
    We conclude that $\cT \simeq \cC$ as $\kk$-linear monoidal categories, as required.
\end{proof}

\begin{Remark} \label{rem:abelian-envelope}
    In the setting of \Cref{thm:abelian-envelope}, the embedding $\Tilt(\mathcal{C}) \to \mathcal{C}$ may fail to be a monoidal abelian envelope, but it is always a \emph{local abelian envelope} of $\Tilt(\mathcal{C})$ in the sense of \cite{Cou-Kernels}.
    This was pointed out to us by Kevin Coulembier and Andrew Snowden.
    Indeed, by Theorem A in \cite{Cou-Kernels}, the embedding $\Tilt(\mathcal{C}) \to \mathcal{C}$ factors through an exact tensor functor $F \colon \mathcal{T} \to \mathcal{C}$, where $\mathcal{T}$ is a local abelian envelope of $\Tilt(\mathcal{C})$.
    Now the arguments in the proof of \Cref{thm:abelian-envelope} (using the universal property of local abelian envelopes instead of the universal property of monoidal abelian envelopes) show that $\cT \simeq \cC$ as $\kk$-linear monoidal categories.
    An example for the failure of $\Tilt(\cC) \to \cC$ to be a monoidal abelian envelope is provided by the local abelian envelope of the ``second'' Delannoy category constructed in \cite{CS-second-Delannoy}, which appeared shortly after the first version of the present article; see in particular \cite[Remark~2.4]{CS-second-Delannoy}.
\end{Remark}

The earlier literature on monoidal abelian envelopes (preceding \cite{CEOPAbEnvQuotProp}) is often concerned with \emph{symmetric} pseudo-tensor categories, and tensor functors are replaced by \emph{symmetric} tensor functors in the universal property \eqref{eq:mon-ab-env-universal-property} (see \cite{CoulembierMonAbEnv,CEH}).
More precisely, for a symmetric pseudo-tensor category $\cA$, we define a \emph{weak symmetric monoidal abelian envelope} as a pair $(\cT,e)$, where $\cT$ is a faithful symmetric tensor category and $e \colon \cA \to \cT$ is a symmetric tensor functor, such that restriction along $e$ induces an equivalence
\[ \mathrm{sTens}^\mathrm{exact}( \cT , \cT' ) \xrightarrow{~~\sim~~} \mathrm{sTens}^\mathrm{faith}( \cA , \cT' ) \]
between categories of symmetric tensor functors that are exact or faithful, respectively, for every symmetric tensor category $\cT'$ (see also \cite[Section 2.5]{CEOPAbEnvQuotProp}).
A \emph{symmetric monoidal abelian envelope} of $\cA$ is a weak symmetric monoidal abelian envelope $(\cT,e)$ such that the functor $e \colon \cA \to \cT$ is fully faithful.
Using essentially the same arguments as in the proof of \Cref{thm:abelian-envelope} (but keeping track of symmetric braidings), we obtain the following variant for symmetric monoidal abelian envelopes.

\begin{Theorem}
\label{thm:abelian-envelope-symmetric}
    Let $\cC$ be a symmetric tensor category with a lower finite highest weight structure such that $\Tilt(\cC)$ is a monoidal subcategory closed under taking duals.
    If $\Tilt(\cC)$ admits a weak symmetric monoidal abelian envelope, then the canonical embedding $\imath \colon \Tilt(\cC) \to \cC$ is a symmetric monoidal abelian envelope.
\end{Theorem}

In the setting of \Cref{thm:abelian-envelope}, the following criterion can be used in special cases to identify whether $\cC$ is a monoidal abelian envelope of $\Tilt(\cC)$.
(In the symmetric setting of \Cref{thm:abelian-envelope-symmetric}, this is proven in \cite[Proposition~3.2.1]{CEH}.)
\begin{Lemma}
\label{lem:abelian-envelope-tilting-tensor}
    Let $\cC$ be a tensor category with a lower finite highest weight structure such that $\Tilt(\cC)$ is a monoidal subcategory closed under taking duals.
    Further suppose that one of the following conditions holds.
    \begin{enumerate}[(a)]
        \item For every object $X$ of $\cC$, there is a tilting object $T$ of $\cC$ such that $T \o X$ is a tilting object.
        \item The monoidal category $\cC$ is braided, and for every simple object $L$ of $\cC$, there is a tilting object $T$ of $\cC$ such that $T \o L$ is a tilting object.
    \end{enumerate}
    Then the canonical embedding $\Tilt(\cC) \to \cC$ is a monoidal abelian envelope.
\end{Lemma}

\begin{proof}
    First suppose that for every object $X$ of $\cC$, there is a tilting object $T$ of $\cC$ such that $T \o X$ is a tilting object.
    As $X$ is a quotient of $T^* \otimes T \otimes X$, we conclude that every object of $\cC$ is a quotient of a tilting object.
    Now \cite[Theorem~2.2.1]{CEOPAbEnvQuotProp} implies that the canonical embedding $\Tilt(\cC) \to \cC$ is a monoidal abelian envelope.
    Now suppose that $\cC$ is braided and that for every simple object $L$ of $\cC$, there is a tilting object $T$ such that $L \otimes T$ is a tilting object.
    For an object $X$ of $\cC$ with simple composition factors $L_1,\ldots,L_n$, we choose tilting objects $T_1,\ldots,T_n$ such that $T_i \o L_i$ is a tilting object for $i=1,\ldots,n$.
    Then for $T \coloneqq T_1 \otimes \cdots \otimes T_n$, the tensor products $T \otimes L_i$ are tilting objects for $i=1,\ldots,n$ because $\cC$ is braided.
    This implies that $T \otimes X$ has a filtration whose subquotients are tilting objects. 
    Since tilting objects have no higher extensions (see \Cref{lem:Extvanishingstandardcostandard}), this means $T \o X$ is a direct sum of tilting objects, and hence is a tilting object.
    This verifies condition (a), and so the canonical embedding $\Tilt(\cC) \to \cC$ is a monoidal abelian envelope.
\end{proof}

For instance, the hypothesis of \Cref{lem:abelian-envelope-tilting-tensor}(b) about tensor products involving simple objects and tilting objects is known to hold for for categories of representations of algebraic groups (see the proof of Theorem 3.3.1 in \cite{CEH}).

Still in the setting of \Cref{thm:abelian-envelope}, another criterion for deciding whether $\cC$ is a monoidal abelian envelope of $\Tilt(\cC)$ is based on the block structure of $\Tilt(\cC)$.
Following \cite[Section 5.1]{ComesOstrikBlocksRepSt}, a \emph{block} of a $\kk$-linear idempotent-complete additive category $\cA$ is an equivalence class of indecomposable objects of $\cA$ with respect to the finest equivalence relation such that two indecomposable objects $X$ and $Y$ of $\cA$ are equivalent if $\Hom_\cA(X,Y) \neq 0$.
We also call \emph{block} any full subcategory of $\cA$ whose objects are finite direct sums of indecomposable objects in a single block.
A block is called \emph{almost trivial} if it contains a unique indecomposable object of $\cA$ (up to isomorphism) and \emph{trivial} if the endomorphism space of the unique indecomposable object is one-dimensional.

\begin{Proposition}
\label{prop:abelian-envelope-trivial-block}
    Let $\cC$ be a tensor category with a lower finite highest weight structure such that $\Tilt(\cC)$ is a monoidal subcategory closed under taking duals.
    Suppose that $\Tilt(\cC)$ has an almost trivial block.
    Then the canonical embedding $\Tilt(\cC) \to \cC$ is a monoidal abelian envelope.
\end{Proposition}
\begin{proof}
    Let $(\Lambda,\leq)$ be the weight poset of $\cC$ and let $\lambda \in \Lambda$ such that $T(\lambda)$ is the unique indecomposable object in an almost trivial block of $\Tilt(\cC)$.
    Then we have $\Hom_\cC\big( T(\lambda) , T(\mu) \big) = 0$ for all $\mu \in \Lambda \setminus \{ \lambda \}$, and using the equivalence
    \[ K^b\big( \Tilt(\cC) \big) \simeq D^b(\cC) \]
    from \eqref{eq:tiltingequivalence}, it follows that $T(\lambda)$ is injective and projective in $\cC$.
    Since $\cC$ is a rigid monoidal category, the tensor product $T(\lambda) \otimes X$ is injective and projective for every object $X$ of $\cC$, and since injective objects admit $\nabla$-filtrations and projective objects admit $\Delta$-filtrations (see \cite[Theorem 3.56]{BrundanStroppelSemiInfiniteHighestWeight}), we conclude that $T(\lambda) \otimes X$ is a tilting object.
    Now the claim follows from \Cref{lem:abelian-envelope-tilting-tensor}.
\end{proof}

\begin{Remark}
    In the setting of \Cref{prop:abelian-envelope-trivial-block}, every almost trivial block of $\Tilt(\cC)$ is in fact a trivial block.
    Indeed, if $T(\lambda)$ is the unique indecomposable tilting object in an almost trivial block, then for $\mu \in \Lambda \setminus \{ \lambda \}$, we have
    \[ 0 = \dim \Hom_\cC\big( T(\lambda) , T(\mu) \big) = \sum_{\nu \in \Lambda} \big[ T(\lambda) : \Delta(\nu) \big]_\Delta \cdot \big[ T(\mu) : \nabla(\nu) \big]_\nabla \geq \big[ T(\lambda) : \Delta(\mu) \big]_\Delta \geq 0 , \]
    where $[-:\Delta(\nu)]_\Delta$ denotes the multiplicity of $\Delta(\nu)$ in a $\Delta$-filtration (and similarly for $\nabla$-filtrations).
    This implies that $T(\lambda) = \Delta(\lambda)$, hence $T(\lambda)$ has one-dimensional endomorphism space.
\end{Remark}

\begin{Remark}
\label{rem:RepSt-trivial-blocks}
    The hypothesis of \Cref{prop:abelian-envelope-trivial-block} about (almost) trivial blocks is known to hold for the interpolation categories $\RepSt$ by the results of Comes--Ostrik \cite{ComesOstrikBlocksRepSt}, so \Cref{prop:abelian-envelope-trivial-block} and \Cref{thm:knop-highest-weight-envelope} below explain the existence of monoidal abelian envelopes with lower finite highest weight structures for these categories.
\end{Remark}

\section{Sam--Snowden's triangular categories}
\label{sec:triangularcategories}

In \cite{SamSnowdenTriangular}, Sam--Snowden have developed a formalism of \emph{(monoidal) triangular categories},%
\footnote{Not to be confused with triangulated categories.}
which can be thought of as a more categorical variant of (sandwich) cellular algebras \cite{GrahamLehrerCellular,TubbenhauerSandwichCellular} or cellular categories \cite{WestburyCellularCategories}.
In this section, we explain how a monoidal triangular category gives rise to an upper finite highest weight category with a monoidal structure that satisfies the conditions \eqref{eq:Deltatensor} and \eqref{eq:DeltaTor}, and hence via monoidal Ringel duality to a lower finite highest weight category with a compatible monoidal structure.
Throughout the section, we fix an algebraically closed field $\kk$ and an essentially small $\kk$-linear category $\cC$ with finite-dimensional $\Hom$-spaces.
We write $\cC(X,Y) = \Hom_\cC(X,Y)$ for the $\Hom$-spaces in $\cC$ and $|\cC|:=\{|X|:X\in\cC\}$ for the set of isomorphism classes of objects in $\cC$.

\begin{Definition}[{\cite[Definition~4.1]{SamSnowdenTriangular}}]
\label{def:triangularstructure}
A \emph{triangular structure} on $\cC$ is a pair of wide subcategories $(\cU,\cD)$ such that for all $X,Y,Z\in\cC$, we have:

(T1) $\cM(X):=\cU(X,X)=\cD(X,X)$, and this $\kk$-algebra is semisimple.

(T2) There is a partial order $\le$ on $|\cC|$ resulting in a lower finite poset (i.e., the set 
$$
\{|X'|\in|\cC|:|X'|\le|X|\}$$
is finite, for all $X\in\cC$) such that $\cU(X,Y) \neq 0$ implies $|X|\le|Y|$ and $\cD(X,Y) \neq 0$ implies $|Y| \leq |X|$. 

(T3) Composition in $\cC$ induces an isomorphism of vector spaces
$$
\cU(Y,Z) \o_{\cM(Y)} \cD(X,Y) \to \cC(X,Y) .
$$
\end{Definition}

Note that if $(\cU,\cD)$ is a triangular structure on $\cC$, then $(\cD^\mathrm{op},\cU^\mathrm{op})$ is a triangular structure on the opposite category $\cC^\mathrm{op}$, with respect to the same partial order $\leq$ on $|\cC| = |\cC^\mathrm{op}|$.

In \Cref{sec:tensorenvelopes} below, we will use the following criterion to show that certain interpolation categories appearing in \cite{KnopTensorEnvelopes} admit triangular structures.

\begin{Proposition}[{\cite[Proposition~4.33]{SamSnowdenTriangular}}] \label{prop:criterion-triangular} Assume $\cC$ as above satisfies (T1) and (T2) for wide subcategories $\cU$, $\cD$ and a poset structure on $|\cC|$. Assume also there are collections of morphisms $\cC'$, $\cD'$, $\cU'$ such that:
\begin{itemize}
\item The elements from $\cC'$ form a basis in every hom-space in $\cC$.
\item The elements from $\cU'$ and $\cD'$ lie in $\cC'$ and span every hom-space in $\cU$ and $\cD$, respectively.
\item Compositions $\beta\alpha$ of compatible morphisms $\alpha\in\cD'$, $\beta\in\cU'$ are elements of $\cC'$.
\item Any morphism in $\cC'$ factors as a composition $\beta\alpha$, with $\alpha\in\cD'$ and $\beta\in\cU'$. For any second factorization $\beta\alpha=\beta'\alpha'$ of the same morphism, with $\alpha'\in\cD$ and $\beta'\in\cU$, there is an isomorphism $i$ contained both in $\cU$ and $\cD$ such that $\alpha'=i\alpha$ and $\beta'=\beta i^{-1}$.  
\end{itemize}
Then $(\cU,\cD)$ is a triangular structure on $\cC$.
\end{Proposition}

Now suppose that $\cC$ admits a triangular structure $(\mathcal{U},\mathcal{D})$ and fix a partial order $\leq$ as in (T2).
Further let $\cM$ be the wide subcategory of $\cC$ with $\Hom$-spaces $\cM(X,Y) = \cU(X,Y) \cap \cD(X,Y)$, so $\cM(X,Y) \neq 0$ only if $|X| = |Y|$.
Following \cite[Section 4.3]{SamSnowdenTriangular}, we write $\Lambda$ for the set of isomorphism classes of simple objects in $\PSh( \mathcal{M}^\mathrm{op} )$.
For all $\lambda \in \Lambda$, we fix a representative $S_\lambda$ for the isomorphism class $\lambda$ and write $\mathrm{supp}(\lambda) = |X|$ for the unique isomorphism class $|X| \in |\cC|$ of objects in $\cC$ such that $S_\lambda(X) \neq 0$.
Then the partial order $\leq$ on $|\cC|$ gives rise to a partial order $\leq$ on $\Lambda$ via $\lambda < \mu$ if and only if $\mathrm{supp}(\lambda) < \mathrm{supp}(\mu)$ for $\lambda,\mu \in \Lambda$, making $(\Lambda,\leq)$ a lower finite poset.

For all $\lambda \in \Lambda$, Sam--Snowden define in \cite[Section 4.5]{SamSnowdenTriangular} a standard object $\Delta(\lambda)$ in $\PSh^\mathrm{fd}(\cC^\mathrm{op})$ and prove that $\Delta(\lambda)$ has a unique simple quotient $L(\lambda)$.%
\footnote{To see that $\Delta(\lambda)$ belongs to the subcategory $\PSh^\mathrm{fd}(\cC^\mathrm{op})$ of $\PSh(\cC^\mathrm{op})$, one uses \cite[Proposition 4.5]{SamSnowdenTriangular}.}
Furthermore, they show that the objects $L(\lambda)$ for $\lambda \in \Lambda$ are a set of representatives for the isomorphism classes of simple objects in $\PSh(\cC^\mathrm{op})$ and that $[ \Delta(\lambda) : L(\mu) ] = 0$ unless $\lambda \leq \mu$ for $\lambda , \mu \in \Lambda$.
By \cite[Proposition 4.28]{SamSnowdenTriangular}, the simple object $L(\lambda)$ admits a projective cover $P(\lambda)$ in $\PSh(\cC^\mathrm{op})$ and $P(\lambda)$ satisfies the condition $(P\Delta)$ of Subsection~\ref{subsec:highestweightcategories} with respect to the standard objects $\Delta(\mu)$ and the partial order $\leq^\mathrm{op}$ on $\Lambda$.
This proves the following result, which is the main reason for our interest in the formalism of triangular structures.

\begin{Proposition}
 \label{prop:triangularstructurehighestweightcategory}
    If $\cC$ admits a triangular structure $(\mathcal{U},\mathcal{D})$, then $\PSh^\mathrm{fd}(\cC^\mathrm{op})$ is an upper finite highest weight category with weight poset $(\Lambda,\leq^\mathrm{op})$.
\end{Proposition}

In \cite[Section 4.9]{SamSnowdenTriangular}, it is explained that all representable presheaves in $\PSh(\cC^\mathrm{op})$ (i.e.\ presheaves in the image of the Yoneda embedding $\cC^\mathrm{op} \to \PSh(\cC^\mathrm{op})$) are finitely generated projective and that all finitely generated projective objects in $\PSh(\cC^\mathrm{op})$ are direct summands of direct sums of representable presheaves.
In particular, the Yoneda embedding $\cC^\mathrm{op} \to \PSh(\cC^\mathrm{op})$ gives rise to an equivalence
\begin{equation} \label{eq:karoubitriangularprojective}
    \Kar(\cC^\mathrm{op}) \simeq \Proj^\mathrm{fg}\big( \PSh(\cC^\mathrm{op}) \big) ,
\end{equation}
where $\Kar(\cC)$ denotes the Karoubi envelope of $\cC$ as in Subsection~\ref{subsec:Karoubi}.

Assume from now on that $\cC$ is as above, but also monoidal with tensor product $\o$.

\begin{Definition}[{\cite[4.11]{SamSnowdenTriangular}}] A \emph{monoidal triangular structure} on $\cC$ is a triangular structure $(\cU,\cD)$ such that $\cU$ and $\cD$ are closed under tensor products and all associators and unitors of $\cC$ are morphisms both in $\cU$ and $\cD$.
\end{Definition}

Now suppose that $\cC$ admits a monoidal triangular structure.
Then the category $\mathcal{D} \coloneqq \PSh(\cC^\mathrm{op})$ is an upper finite highest weight category by \Cref{prop:triangularstructurehighestweightcategory}, and the Day convolution tensor product (see \Cref{rem:Dayconvolution}) equips $\PSh(\cC^\mathrm{op}) \simeq \mathcal{D}^\mathrm{cc}$ with a monoidal structure such that the tensor product is cocontinuous in both arguments.
Furthermore, since the Yoneda embedding $\cC^\mathrm{op} \to \mathcal{D}^\mathrm{cc}$ is monoidal and induces an equivalence $\Kar(\cC^\mathrm{op}) \simeq \Proj^\mathrm{fg}( \mathcal{D} )$ by \eqref{eq:karoubitriangularprojective}, the subcategory $\Proj^\mathrm{fg}( \mathcal{D} )$ is closed under tensor products and we have
\begin{equation} \label{eq:karoubitriangularprojectivemonoidal}
    \Kar(\cC^\mathrm{op}) \simeq \Proj^\mathrm{fg}( \mathcal{D} )
\end{equation}
as monoidal categories.
In particular, we are in the situation of \Cref{prop:torvanishingYtensor}.
The two following results from \cite{SamSnowdenTriangular} provide criteria for checking the conditions \eqref{eq:Deltatensor} and \eqref{eq:DeltaTor} in $\mathcal{D}^\mathrm{cc}$.

\begin{Proposition}[{\cite[Proposition~3.27]{SamSnowdenTriangular}}]
\label{prop:tensor-exact}
    Assume $\chark(\kk)=0$ and $\cU$ is an essentially small $\kk$-linear monoidal category that has a wide subcategory $\cU'$ such that:
\begin{itemize}
\item The elements from $\cU'$ form a basis in every hom-space in $\cU$.
\item $\cU'$ has finite hom-sets.
\item $\cU'$ is closed under the tensor product $\o$ in $\cU$.
\item For every $X \in \cU'$, there are morphisms $(\varphi_i\:X\to Y_i\o Z_i)_{i\in I}$ in $\cU'$ such that for any $Y,Z\in\cU$,
\begin{enumerate}
    \item[(a)] any morphism $\varphi\:X\to Y\o Z$ in $\cU'$ factors as $\varphi=(\alpha\o\beta)\varphi_i$ for some $i$ and $\alpha,\beta\in\cU'$, and
    \item[(b)] for any two such factorizations $\varphi=(\alpha\o\beta)\varphi_i=(\alpha'\o\beta')\varphi_{i'}$ with $\alpha',\beta'\in\cU'$, we have $i=i'$ and there are automorphisms $\sigma_\alpha\in\Aut_{\cU'}(y_i),\sigma_\beta\in\Aut_{\cU'}(z_i)$ such that 
$$
\alpha'\o\beta'=(\alpha\o\beta)(\sigma_\alpha\o\sigma_\beta)
\text{ and }
\varphi_i = (\sigma_\alpha\o\sigma_\beta)\varphi_i .
$$
\end{enumerate} 
\end{itemize}
Then the Day convolution tensor product on $\PSh(\cU)$ is exact in both arguments.
\end{Proposition}

\begin{Proposition}[{\cite[Proposition 4.32]{SamSnowdenTriangular}}]
\label{prop:monoidaltriangularDeltatensorDeltator}
    Let $\cC$ be an essentially small $\kk$-linear monoidal category with a monoidal triangular structure $(\cU,\cD)$.
    If the Day convolution tensor product on $\PSh(\cD^\mathrm{op})$ is exact in both arguments, then $\PSh^\mathrm{fd}(\cC^\mathrm{op})$ satisfies the conditions \eqref{eq:Deltatensor} and \eqref{eq:DeltaTor}.
\end{Proposition}

Before we state the main result of this section, we need to introduce some additional terminology.

\begin{Definition}
\label{def:orderpreservingduality_triangular}
    For an essentially small $\kk$-linear (monoidal) category $\cA$ with a (monoidal) triangular structure $(\cU,\cD)$ and a partial order $\leq$ on $|\cA|$ as in (T2) of \Cref{def:triangularstructure}, an \emph{order preserving (monoidal) duality} is a $\kk$-linear (monoidal) equivalence $\tau \colon \cA \to \cA^\mathrm{op}$ (respectively $\tau \colon \cA^\mathrm{rev} \to \cA^\mathrm{op}$) such that the map $|X| \mapsto |\tau X|$ is an order automorphism on $|\cA|$ with respect to $\leq$.
\end{Definition}

\begin{Remark}
\label{rem:orderpreservingduality_triangularvsKaroubi}
    Using the notation of \Cref{def:orderpreservingduality_triangular}, note that the indecomposable objects of
    \[ \Kar(\cA) \simeq \Proj^\mathrm{fg}\big( \PSh(\cA) \big) \]
    are canonically labeled by $\Lambda$, cf.\ \eqref{eq:karoubitriangularprojective} and \Cref{prop:triangularstructurehighestweightcategory}.
    An order preserving duality on $\cA$ (in the sense of \Cref{def:orderpreservingduality_triangular}) induces an order preserving duality on $\Kar(\cA)$ (in the sense of \Cref{def:orderpreservingduality}) by the definition of the partial order $\leq$ on $\Lambda$.
\end{Remark}

\begin{Theorem} \label{thm:highest-weight-envelope}
	Let $\cA$ be an essentially small $\kk$-linear monoidal category with a monoidal triangular structure $(\cU,\cD)$, and suppose that the Day convolution tensor product on $\PSh( \cD^\mathrm{op} )$ is exact in both arguments.
    \begin{enumerate}
        \item There is a lower finite highest weight category $\cC$ with a monoidal structure $(\otimes,\mathbf{1},\ldots)$ such that $\otimes$ is right exact in both arguments, $\Tilt(\cC)$ is closed under tensor products and there is a $\kk$-linear monoidal equivalence
        \[ \Tilt(\cC) \simeq \Kar(\cA) . \]
        \item If $\cA$ has a (symmetric) braiding, then there is a (symmetric) braiding on $\cC$ such that the monoidal equivalence $\Tilt(\cC) \simeq \Kar(\cA)$ is braided.
        \item If $\cA$ admits an order preserving monoidal duality, then $\otimes$ is exact in both arguments.
        \item If $\cA$ is rigid and the duality functors
        \[ (-)^* \colon \cA \longrightarrow \cA^\mathrm{op} , \hspace{2cm} \prescript{*}{}{(-)} \colon \cA \longrightarrow \cA^\mathrm{op} \]
        are order preserving dualities, then $\cC$ is rigid and $\Tilt(\cC)$ is closed under taking duals.
    \end{enumerate}
\end{Theorem}
\begin{proof}
    Recall from \Cref{prop:triangularstructurehighestweightcategory} that $\mathcal{B} \coloneqq \PSh^\mathrm{fd}(\cA^\mathrm{op})$ is an upper finite highest weight category.
    Note that $\mathcal{B}$ satisfies the conditions \eqref{eq:Deltatensor} and \eqref{eq:DeltaTor} by \Cref{prop:monoidaltriangularDeltatensorDeltator}, and hence $\mathcal{B}$ satisfies the condition \eqref{eq:Ytensor} by \Cref{prop:torvanishingYtensor}.
    Therefore, the upper finite highest weight category $\mathcal{B}^\mathrm{op}$ satisfies the condition \eqref{eq:Xtensor} by \Cref{prop:oppositeXtensorYtensor}.
    Further recall from \eqref{eq:karoubitriangularprojectivemonoidal} that the Yoneda embedding
    induces a monoidal equivalence $\Kar(\cA) \simeq \Proj^\mathrm{fg}(\mathcal{B}^\mathrm{op})$.
    Now \Cref{thm:monoidalRingelduality_uppertolower} shows that the Ringel dual $\cC \coloneqq \prescript{\vee}{}{\mathcal{B}}^\mathrm{op}$ admits a monoidal structure $(\otimes,\mathbf{1},\ldots)$ such that $\otimes$ is right exact in both arguments, $\Tilt(\cC)$ is a monoidal subcategory and there is a $\kk$-linear monoidal equivalence
    \[ \Tilt(\cC) \simeq \Proj^\mathrm{fg}(\mathcal{B}^\mathrm{op}) \simeq \Kar(\cA) . \]
    If $\cA$ is (symmetrically) braided, then so is $\mathcal{B}^\mathrm{cc}$ by \Cref{rem:Dayconvolutionbraided}, and hence $\cC$ is (symmetrically) braided and $\Tilt(\cC) \simeq \Kar(\cA)$ is a braided monoidal equivalence, again by \Cref{thm:monoidalRingelduality_uppertolower}.
    
    If $\cA$ admits an order preserving monoidal duality, then the latter induces an order preserving monoidal duality on $\Kar(\cA) \simeq \Proj^\mathrm{fg}(\mathcal{B}^\mathrm{op})$ by \Cref{rem:orderpreservingduality_triangularvsKaroubi}.
    As $\mathcal{B}^\mathrm{op}$ satisfies the condition \eqref{eq:Xtensor}, this implies that $\mathcal{B}^\mathrm{op}$ also satisfies the condition \eqref{eq:Ytensor} by \Cref{cor:dualityXtensorYtensor}.
    Hence the tensor product $\otimes$ is biexact by part (2) of \Cref{thm:monoidalRingelduality_uppertolower}.
    Finally, if $\cA$ is rigid and the duality functors are order preserving dualities, then $\Kar(\cA) \simeq \Proj^\mathrm{fg}(\mathcal{B}^\mathrm{op})$ is rigid and the duality functors on $\Proj^\mathrm{fg}(\mathcal{B}^\mathrm{op})$ are order preserving dualities, so part (3) of \Cref{thm:monoidalRingelduality_uppertolower} implies that $\cC$ is rigid and $\Tilt(\cC)$ is closed under taking duals.
\end{proof}

\begin{Remark}
    The hypotheses of \Cref{thm:highest-weight-envelope} were verified by Sam--Snowden for the partition categories (also known as interpolation categories $\RepSt$ for symmetric groups), the Brauer categories (also known as interpolation categories $\uRep( \mathrm{O}_t(\C) )$ for orthogonal groups) and the oriented Brauer categories (also known as interpolation categories $\uRep( \GL_t(\C) )$ for general linear groups), see Sections 5, 6 and 7.1.2 in \cite{SamSnowdenTriangular}.
    Since these categories admit symmetric monoidal abelian envelopes by the results of \cite{ComesOstrikRepabSd,CoulembierMonAbEnv,EntovaAizenbudHinichSerganova-RepGLt}, Theorems~\ref{thm:abelian-envelope-symmetric} and \ref{thm:highest-weight-envelope} explain the existence of lower finite highest weight structures on these symmetric monoidal abelian envelopes.
    For further examples of categories that satisfy the hypoteses of \Cref{thm:highest-weight-envelope}, see \cite[Section 7]{SamSnowdenTriangular} and \Cref{sec:tensorenvelopes} below.
\end{Remark}

\section{Knop's tensor envelopes}
\label{sec:tensorenvelopes}

In this section, we prove that certain interpolation categories constructed by Knop \cite{KnopTensorEnvelopes} can be embedded as monoidal subcategories of tilting objects in lower finite highest weight categories with a monoidal structure (see \Cref{thm:knop-highest-weight-envelope}).
Throughout the section, we fix a category $\cR$ that is finitely complete\footnote{We note that the reference also considers a situation in which $\cR$ is not necessarily finitely complete.}, regular, exact, Mal'cev, and subobject-finite, in the sense of \cite[Section 2 and Appendix A]{KnopTensorEnvelopes}.
Two examples that are relevant for our discussion are the opposite category of the category of finite sets and the category of finite-dimensional vector spaces over a finite field $\F_q$ with $q$ elements.

For simplicity, we refer to monomorphisms in $\cR$ as \emph{injective} and to extremal epimorphisms in $\cR$ as \emph{surjective}.
A \emph{subobject} of an object $Y$ of $\cR$ is an equivalence class of injective homomorphisms $X \to Y$, where two injective homomorphisms $\imath \colon X \to Y$ and $\imath' \colon X' \to Y$ are equivalent if there is an isomorphism $\phi \colon X \to X'$ with $\imath' \circ \varphi = \imath$.
We sometimes write $X\subset Y$ to denote a subobject, suppressing the monomorphism.
Dually, \emph{quotients} are defined as equivalence classes of surjective homomorphisms in $\cR$.
Any homomorphism $f\:X\to Y$ in $\cR$ factors as $f=i \circ e$, where $e$ is surjective and $i$ is injective.
The factorization is unique up to an isomorphism, so the quotient of $X$ defined by $e$ and the subobject of $Y$ defined by $i$ are unique across all such factorizations, and we denote them by $\coim(f)$ and $\im(f)$, respectively.

Let $\cS(\cR)$ denote the class of all quotients in $\cR$. A \emph{degree function} on $\cR$ is a function $\delta\:\cS(\cR)\to\kk$ which is multiplicative with respect to the composition of surjective morphisms in $\cR$, constant along pullbacks in $\cR$, and sends the trivial quotients $\id_X \colon X \to X$ in $\cR$ to $1\in\kk$ (\cite[3.1.~Definition]{KnopTensorEnvelopes}).

Recall that a \emph{span} is a diagram of the form $\cdot\leftarrow\cdot\rightarrow \cdot$ in a given category, and the two homomorphisms are sometimes referred to as \emph{legs} of the span.
For all objects $X$ and $Y$ of $\cR$, let $\Rel_{X,Y}$ be the set of \emph{relations} between $X$ and $Y$, that is,
$$
\Rel_{X,Y} = \{ R: R\subset X\times Y \} .
$$
By the universal property of the product in $\cR$, a relation $R \subset X \times Y$ is uniquely determined by the span $X \leftarrow R \to Y$.
For any object $X$ of $\cR$, there is a distinguished \emph{diagonal subobject} $\Delta_X\subset X\times X$, defined using the universal property of the product.
Let $*$ denote the terminal object of $\cR$.
Then $\Delta_X$ can be viewed as an element of $\Rel_{X\times X,*}$, of $\Rel_{X,X}$, or of $\Rel_{*,X\times X}$.

The following is a special case of \cite[3.2.~Definition]{KnopTensorEnvelopes}.
\begin{Definition}
    Let $\cT^0(\cR,\delta)$ be the $\kk$-linear monoidal category with
    \begin{itemize}
        \item objects $[X]$, labeled by objects $X\in\cR$;
        \item homomorphisms given by $\Hom_{\cT^0(\cR,\delta)}([X],[Y])=\kk\Rel_{X,Y}$;
        \item identity homomorphisms given by $\Delta_X\in\Rel_{X,X}$;
        \item compositions constructed using fiber products, image factorizations, and the degree function as follows: for all $R_1\subset X\times Y$ and $R_2\subset Y\times Z$, set $f:=(R_1\times_Y R_2\to R_1\times R_2\to X\times Z)$, then 
        $$
        R_2 \circ R_1 = \delta(\coim(f)) \cdot \im(f) \in \kk\Rel_{X,Z} ;
        $$
        \item the tensor unit $\one=[*]$;
        \item tensor products of objects or morphisms given by products of objects in $\cR$ or of subobjects in $\cR$, respectively.
    \end{itemize}
\end{Definition}

\begin{Remark} \label{rem:tensorenveloperigidsymmetric}
    As explained in the paragraphs following Theorem 3.4 in \cite{KnopTensorEnvelopes}, $\cT^0(\cR,\delta)$ is a $\kk$-linear rigid symmetric monoidal category, and all objects of $\cT^0(\cR,\delta)$ are self-dual.
    Furthermore, $\cT^0(\cR,\delta)$ has finite-dimensional $\Hom$-spaces (because $\cR$ is subobject-finite).
    The endomorphism space of $\mathbf{1} = [*]$
    is one-dimensional if and only if the terminal object $*$ has no non-trivial subobjects in $\cR$ by \cite[Proposition 3.5]{KnopTensorEnvelopes}.
\end{Remark}

\begin{Definition}
\label{def:tensor-envelope}
The \emph{tensor envelope} $\cT(\cR,\delta)$ of $\cR$ is defined as the Karoubi envelope
\[ \cT(\cR,\delta) \coloneqq \Kar\big( \cT^0(\cR,\delta) \big) . \]
\end{Definition}

In particular, $\cT(\cR,\delta)$ is a $\kk$-linear additive idempotent complete rigid symmetric monoidal category with finite-dimensional $\Hom$-spaces.
If the terminal object in $\cR$ has no non-trivial subobjects, then the tensor unit in $\cT(\cR,\delta)$ has a one-dimensional endomorphism space by \cite[Proposition 3.5]{KnopTensorEnvelopes}, so that $\cT(\cR,\delta)$ is a pseudo-tensor category in the sense of Subsection~\ref{subsec:abelian-envelopes}.

\begin{Example} \label{ex:RepSt}
Let $\cR$ be the opposite category of the category of finite sets. Then the objects in $\cT^0(\cR,\delta)$ are finite sets, and the relations $\Rel_{X,Y}$, for finite sets $X,Y$, are the (set) partitions of the disjoint union $X\sqcup Y$. A morphism in this category can be represented by a partition diagram like this:
$$
\tp{1,3,4,0,2,12,0,11,0,13,15,0,14} ,
$$
where the upper points represent the set $X=\{1,2,3,4\}$, the lower points represent the set $Y=\{1',2',3',4',5'\}$, and the connected components of the diagram represent a partition of the set $X\sqcup Y$, in this case
$$
\{\{1,3,4\},\{2,2'\},\{1'\},\{3',5'\},\{4'\}\} .
$$
Any degree function $\delta$ in this case is determined by its value $t\in\kk$ for the subobject $\emptyset\hookrightarrow *$ in the category of finite sets, viewed as a quotient in the opposite category of finite sets. When composing two partition diagrams, the diagrams are stacked horizontally, the common middle points are removed, and any connected component that only involves middle points is evaluated to the scalar $t$.

The tensor envelope constructed in this case is Deligne's interpolation category for the symmetric groups $\RepSt$, see \cite{Deligne-RepSt} and the examples following Definition 9.10 in \cite{KnopTensorEnvelopes}.
\end{Example}

In \cite[Section 5]{KnopTensorEnvelopes} a \emph{core decomposition} of relations is described. We verify some of its properties. In particular, we show that it yields a triangular structure, in the sense of Sam--Snowden \cite{SamSnowdenTriangular}, on Knop's tensor envelopes.

\begin{Definition}
We define three collections of homomorphisms in $\cT^0(\cR,\delta)$:
$$
\Rel := \big\{ R\in\Rel_{X,Y}\subset\Hom_{\cT^0(\cR,\delta)}([X],[Y]) \mathop{\big|} X,Y\in\cR \big\},
$$  
$$
\cD_\Rel := \big\{ R\in\Rel_{X,Y}: X,Y\in\cR, R\to X\times Y\to X \text{ injective}, R\to X\times Y\to Y \text{ surjective} \big\} ,
$$
$$
\cU_\Rel := \big\{ R\in\Rel_{X,Y}: X,Y\in\cR, R\to X\times Y\to X \text{ surjective}, R\to X\times Y\to Y \text{ injective} \big\}  .
$$
\end{Definition}

\begin{Example}
Recall from \Cref{ex:RepSt} that in $\RepSt$, the homomorphisms in $\Rel$ can be identified with partition diagrams.
Under this identification, the homomorphisms in $\cD_\Rel$ correspond to partition diagrams all of whose connected components have at least one upper point and at most one lower point (note the the roles of injective and surjective morphisms are swapped in the opposite category of finite sets).
Dually, the homomorphism $\cU_\Rel$ corresponds to partition diagrams all of whose connected components have at least one lower point and at most one upper point.
For instance, the following partition diagram corresponds to a homomorphism in $\cD_\Rel$:
$$
\tp{1,2,13,0,3,5,0,6,0,4,14} .
$$
\end{Example}

\begin{Lemma} \label{lem:U-D-basics}
\begin{enumerate}[(a)]
    \item The sets of homomorphisms $\cD_\Rel$ and $\cU_\Rel$ are each closed under both the composition and the tensor product in $\cT^0(\cR,\delta)$.
    \item We have $g \circ f\in\Rel$ for all $f\in\cD_\Rel$ and $g\in\cU_\Rel$.
    \item (cf.~\cite[5.2.~Lemma]{KnopTensorEnvelopes}) Any $h\in\Rel$ has an essentially unique factorization $h=f \circ g$ with $f\in\cU_\Rel$ and $g\in\cD_\Rel$, in the sense that if $h=f' \circ g'$ is another factorization with $f'\in\cU_\Rel$ and $g'\in\cD_\Rel$, then there is an isomorphism $m$ contained both in $\cU_\Rel$ and $\cD_\Rel$ such that $f'=f \circ m^{-1}$ and $g=m \circ g'$.
\end{enumerate}
\end{Lemma}

\begin{proof}
(a) That $\cD_\Rel$ and $\cU_\Rel$ are closed under composition follows from the fact that injective and surjective homomorphisms are stable under arbitrary pullbacks in $\cR$, since $\cR$ is regular. 
That $\cD_\Rel$ and $\cU_\Rel$ are closed under tensor products follows from to the fact that both injective and surjective homomorphisms are closed under forming products in $\cR$, again since $\cR$ is regular.

(b) By the assumptions, the relations $f$ and $g$ can be represented by monomorphisms
\[ f \colon R \xrightarrow{~(f_X,f_Y)~} X\times Y , \hspace{2cm} g \colon S \xrightarrow{~(g_Y,g_Z)~} Y\times Z \]
such that $f_X$ and $g_Z$ are injective and $f_Y$ and $g_Y$ are surjective.
Recall that $g \circ f = \delta( \mathrm{coim}(h) ) \cdot h$, where $h$ is the homomorphism
\[ h \colon R\times_Y S\to R\times S\xrightarrow{f_X\times g_Z}X\times Z . \]
As $h$ is the composition of two injective homomorphisms in $\mathcal{R}$, we have $\delta(\mathrm{coim}(h)) = 1 \in \kk$ and therefore $g \circ f = h \in \mathrm{Rel}$, as claimed.

(c) This follows from Lemmas~\ref{lem:reduction-exists} and \ref{lem:reduction-unique}.
\end{proof}

\begin{Remark}
    While $\cU_\Rel$ and $\cD_\Rel$ are closed under composition by \Cref{lem:U-D-basics}(a), the composition of two homomorphisms in $\Rel$ may be a non-trivial scalar multiple of a homomorphism in $\mathrm{Rel}$.
\end{Remark}

We define a relation $\leq$ on the isomorphism classes $|\cT^0(\cR,\delta)|$ of the category $\cT^0(\cR,\delta)$ by setting $|[X]| \le |[Y]|$ if and only if $X$ is a subquotient of $Y$ in $\cR$, that is, if there is a diagram $Y\hookleftarrow\cdot\twoheadrightarrow X$ in $\cR$.
To alleviate notation, we set $|X|:=|[X]|$ for all $X\in\cR$.

\begin{Lemma} \label{lem:partial-order-knop} 
(a) In any span $X\hookleftarrow\cdot\twoheadrightarrow X$ in $\cR$ consisting of an injective and a surjective morphism as indicated, both legs have to be isomorphisms.

(b) The relation $\le$ above is a partial order.
\end{Lemma}
    
\begin{proof}
(a) This is shown in \cite[2.6.~Lemma]{KnopTensorEnvelopes}, which uses that $\cR$ is subobject-finite.

(b) Clearly, $|X|\le|X|$ for all $X$. Transitivity is shown in \cite[2.5.~Lemma]{KnopTensorEnvelopes}, this can be seen by forming a suitable pullback. Assume $|X|\le |Y|$ and $|Y|\le |X|$, for $X,Y\in\cR$. Then $|X|\le|Y|\le|X|$, i.e., there are morphisms $f$, $g$, $f'$, $g'$ in $\cR$ as in the following diagram:
$$
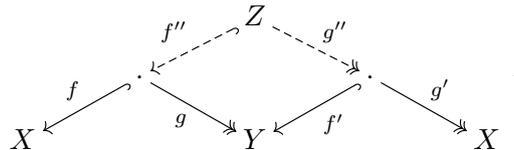
\begin{tikzcd}[column sep=1cm, row sep=0.3cm,remember picture]
 && Z 
    \ar[ld,dashed,hookrightarrow,"f''",swap]
    \ar[rd,dashed,two heads,"g''"] \\
 &  \cdot 
    \ar[ld,hookrightarrow,"f",swap] 
    \ar[rd,two heads,"g",swap] 
 && \cdot 
    \ar[ld,hookrightarrow,"f'"] 
    \ar[rd,two heads,"g'"] \\
 X && Y && X
\end{tikzcd} .
$$
We can form a pullback, indicated by the two dashed arrows, producing a subquotient diagram $X\hookleftarrow Z\twoheadrightarrow X$. Both legs in this span have to be isomorphism by (a). Then $f$ and $g'$ have to be isomorphisms: For instance, $f$ has a right inverse $r$, so $frf=f$, so $rf=\id$, as $f$ is injective, so $r$ is a two-sided inverse of $f$, and dually $g''$ is an isomorphism, which implies that $g'$ is also an isomorphism. The same argument applied to the chain of relations $|Y|\le|X|\le|Y|$ yields that the morphisms called $f'$ and $g$ in the above diagram have to be isomorphisms. Hence, $|X|=|Y|$.
\end{proof}

We now consider the $\kk$-linear wide subcategories $\kk \cU_\mathrm{Rel}$ and $\kk \cD_\mathrm{Rel}$ of $\cT^0(\cR,\delta)$ whose $\Hom$-spaces are spanned by homomorphisms in $\cU_\mathrm{Rel}$ and $\cD_\mathrm{Rel}$, respectively.
Note that $\kk \cU_\mathrm{Rel}$ and $\kk \cD_\mathrm{Rel}$ are indeed subcategories of $\cT^0(\cR,\delta)$, since the collections of homomorphisms $\cU_\mathrm{Rel}$ and $\cD_\mathrm{Rel}$ are closed under composition by \Cref{lem:U-D-basics}.

\begin{Proposition} \label{prop:knop-monoidal-triangular}
    Suppose $\kk$ is of characteristic zero.
    Then $(\kk\cU_\Rel,\kk\cD_\Rel)$ is a monoidal triangular structure in the sense of \cite{SamSnowdenTriangular} on $\cT^0(\cR,\delta)$, with respect to the above poset structure.
\end{Proposition}
\begin{proof}
    Recall that $\cT^0(\cR,\delta)$ is an essentially small $\kk$-linear monoidal category with finite-dimensional $\Hom$-spaces by \Cref{rem:tensorenveloperigidsymmetric}.
    Set $\cU:=\kk\cU_\Rel$ and $\cD:=\kk\cD_\Rel$.
    We will use the criterion from \Cref{prop:criterion-triangular} to show that $(\cU,\cD)$ defines a (monoidal) triangular structure on $\cT^0(\cR,\delta)$ and start by verifying that $\cU$ and $\cD$ satisfy the conditions (T1) and (T2) from \Cref{def:triangularstructure}.
    
    For any object $X$ of $\cR$, the endomorphism spaces $\cU([X],[X])$ and $\cD([X],[X])$ are given by linear combinations of spans that exhibit $X$ as a subquotient of itself.
    The legs of these spans have to be isomorphisms by \Cref{lem:partial-order-knop}(a), and so we have
    \[ \cU([X],[X]) = \cD([X],[X]) \cong \kk\Aut_\cR(X) , \]
    where $\Aut_\cR(X)$ is the group of automorphisms of $X$ in $\cR$.
    This group is finite, as any automorphism is determined by its graph (considered as a subobject of $X\times X$) and as $\cR$ is subobject-finite by assumption.
    Since $\kk$ is of characteristic $0$, the group algebra $\kk \Aut_\cR(X)$ is semisimple, verifying (T1).

    Next observe that the partial order $\leq$ from \Cref{lem:partial-order-knop} is lower finite, because every object of $\cR$ has finitely many subobjects and finitely many quotients (since every quotient $Y$ of an object $X$ is determined by the subobject $X\times_Y X$ of $X\times X$).
    For objects $X$ and $Y$ of $\cR$, we have $\cU([X],[Y]) = 0$ unless $|X| \leq |Y|$ because every span $X \twoheadleftarrow \cdot \hookrightarrow Y$ in $\cU_\mathrm{Rel}$ exhibits $X$ as a subquotient of $Y$.
    Similarly, we have $\cD([X],[Y]) = 0$ unless $|X| \geq |Y|$, and this verifies (T2).

    Now the subsets $\cU' \coloneqq \cU_\mathrm{Rel} \subseteq \cU$ and $\cD' \coloneqq \cD_\mathrm{Rel} \subseteq \cD$ satisfy the hypotheses of \Cref{prop:criterion-triangular} by \Cref{lem:U-D-basics}, hence $(\cU,\cD)$ defines a triangular structure on $\cT^0(\cR,\delta)$.
    As $\cU$ and $\cD$ are closed under tensor products (again by \Cref{lem:U-D-basics}), the triangular structure $(\cU,\cD)$ is monoidal.
\end{proof}

For the interpolation category $\RepSt$, a monoidal triangular structure was already defined by Sam--Snowden \cite[Section 6.6]{SamSnowdenTriangular} using essentially the same strategy as the proof of \Cref{prop:knop-monoidal-triangular}.

\begin{Remark}
\label{rem:knop-triangular-U-Dop}
    There is a $\kk$-linear monoidal functor $\cT^0(\cR,\delta) \to \cT^0(\cR,\delta)^\op$ that is the identity on objects and sends $R \in\Rel_{X,Y}$ to the same subobject, but viewed as an element in $\Rel_{Y,X}$.
    This functor induces a $\kk$-linear monoidal equivalence $\kk\cU_\Rel\simeq \kk\cD_\Rel^\op$.
\end{Remark}

As $\cT^0(\cR,\delta)$ admits a triangular structure $(\kk\cU_\Rel,\kk\cD_\Rel)$ by \Cref{prop:knop-monoidal-triangular}, the category
\[ \mathcal{D} = \PSh^\mathrm{fd}\big( \cT^0(\cR,\delta) \big) \]
is an upper finite highest weight category by \Cref{prop:triangularstructurehighestweightcategory}.
Furthermore, the category
\[ \mathcal{D}^\mathrm{cc} \simeq \PSh\big( \cT^0(\cR,\delta) \big) \]
can be equipped with a monoidal structure via Day convolution.
In order to also define a monoidal structure on the Ringel dual of $\mathcal{D}$ (as in \Cref{thm:highest-weight-envelope}), we next want to show that the Day convolution tensor product on $\PSh( \kk \cD_\mathrm{Rel}^\mathrm{op} )$ is exact in both arguments, using the criterion from \Cref{prop:tensor-exact}.
To that end, we will consider the following collection of relations in $\cR$.

\begin{Definition}
    For all $X\in\cR$, define $\Phi_X$ as the collection of relations $R\in\Rel_{X, Y\times Z}$
    such that the induced morphisms 
    $$ R\to X ,\quad R\to Y, \quad R\to Z$$
    are all surjective, and the induced maps
    \[ R\to X\times Y ,\quad R\to X\times Z, \quad R\to Y\times Z \]
    are all injective.
\end{Definition}

Note that, in particular, all relations in $\Phi_X$, viewed as morphism with source $[X]$, lie in $\cU_\Rel$.

\begin{Lemma} \label{lem:tri-decomposition} Consider $X,Y,Z\in\cR$ and $f\in\cU_\Rel([X],[Y]\o[Z])$.

(a) There are $\alpha,\beta\in\cU_\Rel$ and $\phi\in\Phi_X$ such that $f=(\alpha\o\beta)\phi$ in $\cT^0(\cR)$.

(b) If $f=(\alpha'\o\beta')\phi'$ is a second factorization as in (a), then there are isomorphisms $\sigma_\alpha,\sigma_\beta\in\cU_\Rel$ such that
$$
\alpha'= \alpha \sigma_\alpha^{-1} ,\quad
\beta' = \beta' \sigma_\beta^{-1} ,\quad
\phi'=(\sigma_\alpha\o\sigma_\beta)\phi = \phi .
$$
\end{Lemma}

\begin{proof} This is a reformulation of \Cref{cor:reduced-3-relation}, where the elements $f$ are viewed as $1$-reduced $3$-ary relations and the elements in $\Phi_X$ correspond to $i$-reduced $3$-ary relations for $i=1,2,3$.
\end{proof}

\begin{Example} In $\RepSt$, such factorizations look as follows, where the lower points left and right of the separator correspond to the tensor factors $[Y]$ and $[Z]$, respectively:
$$
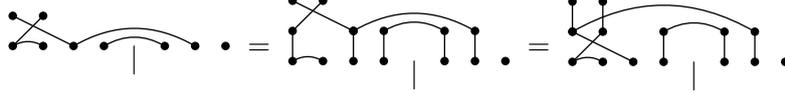
\tp[{\node[below=-4pt] at \makePartPt{15} {$\mid$};}]{1,13,17,0, 2,11,12,0, 14,16,0,18}
= 
\tp[{\node[below=-4pt] at \makePartPt{25} {$\mid$};}]{1,13,17,27,0,13,23,0, 2,11,21,22,0, 24,14,16,26,0, 28}
= 
\tp[{\node[below=-4pt] at \makePartPt{25} {$\mid$};}]{1,11,17,27,0,11,23,0, 2,12,21,22,0, 24,14,16,26,0, 28}
$$    
Note that the two factorizations in the second and third expression are related to each other by isomorphisms, as required.
\end{Example}

\begin{Corollary} \label{cor:knop-tensor-exact}
The Day convolution tensor product on $\PSh(\kk \cD_\Rel^\op)$ is exact in both arguments.
\end{Corollary}

\begin{proof}
    By Lemmas~\ref{lem:U-D-basics} and \ref{lem:tri-decomposition}, the subset $\cU' \coloneqq \cU_\Rel \subseteq \kk\cU_\Rel$ satisfies the hypotheses of \Cref{prop:tensor-exact}, and it follows that the Day convolution tensor product on $\PSh( \kk \cU_\Rel )$ is exact in both arguments.
    Using the $\kk$-linear monoidal equivalence $\kk \cD_\Rel^\mathrm{op} \simeq \kk \cU_\Rel$ from \Cref{rem:knop-triangular-U-Dop}, we conclude that the Day convolution tensor product on $\PSh(\kk \cD_\Rel^\mathrm{op})$ is also exact in both arguments.
\end{proof}

For the interpolation category $\RepSt$, \Cref{cor:knop-tensor-exact} was shown in \cite[Proposition~6.5]{SamSnowdenTriangular} using essentially the same strategy.

Now we are ready to prove the main result of this section.

\begin{Theorem}
\label{thm:knop-highest-weight-envelope}
    Suppose that $\kk$ is of characteristic zero and let $\cR$ be a subobject-finite regular Mal'cev category with a degree function $\delta$.
	Then there is a lower finite highest weight category $\cC = \cC(\mathcal{R},\delta)$ with a rigid symmetric monoidal structure such that $\Tilt(\cC)$ is a monoidal subcategory closed under taking duals and there is a $\kk$-linear symmetric monoidal equivalence
    \[ \cT(\mathcal{R},\delta) \simeq \Tilt(\cC) . \]
\end{Theorem}

\begin{proof}
    Recall from \Cref{prop:knop-monoidal-triangular} that there is a monoidal triangular structure $( \kk\mathcal{U}_\mathrm{Rel} , \kk\cD_\mathrm{Rel} )$ on $\cT^0(\cR,\delta)$ and from \Cref{cor:knop-tensor-exact} that the Day convolution tensor product on $\PSh( \kk\mathcal{D}_\mathrm{Rel}^\mathrm{op} )$ is exact in both arguments.
    Further note that $\cT^0(\cR,\delta)$ is a rigid symmetric monoidal category and that the duality functor $(-)^* \colon \cT^0(\cR,\delta) \to \cT^0(\cR,\delta)^\mathrm{op}$ is an order preserving duality (in the sense of \Cref{def:orderpreservingduality_triangular}) by \Cref{rem:tensorenveloperigidsymmetric}.
    Now \Cref{thm:highest-weight-envelope} implies that there is a lower finite highest weight category $\cC$ with a rigid symmetric monoidal structure such that $\Tilt(\cC)$ is a monoidal subcategory closed under taking duals and
    \[ \Tilt(\cC) \simeq \Kar\big( \cT^0(\cR,\delta) \big) = \cT(\cR,\delta) \]
    as $\kk$-linear symmetric monoidal categories.
\end{proof}

If $\mathcal{R}$ is a regular category whose terminal object has no non-trivial subobjects, then the endomorphism space of the tensor unit in $\cT(\cR,\delta)$ is one-dimensional by \cite[Proposition 3.5]{KnopTensorEnvelopes}.
This implies that $\cT(\cR,\delta)$ is a pseudo-tensor category and that the category $\cC(\cR,\delta)$ from \Cref{thm:knop-highest-weight-envelope} is a tensor category, in the sense of Subsection~\ref{subsec:abelian-envelopes}.
In this situation, it is natural to ask whether $\cC(\cR,\delta)$ is a monoidal abelian envelope of $\cT(\cR,\delta)$.

\begin{Corollary}
    \label{cor:knop-mon-ab-env}
    Suppose that $\kk$ is of characteristic zero and let $\cR$ be a subobject-finite regular Mal'cev category with a degree function $\delta$.
	Further suppose that the terminal object of $\mathcal{R}$ has no non-trivial subobjects.
    If $\cT(\cR,\delta)$ admits a weak (symmetric) monoidal abelian envelope, then $\cC(\cR,\delta)$ is a (symmetric) monoidal abelian envelope of $\cT(\cR,\delta)$.
\end{Corollary}
\begin{proof}
    This follows from Theorems~\ref{thm:abelian-envelope}, \ref{thm:abelian-envelope-symmetric} and \ref{thm:knop-highest-weight-envelope}.
\end{proof}

\begin{Corollary}
    For the interpolation categories $\RepSt$ and $\uRep(\GL_t(\mathbb F_q))$, the lower finite highest weight category $\cC$ from \Cref{thm:knop-highest-weight-envelope} is a (symmetric) monoidal abelian envelope.
\end{Corollary}

\begin{proof}
    Both $\RepSt$ and $\uRep(\GL_t(\mathbb F_q))$ are instances of Knop's tensor envelopes $\cT(\cR,\delta)$, where $\mathcal{R}$ is the opposite category of finite sets or the category of finite-dimensional $\F_q$-vector spaces, respectively.
    In both cases, the terminal object of $\mathcal{R}$ (the empty set or the trivial vector space) have no non-trivial subobjects.
    Therefore, it suffices (by \Cref{cor:knop-mon-ab-env}) to observe that $\RepSt$ and $\uRep(\GL_t(\mathbb F_q))$ have (symmetric) monoidal abelian envelopes.
    For $\RepSt$, this follows from the results of Comes--Ostrik \cite{ComesOstrikRepabSd} (in the symmetric case) and from \Cref{rem:RepSt-trivial-blocks}.
    For $\uRep(\GL_t(\mathbb F_q))$, this follows from \cite[Theorem~15.15]{HS-oligo} in the symmetric case, but the same proof, combined with \cite[Theorem 2.2.1]{CEOPAbEnvQuotProp}, implies that the symmetric monoidal abelian envelope of $\uRep(\GL_t(\mathbb F_q))$ constructed in \cite{HS-oligo} is also a monoidal abelian envelope.
\end{proof}

More generally, the existence of monoidal abelian envelopes for Knop's tensor envelopes can be obtained via a comparison with certain categories constructed from pro-oligomorphic groups by Harman--Snowden in \cite{HS-oligo}. 

Assume for the remainder of this section that $\kk$ is of characteristic $0$ and that the terminal object of $\cR$ has no non-trivial subobjects.
As is explained in \cite{Snow-regular},
there is a pro-oligomorphic group $G$ with a so-called \emph{measure} $\mu$ such that the Karoubi envelope of a category $\uPerm(G,\mu)$ of ``permutation modules'' constructed in \cite[Definition 8.1]{HS-oligo} contains $\cT(\cR,\delta)$ as a full $\kk$-linear monoidal subcategory.

\begin{Remark}
\label{rem:Knop-oligomorphic-Snowden}
    More specifically, \cite[Theorem~7.5]{Snow-regular} shows that there is a pro-oligomorphic group $G$ associated to $\cR$, a full subcategory $S(\cR)$ of a certain category $S(G)$ of $G$-sets \cite[Section 2.3]{HS-oligo}, and a measure $\mu$ on $S(\cR)$ in the sense of \cite[Definition 6.4]{Snow-regular} associated to $(\cR,\delta)$, such that
    \[ \cT(\cR,\delta) \simeq \Kar\big( \uPerm(S(\cR),\mu) \big) \]
    as $\kk$-linear symmetric monoidal categories, where $\uPerm(S(\cR),\mu) \big)$ is as in \cite[Construction 7.2]{Snow-regular}.
    Now the category $\uPerm(S(\cR),\mu)$ embeds as a full $\kk$-linear monoidal subcategory in a category of ``permutation modules'' for $G$, in the sense of \cite{HS-oligo}, by the following argument, which was explained to us by Andrew Snowden:
    By \cite[Proposition~4.4]{Snow-regular}, every object of $S(G)$ has a finite cover by an object of $S(\cR)$.
    Using \cite[Section~3.4(d)]{HS-oligo} (proved in \cite[Section 2.6]{NS}), it follows that the measure $\mu$ on $S(\cR)$ can be extended to a measure $\mu$ on $S(G)$, that is, to a measure on $G$.
    Then $\uPerm(S(\cR),\mu)$ is a full $\kk$-linear monoidal subcategory of $\uPerm(G,\mu)$ by \cite[Construction~7.2]{Snow-regular} and  \cite[Definition~8.1, Section~8.3]{HS-oligo}, and therefore $\cT(\cR,\delta) \simeq \Kar\big( \uPerm(S(\cR),\mu) \big)$ is a full $\kk$-linear monoidal subcategory of $\Kar\big( \uPerm(G,\mu) \big)$.

\end{Remark}

It is shown in \cite{HS-oligo} that if a measure $\mu$ on a pro-oligomorphic group $G$ is \emph{quasi-regular} and has the so-called \emph{property (P)}, then the category $\Kar(\uPerm(G,\mu))$ has a (symmetric) monoidal abelian envelope $\cC' = \uRep(G,\mu)$.
More precisely, it is shown in \cite[Theorem 13.13]{HS-oligo} that $\uPerm(G,\mu)$ admits a symmetric monoidal abelian envelope $\cC' = \uRep(G,\mu)$ if $\mu$ is regular and satisfies (P), but the same proof, combined with \cite[Theorem 2.2.1]{CEOPAbEnvQuotProp}, shows that $\cC'$ is also a monoidal abelian envelope.

\begin{Proposition} \label{prop:comparison-Knop-HS}
    Assume $\kk$ is of characteristic $0$ and the measure $\mu$ associated to $(\cR,\delta)$ is quasi-regular and has property (P).
    Then the lower finite highest weight category $\cC(\cR,\delta)$ from \Cref{thm:knop-highest-weight-envelope} is a (symmetric) monoidal abelian envelope of $\cT(\cR,\delta)$, and agrees with the (symmetric) monoidal abelian envelope $\cC' = \uRep(G,\mu)$ of $\Kar(\uPerm(G,\mu))$ as constructed in \cite{HS-oligo}.
\end{Proposition}

\begin{proof}
    By the proof of Theorem~13.13 in \cite{HS-oligo}), there is a fully faithful $\kk$-linear symmetric monoidal functor
    \[ \Phi \colon \uPerm(G,\mu) \longrightarrow \cC' . \]
    Following \cite[Definition~8.1]{HS-oligo} and \cite[Construction~7.2]{Snow-regular}, the objects of the categories $\uPerm(G,\mu)$ and $\uPerm(S(\cR),\mu)$ are formal symbols $\mathrm{Vec}_X$, for objects $X$ of $S(G)$ or $S(\cR)$, respectively, and as in \cite[Proposition~10.13]{HS-oligo}, we write $C(X) = \Phi( \mathrm{Vec}_X )$. 

    Now \cite[Proposition~4.4]{Snow-regular} implies that for every object $X$ of $S(G)$, there is an object $Y$ of $S(\cR)$ with a surjective homomorphism $Y \to X$ in $S(G)$.
    As the measure $\mu$ is quasi-regular, it is \emph{normal} in the sense of \cite[Definition~11.1]{HS-oligo} by \cite[Proposition~11.3]{HS-oligo}, which implies that $C(X)$ is a quotient of $C(Y)$ in $\cC'$.
    As every object of $\cC'$ is a quotient of an object of the form $C(X)$ for an object $X$ of $S(G)$ by \cite[Proposition~11.11]{HS-oligo}, we conclude that every object of $\cC'$ is a quotient of an object of $\cT(\cR,\delta) \simeq \mathrm{Kar}\big( \uPerm( S(\cR) , \mu ) \big)$ (see \Cref{rem:Knop-oligomorphic-Snowden}).
    This implies that $\cC'$ is a (symmetric) monoidal abelian envelope of $\cT(\cR,\delta)$ by Theorem~2.2.1 and Lemma 2.5.1 in \cite{CEOPAbEnvQuotProp}, and so we have $\cC(\cR,\delta) \simeq \cC'$ as $\kk$-linear (symmetric) monoidal categories by Theorems~\ref{thm:abelian-envelope} and \ref{thm:abelian-envelope-symmetric}.
\end{proof}

\section{Tensor structures for affine Lie algebras at positive levels}
\label{sec:affineLiealgebras}

In the celebrated series of papers \cite{KL12,KL34}, Kazhdan and Lusztig have defined monoidal structures on certain categories of representations of affine Lie algebras at negative or non-rational levels.
In this section, we explain how lower to upper monoidal Ringel duality (\Cref{thm:monoidalRingelduality_lowertoupper}) can be used to define monoidal structures on categories of representations of affine Lie algebras at positive rational levels, generalizing recent results obtained by McRae--Yang \cite{McRaeYang} for the affine Lie algebra $\mathfrak{\widehat{sl}}_2(\C)$.
Using monoidal Ringel duality, the existence of monoidal structures at positive levels is an easy consequence of the fact (well known to experts) that the Arkhipov--Soergel duality \cite{ArkhipovSemiInfiniteAssociative,SoergelCharakterformeln} between negative-level and positive-level representations of affine Lie algebras can be seen as an instance of Ringel duality; cf.\ \cite[Section 6.3]{BrundanStroppelSemiInfiniteHighestWeight}.

Let $\mathfrak{g}$ be a complex simple Lie algebra with
root system $\Phi$ and weight lattice $X$.
We fix a set $\Pi \subseteq \Phi$ of simple roots and write $X^+ \subseteq X$ for the corresponding set of dominant weights.
For $\lambda,\mu \in X$, we write $\lambda \leq \mu$ if $\mu - \lambda$ is a non-negative integer linear combination of simple roots.

Consider the loop algebra $\mathcal{L}\mathfrak{g} = \mathfrak{g} \otimes \C[t,t^{-1}]$ and the affine Lie algebra $\aff{g} = \mathcal{L} \mathfrak{g} \oplus \C c$, which is a central extension of $\mathcal{L}\mathfrak{g}$ with commutator defined by $[c,\aff{g}] = 0$ and
\[ [ x \otimes t^m , y \otimes t^n ] = [x,y] \otimes t^{m+n} + \delta_{m,-n} \cdot m \cdot \langle x , y \rangle \cdot c \]
for $x,y \in \mathfrak{g}$ and $m,n \in \Z$, where $\langle - \,, - \rangle$ denotes the Killing form on $\mathfrak{g}$.
We further define the Lie subalgebra $\aff{p} = \mathfrak{g} \otimes \C[t] \oplus \C c$.
A $\aff{g}$-module $M$ is said to be of \emph{level} $\kappa \in \C$ if the central element $c$ acts on $M$ via multiplication with the scalar $\kappa - h^\vee$, where $h^\vee$ is the dual Coxeter number of $\mathfrak{g}$.
Following \cite[Section 7]{SoergelCharakterformeln}, we consider the category $\mathcal{O}_\kappa$ of $\aff{g}$-modules of level $\kappa$ on which $\aff{p}$ acts locally finitely and $\mathfrak{g} \otimes t \C[t]$ acts locally nilpotently.
Via the unique homomorphism $\aff{p} \to \mathfrak{g}$ that restricts to the identity on $\mathfrak{g}$ and sends $\mathfrak{g} \otimes t \C[t]$ to zero and $c$ to $\kappa-h^\vee$, we can define for every dominant weight $\lambda \in X^+$ the Verma module $\Delta_\kappa(\lambda) = U(\aff{g}) \otimes_{U(\aff{p})} L_\mathfrak{g}(\lambda)$ in $\mathcal{O}_\kappa$, where $L_\mathfrak{g}(\lambda)$ is the simple $\mathfrak g$-module of highest weight $\lambda$.
We write $L_\kappa(\lambda)$ for the unique simple quotient of $\Delta_\kappa(\lambda)$.
Then the $\aff{g}$-modules $L_\kappa(\lambda)$ with $\lambda \in X^+$ form a set of representatives for the isomorphism classes of simple $\aff{g}$-modules on $\mathcal{O}_\kappa$.

\begin{Remark} \label{rem:affinevsKacMoody}
    The affine Lie algebra $\aff{g}$ has a derivation $\partial$ with $\partial( x \otimes t^m ) = m \cdot x \otimes t^m$ for $x \in \mathfrak{g}$ and $m \in \Z$ and $\partial(c) = 0$.
    By adjoining this derivation to $\aff{g}$, one obtains a Kac--Moody algebra $\kacmoody{g} = \aff{g} \oplus \C \partial$, see \cite[Theorem 13.1.3]{KumarKacMoody}.
    A significant part of the literature on representations of affine Lie algebras is concerned with representations of $\kacmoody{g}$ instead of $\aff{g}$, but it is usually straightforward to translate between the two settings.
    Specifically, let $\kacmoody{p} = \aff{p} \oplus \C \partial$ and consider the category $\mathbb{O}_\kappa$ of $\kacmoody{g}$-modules of level $\kappa$ that are locally finite with respect to $\kacmoody{p}$ and semisimple with respect to $\mathfrak{g} \oplus \C c \oplus \C \partial$.
    The Casimir operator $\Omega$ from equation (12.8.3) in \cite{KacInfDimLiealgebras} acts as a locally finite endomorphism on every $\aff{g}$-module in $\mathbb{O}_\kappa$, and we write $\mathbb{O}_\kappa(\Omega \simeq 0)$ for the full subcategory of $\mathbb{O}_\kappa$ whose objects are the $\aff{g}$-modules in $\mathbb{O}_\kappa$ on which $\Omega$ acts locally nilpotently.
    Then, according to Proposition 8.1 in \cite{SoergelCharakterformeln} (attributed to Polo), forgetting the action of $\partial$ gives rise to an equivalence
    \begin{equation} \label{eq:categoryOequivalencecasimir}
        \mathbb{O}_\kappa(\Omega \simeq 0) \simeq \mathcal{O}_\kappa
    \end{equation}
    for all $\kappa \neq 0$; see also \cite[Section 3]{ParshallScottSemisimpleSeries} for more details.
    In particular, $\mathcal{O}_\kappa$ is equivalent to a union of blocks of $\mathbb{O}_\kappa$, cf.\ Proposition 5.5 in \cite{ParshallScottSemisimpleSeries}.
    In view of the equivalence \eqref{eq:categoryOequivalencecasimir}, we will occasionally cite results about representations of $\kacmoody{g}$ (e.g.\ from \cite{SoergelCharakterformeln} or \cite[Section 6.3]{BrundanStroppelSemiInfiniteHighestWeight}) to justify statements about representations of $\aff{g}$ without further explanation. 
\end{Remark}

From now on, we fix a non-zero rational level $0 \neq \kappa \in \Q$.
(The category $\mathcal{O}_\kappa$ is semisimple at non-rational levels and poorly understood at the ``critical'' level zero.)
We consider the full subcategories of $\mathcal{O}_\kappa$ given by the following finiteness conditions.
\begin{flalign*}
    \mathcal{O}_\kappa^\mathrm{fl} \; & = \; \text{finite length } \aff{g}\text{-modules in } \mathcal{O}_\kappa \\
    \mathcal{O}_\kappa^\mathrm{fg} \; & = \; \text{finitely generated } \aff{g}\text{-modules in } \mathcal{O}_\kappa \\
    \mathcal{O}_\kappa^\mathrm{fm} \; & = \; \aff{g}\text{-modules in } \mathcal{O}_\kappa \text{ with finite composition multiplicities}
\end{flalign*}
Note that the category $\mathcal{O}_\kappa^\mathrm{fl}$ is denoted by $\mathscr{O}_\kappa$ in \cite[Definition 2.15]{KL12}.
We have the following result, analogous to \cite[Theorem 6.4]{BrundanStroppelSemiInfiniteHighestWeight}.

\begin{Proposition}
    Let $0 \neq \kappa \in \Q$.
    \begin{enumerate}
        \item If $\kappa < 0$, then $\mathcal{O}_\kappa^\mathrm{fl}$ is a lower finite highest weight category with weight poset $(X^+,\leq)$.
        \item If $\kappa > 0$, then $\mathcal{O}_\kappa^\mathrm{fm}$ is an upper finite highest weight category with weight poset $(X^+,\leq^\mathrm{op})$.
    \end{enumerate}
    In both cases, the standard objects are the Verma modules $\Delta_\kappa(\lambda)$ for $\lambda \in X^+$.
\end{Proposition}
\begin{proof}
    First suppose that $\kappa > 0$.
    Then for $\lambda,\mu \in X^+$, we have $[ \Delta_\kappa(\lambda) : L_\kappa(\mu) ] = 0$ unless $\lambda \leq \mu$ by the character formula in \cite[Theorem 1.4]{KashiwaraTanisakiKazhdanLusztigPositiveRational}.
    As the poset $(X^+,\leq^\mathrm{op})$ is upper finite, Theorem 3.2 and Remark 4.3 in \cite{SoergelCharakterformeln} imply that $\mathcal{O}_\kappa^\mathrm{fm}$ satisfies the condition $(P\Delta)$.
    Now it only remains to show that $\mathcal{O}_\kappa^\mathrm{fm}$ is Schurian, and this can be done exactly as in \cite[Proof of Theorem 6.4]{BrundanStroppelSemiInfiniteHighestWeight}.

    Now suppose that $\kappa < 0$.
    Then the category $\mathcal{O}_\kappa^\mathrm{fl}$ is locally finite by \cite[Proposition 2.29]{KL12}, and we have $[ \Delta_\kappa(\lambda) : L_\kappa(\mu) ] = 0$ unless $\mu \leq \lambda$ for $\lambda,\mu \in X^+$ by \cite[Theorem 6.4]{TanisakiCharacterFormulas}.
    Again by Theorem 3.2 and Remark 4.3 in \cite{SoergelCharakterformeln}, every finite truncation of $\mathcal{O}_\kappa^\mathrm{fl}$ satisfies the condition $(P\Delta)$, and so $\mathcal{O}_\kappa^\mathrm{fl}$ is a lower finite highest weight category.
\end{proof}

\begin{Remark}
    If $\kappa > 0$, then for the upper finite highest weight category $\cC = \mathcal{O}_\kappa^\mathrm{fm}$, we have
    \[ \cC^\mathrm{cc} = \mathrm{Ind}( \cC_\mathrm{c} ) \simeq \mathcal{O}_\kappa , \]
    cf.\ the proof of Theorem 6.3 in \cite{BrundanStroppelSemiInfiniteHighestWeight}.
    If $\kappa < 0$, then every $\aff{g}$-module in $\mathcal{O}_\kappa$ is the union of its finite length submodules and all $\aff{g}$-modules in $\mathcal{O}_\kappa^\mathrm{fl}$ are compact in $\mathcal{O}_\kappa$.
    Consequently, for the lower finite highest weight category $\cC = \mathcal{O}_\kappa^\mathrm{fl}$, we have
    \[ \cC^\mathrm{cc} = \mathrm{Ind}(\cC) \simeq \mathcal{O}_\kappa \]
    by \cite[Corollary 6.3.5]{KashiwaraSchapiraCategoriesSheaves}.
\end{Remark}

Now for $\kappa \in \Q_{<0}$ (or $\kappa \in \Q_{>0}$) and $\lambda \in X^+$, let us write $\nabla_\kappa(\lambda)$ for the costandard object of highest weight $\lambda$ in the lower finite (or upper finite) highest weight category $\mathcal{O}_\kappa^\mathrm{fl}$ (or $\mathcal{O}_\kappa^\mathrm{fm}$, respectively), so $\nabla_\kappa(\lambda)$ is a dual Verma module.
We further write $T_\kappa(\lambda)$ for the tilting module of highest weight $\lambda$ and $P_\kappa(\lambda)$ for the projective cover of $L_\kappa(\lambda)$, if it exists.

\begin{Remark} \label{rem:affineBGGduality}
    Let $\kappa \in \Q_{<0}$.
    Then by Propositions 2.24--2.26 in \cite{KL12}, there is an equivalence
    \[ D \colon \mathcal{O}_\kappa \longrightarrow (\mathcal{O}_\kappa)^\mathrm{op} \]
    such that $D\big( L_\kappa(\lambda) \big) \cong L_\kappa( -w_0 \lambda )$ for all $\lambda \in X^+$, where $w_0$ denotes the longest element of the Weyl group of $\mathfrak{g}$.
    This implies that we also have
    \[ D\big( \Delta_\kappa(\lambda) \big) \cong \nabla_\kappa( -w_0 \lambda ) , \hspace{2cm} D\big( \nabla_\kappa(\lambda) \big) \cong \Delta_\kappa( -w_0 \lambda ) , \hspace{2cm} D\big( T_\kappa(\lambda) \big) \cong T_\kappa( -w_0 \lambda ) \]
    for all $\lambda \in X^+$.
\end{Remark}

Now we are ready to discuss Ringel duality for the categories $\mathcal{O}_\kappa$.

\begin{Theorem}[Arkhipov, Soergel] \label{thm:ArkhipovSoergel}
    For $\kappa \in \Q_{<0}$, the upper finite highest weight category $\mathcal{O}_{-\kappa}^\mathrm{fm}$ is the Ringel dual of the lower finite highest weight category $\mathcal{O}_\kappa^\mathrm{fl}$.
\end{Theorem}
\begin{proof}
    By Theorem 6.6 and the proof of Corollary 7.6 in \cite{SoergelCharakterformeln}, there is an equivalence
    \[ t \colon \mathcal{O}_{-\kappa}^{\mathrm{fm},\Delta} \xrightarrow{~\sim~} ( \mathcal{O}_\kappa^{\mathrm{fl},\Delta} )^\mathrm{op} \]
    that preserves exact sequences and satisfies
    \[ t( \Delta_{-\kappa}(\lambda) ) \cong \Delta_\kappa(-w_0\lambda) , \hspace{2cm} t( P_{-\kappa}(\lambda) ) \cong T_\kappa(-w_0\lambda) \]
    for all $\lambda \in X^+$.
    By composing with the equivalence $D \colon \mathcal{O}_\kappa^\mathrm{fl} \to (\mathcal{O}_\kappa^\mathrm{fl})^\mathrm{op}$ from \Cref{rem:affineBGGduality}, we obtain an equivalence $\Proj^\mathrm{fg}( \mathcal{O}_{-\kappa}^\mathrm{fm} ) \simeq \Tilt(\mathcal{O}_\kappa^\mathrm{fl})$ that sends $P_{-\kappa}(\lambda)$ to $T_\kappa(\lambda)$ for all $\lambda \in X^+$, and \Cref{rem:tiltingequivalentprojectiveimpliesRingeldual} implies that $\mathcal{O}_{-\kappa}^\mathrm{fm}$ is the Ringel dual of $\mathcal{O}_\kappa^\mathrm{fl}$, as required.
\end{proof}

Next we explain how for $\kappa \in \mathbb{Q}_{<0}$, the Ringel duality between $\mathcal{O}_\kappa$ and $\mathcal{O}_{-\kappa}$ can be enhanced to a monoidal Ringel duality in the sense of \Cref{thm:monoidalRingelduality_lowertoupper}.

First note that by the results of \cite{KL12,KL34}, there is a braided monoidal structure on $\mathcal{O}_\kappa^\mathrm{fl}$, and the latter canonically extends to a braided monoidal structure on $\mathcal{O}_\kappa \simeq \Ind( \mathcal{O}_\kappa^\mathrm{fl} )$.
Furthermore, the category $\mathcal{O}_\kappa^\mathrm{fl}$ is rigid provided that there is a non-zero weight $\lambda \in X^+$ such that $\Delta_\kappa(\lambda)$ is simple and rigid, see \cite[Theorem 32.1]{KL34}.
We say that $\kappa \in \Q_{<0}$ is \emph{KL-good} if there is a non-zero weight $\lambda \in X^+$ such that $\Delta_\kappa(\lambda)$ is simple and rigid.
Writing $\kappa = -p/q$ for non-negative coprime integers $p$ and $q$, Lemma 31.5 in \cite{KL34} implies that there is a constant $p_0$ (depending on the type of $\mathfrak{g}$) such that $\kappa$ is KL-good for all $p \geq p_0$.
If $\mathfrak{g}$ is of type $\mathrm{A}_n$ or $\mathrm{D}_{2n}$, then all $\kappa \in \Q_{<0}$ are KL-good, and if $\mathfrak{g}$ is of type $\mathrm{D}_{2n+1}$, then $\kappa = - p/q$ is KL-good for $p \neq 2$, by Lemmas 31.6 and 31.7 in \cite{KL34}.

Now let $D$ be the square of the ratio of the length of long roots to the length of short roots in the root system $\Phi$ of $\mathfrak{g}$,%
\footnote{In other words, we set $D=1$ if $\mathfrak{g}$ is of type $\mathrm{ADE}$, $D = 2$ if $\mathfrak{g}$ is of type $\mathrm{BCF}$ and $D = 3$ if $\mathfrak{g}$ is of type $\mathrm{G}_2$.}
and let $\zeta = \exp( \frac{\pi \mathrm{i}}{D\kappa} ) \in \C$.
We write $U_\zeta = U_\zeta(\mathfrak{g})$ for the quantum group (with divided powers and specialized at $\zeta$) corresponding to $\mathfrak{g}$ and consider the category $\Rep(U_\zeta)$ of finite-dimensional $U_\zeta$-modules with a weight space decomposition, as in \cite[Section 7]{TanisakiCharacterFormulas}.
For a dominant weight $\lambda \in X^+$, we write $\Delta_\zeta(\lambda)$ for the Weyl module over $U_\zeta$ and $L_\zeta(\lambda)$ for its unique simple quotient, so that the $U_\zeta$-modules $L_\zeta(\lambda)$ for $\lambda \in X^+$ form a set of representatives for the isomorphism classes of simple objects in $\Rep(U_\zeta)$.
The category $\Rep(U_\zeta)$ has a canonical braided monoidal structure, arising from the Hopf algebra structure and $R$-matrix of $U_\zeta$.
The following theorem is proven in \cite[Section 38]{KL34}; see also \cite[Section 8.4]{LusztigMonodromicSystems} and \cite{LusztigMonodromicSystemsErrata} for the non-simply laced case.

\begin{Theorem}[Kazhdan--Lusztig] 
    There is a braided monoidal functor
    \[ F_\kappa \colon \mathcal{O}_\kappa^\mathrm{fl} \longrightarrow \Rep(U_\zeta) . \]
    If $\kappa$ is KL-good, then $F_\kappa$ is an equivalence.
\end{Theorem}

From now on, we assume that $\kappa \in \Q_{<0}$ is KL-good.
The subcategory $\mathcal{O}_{\kappa}^{\mathrm{fl},\Delta}$ is closed under tensor products in $\mathcal{O}_\kappa^\mathrm{fl}$ by \cite[Proposition 28.1]{KL34}, and using the rigidity of $\mathcal{O}_\kappa^\mathrm{fl}$, it follows that $\Tilt(\mathcal{O}_\kappa)$ is also closed under tensor products.
(Alternatively, we can deduce this from the corresponding fact about tilting modules in $\Rep(U_\zeta)$ \cite{ParadowskiFiltration}.)
Thus, our monoidal Ringel duality from \Cref{thm:monoidalRingelduality_lowertoupper} gives the following result.

\begin{Theorem} \label{thm:positivelevelmonoidalstructure}
    Let $\kappa \in \Q_{>0}$ such that $-\kappa$ is KL-good and let $\zeta = \exp(-\frac{\pi \mathrm{i}}{D \kappa})$.
    \begin{enumerate}
        \item There is a braided monoidal structure on $\mathcal{O}_{\kappa}$ and a braided monoidal functor
    \[ G_\kappa \colon \mathcal{O}_{\kappa} \longrightarrow \Ind\big(  \Rep(U_\zeta) \big) \]
    such that $G_\kappa\big( \Delta_\kappa(\lambda) \big) \cong \nabla_\zeta(\lambda)$ for all $\lambda \in X^+$.
    \item The subcategory $\mathcal{O}_\kappa^\mathrm{fg}$ is closed under tensor products and the functor $G_\kappa$ restricts to a functor
    \[ \mathcal{O}_\kappa^\mathrm{fg} \longrightarrow \Rep(U_\zeta) . \]
    \item The finitely generated projective objects in $\mathcal{O}_\kappa$ are precisely the rigid objects in $\mathcal{O}_\kappa^\mathrm{fg}$.
    \item The functor $G_\kappa$ is exact and essentially surjective.
    \item The functor $G_\kappa$ admits a right adjoint functor $G'_\kappa \colon \Ind( \Rep(U_\zeta) ) \to \mathcal{O}_\kappa$ that is fully faithful and exhibits $\mathcal{O}_{-\kappa} \simeq \Ind( \Rep(U_\zeta) )$ as a reflective subcategory of $\mathcal{O}_\kappa$.
    \end{enumerate}
\end{Theorem}
\begin{proof}
    Let us write $\cC = \mathcal{O}_{-\kappa}^\mathrm{fl}$ and recall from \Cref{thm:ArkhipovSoergel} that $\mathcal{O}_\kappa^\mathrm{fm} \simeq \cC^\vee$ is the Ringel dual of $\cC$.
    By our assumption on $\kappa$, the category $\mathcal{O}_{-\kappa}^\mathrm{fl}$ is rigid and there is a monoidal equivalence $\mathcal{O}_{-\kappa}^\mathrm{fl} \simeq \Rep(U_\zeta)$, and so the claims (1)--(3) are immediate from \Cref{thm:monoidalRingelduality_lowertoupper} if we define
    \[ G_\kappa = L_\cC \colon \quad \mathcal{O}_\kappa \simeq \cC^{\vee,\mathrm{cc}} \longrightarrow \cC^\mathrm{cc} \simeq \mathcal{O}_{-\kappa} \simeq \Ind( \Rep(U_\zeta) ) \]
    to be the opposite Ringel duality functor (composed with the Kazhdan--Lusztig equivalence $F_\kappa$).
    In order to prove the remaining claims, first note that by \cite[Theorem 10.12]{NegronRevisitingSteinberg}, the category $\Rep(U_\zeta)$ has enough injective and projective objects and that injectives and projectives coincide.
    Since all projectives admit a standard filtration and all injectives admit a costandard filtration, all projective and injective objects of $\mathcal{O}_{-\kappa}^\mathrm{fl} \simeq \Rep(U_\zeta)$ are also tilting modules, and \Cref{prop:RingelDualReflective} implies that the Ringel duality functor
    \[ G'_\kappa = R_\cC \colon \quad \Ind( \Rep(U_\zeta) ) \simeq \mathcal{O}_{-\kappa} \simeq \cC^\mathrm{cc} \longrightarrow \cC^{\vee,\mathrm{cc}} \simeq \mathcal{O}_\kappa \]
    is fully faithful and exhibits $\mathcal{O}_{-\kappa} \simeq \Ind( \Rep(U_\zeta) )$ as a reflective subcategory of $\mathcal{O}_\kappa$.
    As $G_\kappa'$ is right adjoint to $G_\kappa$ (cf.\ Subsection~\ref{subsec:tiltingobjectsRingelduality}), this establishes claim (5).
    By general properties of reflective subcategories (see \cite[Section IV.3]{MacLaneCategories}), we further obtain that the counit of the adjunction $G_\kappa \dashv G'_\kappa$ is an isomorphism
    \[ G_\kappa \circ G'_\kappa \cong \id_{\mathcal{O}_{-\kappa}} , \]
    and this implies that $G_\kappa$ is  essentially surjective.
    Finally, as all injective objects in $\cC$ are tilting, all tilting objects in $\cC^\vee$ are projective by \Cref{thm:Ringelduality}.
    By Remarks~\ref{rem:Chevalleydualitylowertoupper} and \ref{rem:affineBGGduality}, the category $\cC^\vee$ admits a duality $D \colon \cC^\vee \xrightarrow{~\sim~} \cC^{\vee,\mathrm{op}}$ that sends tilting modules to tilting modules and projective modules to injective modules,%
    \footnote{Alternatively, one can define the usual duality on the BGG category $\mathbb{O}_\kappa$ for the Kac--Moody algebra $\widetilde{\mathfrak{g}}$ as in \cite[Section 2.1]{KumarKacMoody} and use \Cref{rem:affinevsKacMoody}.}
    hence all tilting modules in $\cC^\vee$ are also injective.
    Now for every short exact sequence $0 \to X \to Y \to Z \to 0$ in $\mathcal{O}_\kappa$ and for every injective module $I$ in $\cC$, we have that $G'_\kappa(I)$ is tilting (by \Cref{thm:Ringelduality}) and hence injective in $\cC^\vee$, and using the adjunction $G_\kappa \dashv G'_\kappa$, we obtain a commutative diagram
    \[ \begin{tikzcd}
        0 \ar[r] & \Hom_{\cC^\vee}\big( Z , G'_\kappa I \big) \ar[r] & \Hom_{\cC^\vee}\big( Y , G'_\kappa I \big) \ar[r] & \Hom_{\cC^\vee}\big( X , G'_\kappa I \big) \ar[r] & 0 \\
        0 \ar[r] & \Hom_\cC\big( G_\kappa Z , I \big) \ar[r] \arrow[u, phantom, sloped, "\cong"] & \Hom_\cC\big( G_\kappa Y , I \big) \ar[r] \arrow[u, phantom, sloped, "\cong"] & \Hom_\cC\big( G_\kappa X , I \big) \ar[r] \arrow[u, phantom, sloped, "\cong"] & 0
    \end{tikzcd} \]
    with exact rows because $G'_\kappa I$ is injective.
    As $\cC$ has enough injectives, this implies that the sequence
    \[ 0 \to G_\kappa X \to G_\kappa Y \to G_\kappa Z \to 0 \]
    is exact, and so $G_\kappa$ is exact, as claimed.%
    \footnote{Our proof of the fact that $G_\kappa$ is exact was adapted from the MathOverflow post \url{https://math.stackexchange.com/questions/1639186/exactness-of-a-right-adjoint-functor}.
    An alternative way to prove this is to note that by the definition of the opposite Ringel duality functor in \cite[Equation (4.15)]{BrundanStroppelSemiInfiniteHighestWeight}, this functor is exact provided that all tilting modules in $\cC'$ are injective.}
\end{proof}

\begin{Remark}
    \Cref{thm:positivelevelmonoidalstructure} recovers and generalizes an analogous result that was obtained by McRae--Yang for $\mathfrak{g}=\mathfrak{sl}_2(\C)$.
    More precisely, for $\kappa \in \Q_{>0}$, it is shown in \cite[Theorem 1.1]{McRaeYang} that the category $\mathcal{O}_\kappa^\mathrm{fl}$ admits a braided monoidal structure, and if $\kappa = \frac{p}{q}$ with $p \in \Z_{\geq 2}$ and $q \in \Z_{\geq 1}$, then there is an exact and essentially surjective monoidal functor
    \[ H_\kappa \colon \mathcal{O}_\kappa^\mathrm{fl} \longrightarrow \Rep(U_\zeta) . \]
    by \cite[Theorem 1.6]{McRaeYang} (there denoted by $\mathcal{F}$).
    As the Verma modules in $\mathcal{O}_\kappa$ have finite length in the special case $\mathfrak{g} = \mathfrak{sl}_2(\C)$ (see for instance \cite[Theorem 2.2]{McRaeYang}), the categories $\mathcal{O}_\kappa^\mathrm{fm}$, $\mathcal{O}_\kappa^\mathrm{fg}$ and $\mathcal{O}_\kappa^\mathrm{fl}$ all coincide in this case, so that our \Cref{thm:positivelevelmonoidalstructure} directly generalized the results of McRae--Yang.

    McRae--Yang also explicitly describe the kernel of their functor $H_\kappa$ in \cite[Theorem 1.5]{McRaeYang}.
    In our setting (i.e.\ for arbitrary $\mathfrak{g}$), it is also possible to describe the kernel of $G_\kappa$ using the techniques from the proof of \Cref{thm:positivelevelmonoidalstructure}:
    For all $\lambda \in X^+$, let $I_\zeta(\lambda)$ be the injective hull of $L_\zeta(\lambda)$ and let $\hat \lambda \in X^+$ be the unique dominant weight with $I_\zeta(\lambda) \cong T_\zeta(\hat \lambda)$.
    Then by \Cref{thm:Ringelduality}, we have
    \[ T_\kappa(\lambda) \cong G_\kappa' I_\zeta(\lambda) \cong G_\kappa' T_\zeta(\hat \lambda) \cong P_\kappa(\hat \lambda) , \]
    and using the equivalence $\mathcal{O}_\kappa \simeq \mathcal{O}_\kappa^\mathrm{op}$ from the proof of \Cref{thm:positivelevelmonoidalstructure}, we further get
    \[ I_\kappa(\hat \lambda) \simeq P_\kappa(\hat \lambda) \simeq G_\kappa' I_\zeta(\lambda) . \]
    Now for an object $M$ of $\mathcal{O}_\kappa$, we have $G_\kappa M = 0$ if and only if
    \[ 0 = \Hom_{U_\zeta}\big( G_\kappa M , I_\zeta(\lambda) \big) \cong \Hom_{\mathcal{O}_\kappa}\big( M , G_\kappa' I_\zeta(\lambda) \big) \cong \Hom_{\mathcal{O}_\kappa}\big( M , I_\kappa(\hat \lambda) \big) \]
    for all $\lambda \in X^+$.
    In other words, $M$ belongs to the kernel of $G_\kappa$ if and only if $M$ has no simple subquotients $L_\kappa(\mu)$ with $\mu \in \hat X^+ \coloneqq \{ \hat \lambda \mid \lambda \in X^+ \}$, or equivalently, if all simple subquotients of $M$ are of the form $L_\kappa(\mu)$ with $\mu \in X^+ \setminus \hat X^+$.
    If $\ell = \mathrm{ord}(\zeta)$ is odd (and coprime to $3$ if $\mathfrak{g}$ is of type $\mathrm{G}_2$), then the map $\lambda \mapsto \hat \lambda$ is explicitly described in \cite[Section 5.5]{AndersenpfiltrationSteinberg}, and it is shown that $\hat X^+ = (\ell-1) \cdot \rho + X^+$, where $\rho$ denotes the half sum of positive roots in the root system of $\mathfrak{g}$.%
    \footnote{For an arbitrary level $\kappa \in \Q_{>0}$ such that $-\kappa$ is KL-good, one can give an analogous description of $\hat X^+$ using the arguments in \cite[Section 5.5]{AndersenpfiltrationSteinberg} and the treatment of Steinberg modules in \cite{NegronRevisitingSteinberg}.
    We leave the details to the interested reader.}
    We note that this matches the description of the kernel of $H_\kappa$ in \cite{McRaeYang}.
\end{Remark}

\appendix

\section{Reduced \texorpdfstring{$n$}{n}-ary relations}
\label{sec:relations}

Fix a regular category $\cR$. We will describe some aspects of a kind of operadic structure formed by $n$-ary relations, for $n\ge0$. This is used to derive results on what we call \emph{reduced} $n$-ary relations, which is used and referenced in the body of the document, specifically for proving \Cref{lem:U-D-basics}(c) and \Cref{lem:tri-decomposition}.

Recall that for us, a regular category is finitely complete, and we use the terms \emph{injective} and \emph{surjective} when referring to monomorphisms and extremal epimorphisms in $\cR$. For any morphism $f$ in $\cR$, we denote its source and target object by $S(f)$ and $T(f)$, respectively.

For any $n\ge0$, an \emph{$n$-ary span} is a tuple $f = (f_1,\dots,f_n)$ of morphisms in $\cR$ with a common source object $S(f):=S(f_1)=\dots=S(f_n)$. An isomorphism of $n$-ary spans $f$ and $f'$ is an isomorphism $\phi\:S(f)\to S(f')$  such that $f_i=f'_i \phi$ for all $i$. An $n$-ary span as above is called \emph{jointly injective} if the induced morphism
$$
S(f)\to T(f_1)\times\dots\times T(f_n)
$$
is injective. An \emph{$n$-ary relation} is an isomorphism class of jointly injective $n$-ary spans. We will usually describe $n$-ary relations or isomorphism classes of $n$-ary spans by $n$-ary spans. 

\begin{Definition} Let $f$ be an $n$-ary span, fix some $1\le j\le n$. Let $s$ be a $2$-ary span, i.e., an ordinary span. We define the \emph{$j$-composition}
$$
f *_j s
$$
as the isomorphism class of the $n$-ary span $(f'_i)_i$, where $f'_i:=f_i$ for all $i\neq j$, and $f'_j:=s_2 s_1^*(f_j)$, where $s_1^*(f_j)$ is a the pullback of $f_j$ along $s_1$.
\end{Definition}

For any $n,j,s$ as in the definition, the operation $-*_j s$ extends to an operation on isomorphism classes of $n$-ary spans.

\begin{Lemma} For spans $s,s'$, an $n$-ary span $f$, and any $1\le j\neq k\le n$,
$$
(f *_j s) *_j s' = f *_j (s*_2 s') ,
\quad
(f *_j s) *_k s' = (f *_k s') *_j s .
$$
\end{Lemma}

\begin{proof} The first identity follows from associativity of pullbacks. The second identity follows immediately from the definitions.
\end{proof}

\begin{Lemma} \label{lem:joint-injective} Consider a diagram
$$
\begin{tikzcd}[column sep=1cm, row sep=0.3cm,remember picture]
 && \ti R 
    \ar[ld,"\ti s_1",swap]
    \ar[rd,"\ti f_c"] \\
 &  R 
    \ar[ld,"{(f_a,f_b)}",swap] 
    \ar[rd,"f_c"] 
 && \cdot 
    \ar[ld,"s_1"] 
    \ar[rd,"s_2",hookrightarrow] \\
 A \times B && C && C'
\end{tikzcd}
$$
in $\cR$, where the square is a pullback and $s_2$ is injective. Assume $R\xrightarrow{(f_a,f_c)} A\times C$ is injective. Then $\ti R\xrightarrow{(f_a \tilde s_1,s_2 \tilde f_c)} A\times C'$ is injective.    
\end{Lemma}

\begin{proof} Assume $h_1,h_2$ are parallel morphisms with target $\ti R$, such that 
$$
(f_a \ti s_1, s_2 \ti f_c) h_1
=(f_a \ti s_1, s_2 \ti f_c) h_2 .
$$
As $s_2$ is injective, it can be canceled and we get 
$$
(f_a \ti s_1, \ti f_c)h_1
=(f_a \ti s_1, \ti f_c)h_2 .
$$
Using a suitable postcomposition, this yields
$$
(f_a \ti s_1, s_1 \ti f_c, \ti f_c)h_1
= (f_a \ti s_1, s_1 \ti f_c, \ti f_c)h_2 ,
$$
but $s_1 \ti f_c=f_c \ti s_1$, hence we get
$$
(f_a \ti s_1, f_c \ti s_1, \ti f_c)h_1
=(f_a \ti s_1, f_c \ti s_1, \ti f_c)h_2 .
$$
Now $(f_a,f_c)$ is injective by our assumptions, so we can cancel it to obtain
$$
(\ti s_1, \ti f_c)h_1
=(\ti s_1, \ti f_c)h_2 .
$$
But this means $h_1=h_2$, as the morphisms agree on the components of the fiber product $\ti R$.
\end{proof}

To simplify arguments, we will sometimes view an $n$-ary span $f$ as the $2$-ary span $\sigma_j(f)$ given by $((f_1,\dots,\widehat{f_j},\dots, f_n),f_j)$, where $f_j$ is omitted as indicated by the hat decoration in the first component, for any $n\ge2$ and $1\le j\le n$. For a $1$-ary span $f=(f_1)$ we define $\sigma_1(f)$ as the $2$-ary span $(\id,f)$. It can be seen that these definitions of $\sigma_j$ extend to isomorphism classes of spans and to relations.

\begin{Lemma} Assume $s$ is a span and $f$ is an $n$-ary span such that $f$ is jointly injective and $s_2$ is injective. Then $f*_j s$ is jointly injective and $-*_js$ induces an operation on $n$-ary relations, for all $1\le j\le n$.
\end{Lemma}

\begin{proof} Using $\sigma_j$, it suffices to consider the case $n=2$ and $j=2$.

Now the first assertion follows immediately from \Cref{lem:joint-injective}.

For the second assertion, it suffices to note that the operation $-*_j s$ not only preserves jointly injective spans, but is also compatible with isomorphisms between the sources of two such spans.
\end{proof}

\begin{Definition} An $n$-ary span $(f_i)_i$ is called \emph{$j$-reduced}, for some $1\le j\le n$, if $f_j$ is surjective and  the induced morphism
$$ (f_1,\dots,\widehat{f_j},\dots,f_n)\:
T(f)\to S(f_1)\times\dots\times\widehat{S(f_j)}\times\dots\times S(f_n) ,
$$
is injective, where $f_j$ and $S(f_j)$ are omitted.
\end{Definition}

Again, the definition extends to isomorphism classes of spans. Moreover, an $n$-ary relation is called $j$-reduced if any or, equivalently, every $n$-ary span representing it is $j$-reduced.

Note that $f$ is $j$-reduced if and only if $\sigma_j(f)$ is $2$-reduced, and a $2$-span $f$ is $2$-reduced if and only if $f_1$ is injective and $f_2$ is surjective.

\begin{Definition} Let $\cS^-$ be the set of spans $s$ such that $s_1$ is surjective and $s_2$ is injective.
\end{Definition}

From now on, assume $\cR$ is an \emph{exact Mal'cev} category, i.e., $\cR$ is regular, every congruence relation in $\cR$ is a kernel pair, and every reflexive relation in $\cR$ is an equivalence relation. The main relevant feature of such categories for us is the following fact:

\begin{Lemma}[{\cite[{A.1.2~Proposition}]{KnopTensorEnvelopes}, \cite[Theorem~5.7]{CKP}}] \label{lem:exact-malcev} Any pair of surjective morphisms with a common source in $\cR$ has a pushout. Moreover, a diagram of the shape
$$
\begin{tikzcd}
\cdot \ar[r,two heads] \ar[d,two heads] 
    & \cdot \ar[d,two heads] \\ 
\cdot \ar[r,two heads] 
    & \cdot \\ 
\end{tikzcd}
$$ in $\cR$ consisting of four surjective morphisms is a pullback diagram if and only if it is a pushout diagram such that the unique span contained in the diagram is jointly injective. 
\end{Lemma}

\begin{Lemma} \label{lem:reduction-exists} For any $n$-ary relation $f$ and any $1\le j\le n$, there is a span $s$ in $\cS^-$ and an $n$-ary relation $g$ that is $j$-reduced such that $f=g*_j s$.    
\end{Lemma}

\begin{proof} Using $\sigma_j$, it suffices again to consider the case $n=2$, $j=2$, which is explained in slightly different language at the beginning of \cite[Section 5]{KnopTensorEnvelopes}, as follows:

Let $\coim(f_1)$, $\im(f_1)$, $\coim(f_2)$, $\im(f_2)$ be morphisms forming image factorizations of $f_1$ and $f_2$, respectively. 
$$
\begin{tikzcd}[column sep=1cm, row sep=0.3cm,remember picture]
 && \cdot 
    \ar[ld,two heads,"\coim(f_1)",swap]
    \ar[rd,two heads,"\coim(f_2)"] \\
 &  \cdot 
    \ar[ld,hookrightarrow,"\im(f_1)",swap] 
    \ar[rd,two heads,"g_1"] 
 && \cdot 
    \ar[ld,two heads,"g_2"] 
    \ar[rd,hookrightarrow,"\im(f_2)"] \\
 \cdot && \cdot && \cdot
\end{tikzcd} ,
$$

As the morphisms $\coim(f_1)$ and $\coim(f_2)$ are surjective, their pushout exists and is at the same time a pullback by \Cref{lem:exact-malcev}. Let $g_1$ and $g_2$ be the two other maps in the resulting bicartesian square. As the pushouts of surjective morphisms, they are surjective. Then $f$ factors as $(\im(f_1),g_1)*_2(g_2,\im(f_2))$, in the desired way.   
\end{proof}

\begin{Lemma} \label{lem:reduction-preserved} Let $f$ be an $n$-ary relation, let $s$ be a span in $\cS^-$. Then $f\circ_j s$ is $k$-reduced if $f$ is $k$-reduced, for all $1\le j\neq k\le n$.
\end{Lemma}

\begin{proof} We may assume $n\ge2$. Set $X_\bullet:=\prod_{j\neq i\neq k} X_i$ and let $f_\bullet$ be the $(n-2)$-ary span obtained by removing $f_j$ and $f_k$ from $f$. We can represent the situation using the diagram
$$
\begin{tikzcd}[column sep=1cm, row sep=0.3cm,remember picture]
 && \cdot 
    \ar[ld,two heads,"\ti s_1",swap]
    \ar[rd,"\ti f_j"] \\
 &  \cdot 
    \ar[ld,"{(f_\bullet,f_k)}",swap] 
    \ar[rd,"f_j"] 
 && \cdot 
    \ar[ld,"s_1",two heads] 
    \ar[rd,"s_2",hookrightarrow] \\
 X_\bullet \times X_k && X_j && \cdot
\end{tikzcd} ,
$$
where the square is bicartesian. The span in the bottom left of the diagram represents $f$, and the span along the outer contour represents $f*_j s$. Note that $\ti s_1$ is surjective, as it is a pullback of the surjective morphism $s_1$.

Now assume $f$ is $k$-reduced, i.e., $f_k$ is surjective and $(f_\bullet,f_j)$ is injective. Then $f_k \ti s_1$ is surjective and $(f_\bullet \tilde s_1,s_2 \tilde f_j)$ is injective by \Cref{lem:joint-injective}. This means $f*_j s$ is $k$-reduced.
\end{proof}

\begin{Lemma} \label{lem:reduction-unique} If $s,s'\in\cS^-$, $g,g'$ are $n$-ary relations that are $j$-reduced, for some $1\le j\le n$, and
$$
g*_j s=g'*_j s' .
$$
Then there exists an isomorphism $\alpha$ such that 
$$
g' = g *_j \sigma_1(\alpha) ,
\quad
s' = \sigma_1(\alpha^{-1}) *_2 s .
$$
\end{Lemma}

\begin{proof} Using $\sigma_j$, this can again be reduced to the case $n=2$ and $j=2$. Then the two factorizations of the same relation can be represented by the following diagram, up to the dashed arrows:
$$
\begin{tikzcd}[column sep=1cm, row sep=0.3cm,remember picture]
&& \cdot \ar[ld,two heads] \ar[rd,two heads] \ar[ddddd,bend left=90, looseness=3.5,"\phi"]
\\
& S(g) \ar[ld,hookrightarrow] \ar[rd,two heads] \ar[ddd,dashed,"\psi_1"] 
&& S(s) \ar[ddd,dashed,"\psi_2"] \ar[ld,two heads] \ar[rd,hookrightarrow]
\\
\cdot \ar[d,"\id"] && \cdot \ar[d,"\psi",dashed]  && \cdot \ar[d,"\id"] 
\\
\cdot && \cdot && \cdot
\\
& S(g') \ar[lu,hookrightarrow] \ar[ru,two heads] 
&& S(s') \ar[lu,two heads] \ar[ru,hookrightarrow] 
\\
&& \cdot \ar[lu,two heads] \ar[ru,two heads] 
\end{tikzcd} ,
$$
where the top and bottom squares are bicartesian and $\phi$ is an isomorphism forming a commutative diagram with (both halves of) the outer contour of the remaining diagram. Uniqueness of the image factorization in $\cR$ then implies the existence of $\psi_1$, $\psi_2$ as in the diagram such that the squares formed by $\psi_i$ and $\phi$, and by $\psi_i$ and an identity, for $i=1,2$, are all commutative diagrams. Uniqueness of the pushout then implies the existence of $\psi$ as in the diagram such that the squares formed by $\psi$ and $\psi_i$ are commutative diagrams, for $i=1,2$. Now $\alpha:=\psi$ is as desired.
\end{proof}

\begin{Lemma} \label{lem:pullback-product} Let $\cdot\xrightarrow{f_i}\cdot\xleftarrow{g_i}\cdot$ be a cospan with a pullback span $\cdot\xleftarrow{\tilde g_i}\cdot\xrightarrow{\tilde f_i}\cdot$, for $i=1,2$. Then $\cdot\xleftarrow{\tilde g_1\times\tilde g_2}\cdot\xrightarrow{\tilde f_1\times\tilde g_2}\cdot$ is a pullback span for the cospan $\cdot\xrightarrow{f_1\times f_2}\cdot\xleftarrow{g_1\times f_2}\cdot$.    
\end{Lemma}

\begin{proof} This is an elementary computation: assume $h=(h_1,h_2)$, $h'=(h'_1,h'_2)$ are parallel morphisms with $T(h_i)=S(f_i)$, $T(h'_i)=S(g_i)$, for $i=1,2$, such that
$$
(f_1\times f_2) h = (g_1\times g_2) h' ,
\quad\text{ i.e., }\quad
f_i h_i = g_i h'_i ,
\quad i=1,2 .
$$
Then there is are unique morphisms $k_1$, $k_2$ such that
$$
h_i = \tilde g_i k_i \quad \text{ and }\quad
h'_i = \tilde f_i k_i , \quad i=1,2,
\quad\text{ i.e., }\quad
h = (\tilde g_1\times\tilde g_2) k \quad \text{ and }\quad
h' = (\tilde f_1\times\tilde f_2) k,
$$ 
with $k=(k_1,k_2)$.
\end{proof}

\begin{Corollary} \label{cor:reduced-3-relation} Let $f$ be a $3$-ary relation that is $1$-reduced. 

(a) Then there are spans $s,t\in\cS^-$ and a $3$-ary relation $g$ that is $i$-reduced for $i=1,2,3$, such that
$$
f = (g *_2 s) *_ 3 t = (g *_3 t) *_ 2 s
= (g_1, (g_2,g_3)) *_2 (s_1\times t_1, s_2\times t_2).
$$

(b) If $g',s',t'$ form another factorization of $f$ as in (a), then there exist isomorphisms $\alpha,\beta$ such that
$$
g' = (g *_2 \sigma_1(\alpha)) *_3 \sigma_1(\beta)
=  (g *_3 \sigma_1(\beta)) *_2 \sigma_1(\alpha)
= (g_1, (g_2,g_3)) *_2 \sigma_1(\alpha\times\beta)
,
$$$$
s' = \sigma_1(\alpha^{-1}) *_2 s ,
\quad
t' = \sigma_1(\beta^{-1}) *_2 t .
$$
\end{Corollary}

\begin{proof} (a) This follows from \Cref{lem:reduction-exists} applied twice, using \Cref{lem:reduction-preserved}, and using \Cref{lem:pullback-product} for the last asserted identity.

(b) This is \Cref{lem:reduction-unique} applied twice.
\end{proof}

\bibliographystyle{alpha}
\bibliography{highestweight}

\end{document}